\documentclass[a4paper,11pt]{article}
\usepackage{amsmath,a4wide}
\usepackage{amssymb}
\usepackage{amscd}
\usepackage{epsfig}
\usepackage{graphics}
\usepackage{multirow}
\usepackage{tikz}
\usetikzlibrary{arrows.meta,positioning,fit}
\usepackage{graphicx}
\usepackage{mathrsfs}
\usepackage{amsfonts}
\usepackage{subfigure} 
\usepackage{epstopdf}
\usepackage{booktabs}

\usepackage{ifpdf}
\usepackage[a4paper,top=1.5in, bottom=1.5in, left=0.75in, right=0.75in]{geometry}
\usepackage{changes}
\usepackage[makeroom]{cancel}
\usepackage{bbm}

\usepackage{amsthm}
\usepackage{color}
\definecolor{amaranth}{rgb}{0.9, 0.17, 0.31}	
\definecolor{myblue}{rgb}{0.2,0,0.9}
\definecolor{auburn}{rgb}{0.43, 0.21, 0.1}
\definecolor{bittersweet}{rgb}{1.0, 0.44, 0.37}
\definecolor{blue-violet}{rgb}{0.54, 0.17, 0.89}
\usepackage[authoryear]{natbib}

\usepackage{varioref}
\usepackage{xr-hyper}
\usepackage{hyperref}% \externaldocument[A-]{online-appendix}

\hypersetup{
	colorlinks,linkcolor=myblue,citecolor=blue-violet,filecolor=magenta,urlcolor=myblue, pdfauthor=author}

\newtheorem{pro}{Proposition}[section]
\newtheorem{pr}{Problem}
\newtheorem{dfn}{Definition}[section]
\newtheorem{lem}{Lemma}[section]
\newtheorem{thm}[lem]{Theorem}
\newtheorem{cor}[lem]{Corollary}
\newtheorem{rem}[lem]{Remark}

\newtheorem{as}{Assumption}[section]

\def\a{\alpha}
\def\b{\beta}
\def\d{\delta}

\def\e{\varepsilon}
\def\g{\gamma}
\def\l{\lambda}

\def\O{\Omega}
\def\pat{\partial}
\def\r{\rho}

\def\t{\theta}

\def\z{\zeta}
\def\k{\kappa}
\def\gam{\frac{\gamma_1-\gamma}{\gamma_1}}

\newcommand{ \loc }{ \mathrm{loc} }

\newcommand{\cC}{{\mathcal C}}
\newcommand{\cD}{{\mathcal D}}

\newcommand{\cF}{{\mathcal F}}

\newcommand{\cL}{{\mathcal L}}

\newcommand{\cN}{{\mathcal N}}

\newcommand{\cP}{{\mathcal P}}

\newcommand{\cX}{{\mathcal X}}

\numberwithin{equation}{section}

\begin{document}
%	\externaldocument[A-]{Supplemental-Material}

	\title{\LARGE {\bf  Finite-Horizon Portfolio Choice, Labor Supply, and Early Retirement under Borrowing Constraints}}
	
	\author{
	Gugyum Ha \footnote{E-mail: \href{mailto:ggha@sogang.ac.kr}{\tt ggha@sogang.ac.kr}\;Department of Mathematics, Sogang University, Seoul 04107, Republic of Korea.}
	\and
	Junkee Jeon \footnote{E-mail: \href{mailto:junkeejeon@khu.ac.kr}{\tt junkeejeon@khu.ac.kr}\;Department of  Applied Mathematics, Kyung Hee University, Yongin-Si 17104, Republic of Korea.}
	\and
    Jihoon Ok\footnote{E-mail: \href{mailto:jihoonok@sogang.ac.kr}{\tt jihoonok@sogang.ac.kr}\;Department of Mathematics, Sogang University, Seoul 04107, Republic of Korea.}
}
	
	\date{\today}
	
	\maketitle \pagestyle{plain} \pagenumbering{arabic}
	
                    	\begin{abstract}
We study a finite-horizon optimal consumption and portfolio problem with labor supply flexibility and an irreversible early retirement option under a borrowing constraint. The agent chooses consumption, risky investment, and leisure before retirement, while after retirement labor income disappears and leisure is fixed at its maximal level. Preferences are described by a Cobb--Douglas utility, and wealth must remain nonnegative.

{Using a dual martingale method, we transform the primal problem into a zero-sum stopper--singular-controller game. The associated dual value is characterized by a min--max parabolic variational inequality with obstacle and gradient constraints. We show that the maximal strong solution of the resulting variational inequality is the unique admissible strong solution whose gradient-constrained free boundary, namely the binding boundary, is monotone increasing in calendar time. A verification argument then identifies this strong solution with the value of the stopper--singular-controller game, and duality recovers the optimal retirement, consumption, leisure, and portfolio policies.}

The numerical analysis recovers the value function and optimal policies, and illustrates how labor supply flexibility affects consumption, portfolio choice, and retirement timing under borrowing constraints.
\end{abstract}

\vspace{1.0cm}

%{\em JEL classification} : E21, G11
%%%%%%%%%%%%%%%%%%%%%%%%%%%%%%%%%%%%%%%%%%%
\newpage

%%%%%%%%%%%%%%%%%%%%%%%%%%%%%%%%%%%%%%%%%%%%%%%%%%%%%%%%%%%%%%%%%%%%%%
\section{Introduction}

Retirement decisions are naturally intertwined with portfolio choice, consumption smoothing, and labor supply. In a finite-horizon setting with a mandatory retirement date, the option to retire early is valuable because the agent can trade off current labor income against leisure while adjusting financial investment over the remaining working life. When labor supply is flexible before retirement, this trade-off becomes richer: the agent does not simply choose between working and retiring, but also how intensively to work prior to retirement. At the same time, borrowing constraints are economically important because future labor income is not fully pledgeable. An agent may therefore have substantial future earnings and yet be unable to finance current consumption by borrowing against them. This raises a natural question: how do early retirement, labor supply flexibility, and borrowing constraints jointly shape optimal consumption, portfolio choice, leisure, and retirement timing?

Since the seminal work of \citet{YK}, a substantial literature has studied finite-horizon consumption--investment problems with an endogenous early retirement option; see, among others, \citet{YKS21}, \citet{CJW22}, \citet{JKP23}, \citet{PW23}, \citet{PWY23}, and \citet{JeonOh2023}. These papers have clarified how the retirement option interacts with partial information, ambiguity, and market environments, and they have significantly deepened our understanding of retirement timing in continuous-time portfolio problems. However, this line of research does not incorporate a borrowing constraint that prevents the agent from borrowing against future labor income. This omission is nontrivial, because such a constraint can materially affect savings, risky investment, wealth accumulation, and retirement incentives.

A recent paper by \citet{JeonKimYang2026} closes part of this gap by analyzing a finite-horizon retirement problem with a mandatory retirement date under a borrowing constraint. Nevertheless, that paper does not allow for flexible labor supply before retirement. In contrast, \citet{JeonOh2023} studies labor supply flexibility and early retirement, but without a borrowing constraint. The present paper brings these two strands together. We study a finite-horizon consumption--investment problem in which the agent chooses consumption, portfolio allocation, leisure, and an irreversible early retirement time under a nonnegative-wealth constraint. Before retirement, labor income depends on the chosen leisure level; after retirement, labor income disappears and leisure is fixed at its maximal level. Preferences are modeled by a Cobb--Douglas utility over consumption and leisure. In this way, the model captures in a unified framework the interaction between labor supply flexibility, early retirement, and borrowing constraints.

Labor supply flexibility has also been recognized more broadly as an important margin in continuous-time portfolio choice. Early contributions such as \citet{BodieMertonSamuelson1992} and \citet{BodieDetempleOtrubaWalter2004} show that the labor--leisure margin can substantially affect consumption, risky investment, and retirement planning over the life cycle. This line of research was extended by \citet{ChoiShimShin2008}, who combine flexible labor supply with an endogenous retirement decision in a continuous-time model with CES preferences. More recently, \citet{ChoiKwakLim2024} study flexible labor--leisure choice through the lens of work--life balance and recursive preferences, further highlighting the broader economic relevance of labor flexibility. Relative to these studies, our contribution is to analyze labor supply flexibility jointly with irreversible early retirement and a borrowing constraint in a finite-horizon setting.

{Our objective is to obtain a rigorous and computationally tractable solution. The borrowing constraint is handled by a non-increasing dual multiplier, while the irreversible retirement option becomes a discretionary stopping time. Thus the dual problem takes the form of a zero-sum game between a singular controller and a stopper. Its value is governed by a min--max parabolic variational inequality with a gradient constraint and an obstacle constraint, and the solution is encoded by two time-dependent free boundaries: one for retirement and one for the binding borrowing constraint.}

The dual reformulation also connects our paper to the growing literature on zero-sum games between a singular controller and a discretionary stopper. Recent studies have progressively expanded the analytical framework for such games. In particular, \citet{BDI2022} analyze finite-horizon games on multidimensional unbounded domains and characterize the value as a maximal Sobolev solution of a min--max variational inequality with gradient and obstacle constraints. \citet{BDP24} extend this framework to constrained control directions, \citet{Bovo2024} address degenerate diffusions through a perturbation argument, and \citet{BOVO2025104555} establish a saddle point and characterize the associated control and stopping boundaries in a finite-horizon setting. See also \citet{BovoDeAngelis2026} for a recent finite-horizon analysis on the half-line. Compared with this literature, our contribution is to derive such a controller--stopper game endogenously from an economically motivated consumption--investment problem with labor supply flexibility, irreversible retirement, and a borrowing constraint. We then combine this game formulation with a duality theorem that links the game value to the original primal problem and yields the optimal consumption, portfolio, leisure, and retirement policies in feedback form.

Although our overall duality strategy is close in spirit to \citet{JeonKimYang2026}, the present problem is substantially more intricate. The key difference is that our framework does not support a direct comparison principle for deriving the desired global gradient estimate. This poses a fundamental challenge in establishing the parametrization and monotonicity of the retirement and wealth-binding boundaries. To overcome this difficulty, we introduce a contradiction argument combined with the strong maximum principle. Moreover, building on the regularity results of \citet{Blanchet-et-al-2006}, we show that the gradient of the solution increases strictly from the slope of the obstacle at the retirement boundary, a property that follows from the positivity of the second-order derivative. This positivity implies the monotonicity of the gradient in the nearby non-contact region and thereby provides the key ingredient for parametrizing the gradient-constrained free boundary. It also plays an essential role in deriving an estimate that yields the local Lipschitz continuity of the free boundary. Unlike earlier approaches based on penalized problems and indirect estimates, we establish this estimate through a direct local comparison between the time derivative and the spatial derivative of the solution inside the non-contact region.

The contributions of the paper are both economic and methodological. On the economic side, we identify how labor supply flexibility reshapes retirement incentives when the agent cannot borrow against future earnings. The model shows how the labor--leisure margin serves as an intermediate adjustment mechanism before the agent exercises the irreversible retirement option, and how this channel interacts with portfolio choice under a binding wealth constraint. On the methodological side, we prove the existence and uniqueness of a strong solution to the dual variational inequality, establish a free-boundary characterization of the optimal policy, and derive a duality theorem that recovers the optimal consumption, leisure, portfolio, and retirement strategies in feedback form.

{The key verification point is the following. Starting from the singular-controller--stopper game, the dual variational inequality produces a maximal strong solution. This maximal solution is the only admissible strong solution consistent with the gradient-constrained free boundary: the binding boundary is monotone increasing in calendar time, or equivalently decreasing in the remaining-time variable used in the PDE analysis. The verification theorem then shows that this analytic object is not merely a candidate: it is exactly the value of the dual game, and hence the object that enters the primal-dual representation of the original optimization problem.}

The numerical analysis illustrates the economic effects of labor supply flexibility on consumption, portfolio choice, and retirement timing under borrowing constraints.

The rest of the paper is organized as follows. Section~\ref{sec:model} introduces the model and formulates the primal problem. Section~\ref{sec:optimization} develops the dual reformulation and derives the associated min--max type parabolic variational inequality. The subsequent sections analyze the variational inequality, characterize the two free boundaries, and establish the verification and duality results. Section~\ref{sec:numerical} presents the numerical implementation and discusses the economic implications of the model.

\section{Model}\label{sec:model}

We consider an agent who lives in a continuous-time economy.

\subsection*{Preferences}

The agent chooses consumption, leisure, portfolio holdings, and an early retirement time so as to maximize expected lifetime utility. Let $\tau_D$ denote the death time and let $\tau$ denote the early retirement time, where $0\le \tau\le T$ and $T>0$ is the mandatory retirement date. We assume that death occurs exogenously and is independent of financial risk.

Let $\widehat{\beta}>0$ be the pure subjective time-preference rate, and let $\lambda>0$ be the constant mortality hazard rate. Then the agent's objective is
\begin{equation}\label{eq:expected-utility}
\mathfrak{U}
=
\mathbb{E}\left[
\int_0^{\tau\wedge\tau_D} e^{-\widehat{\beta} t}u(c_t,l_t)\,dt
+
{\bf 1}_{\{\tau<\tau_D\}}
\int_\tau^{\tau_D} e^{-\widehat{\beta} t}u(c_t,\bar L)\,dt
\right].
\end{equation}
Here, $c_t\ge 0$ and $l_t\ge 0$ denote consumption and leisure at time $t$, respectively. The agent works before retirement and enjoys full leisure after retirement. As retirement is irreversible, once the agent retires at time $\tau$, she cannot return to work.

We let $\bar L$ denote the maximal leisure level. Before retirement, the agent must supply at least some labor, so leisure cannot exceed a fixed upper bound $L<\bar L$. Thus,
\begin{equation}\label{eq:condition-leisure}
0< l_t \le L < \bar L
\quad \text{for } 0\le t<\tau,
\qquad
l_t=\bar L
\quad \text{for } t\ge \tau.
\end{equation}
The quantity $\bar L-l_t$ therefore represents labor supply prior to retirement.

As in \citet{JeonOh2023}, we assume throughout the paper that preferences over consumption and leisure are represented by the Cobb--Douglas utility function
\begin{equation}\label{eq:utility}
u(c_t,l_t)
=
\frac{1}{\alpha}\frac{(c_t^\alpha l_t^{1-\alpha})^{1-\gamma}}{1-\gamma},
\qquad
\gamma\in(0,1)\cup(1,\infty),
\qquad
\alpha\in(0,1).
\end{equation}

Define
$$
\gamma_1:=1-\alpha(1-\gamma).
$$
Then \eqref{eq:utility} can be rewritten as
\begin{equation}\label{eq:utility2}
u(c_t,l_t)
=
\frac{c_t^{1-\gamma_1}l_t^{\gamma_1-\gamma}}{1-\gamma_1}.
\end{equation}

Because $\tau_D$ is independent of the financial market and has constant hazard rate $\lambda$, the survival probability is
$$
\mathbb{P}(\tau_D>t)=e^{-\lambda t}.
$$
Using this and Fubini's theorem, \eqref{eq:expected-utility} becomes
\begin{equation}\label{eq:expected-utility-reduced}
\mathfrak{U}
=
\mathbb{E}\left[
\int_0^\tau e^{-(\widehat{\beta}+\lambda)t}u(c_t,l_t)\,dt
+
\int_\tau^\infty e^{-(\widehat{\beta}+\lambda)t}u(c_t,\bar L)\,dt
\right].
\end{equation}
To simplify notation, we define
$$
\beta:=\widehat{\beta}+\lambda,
$$
and hereafter refer to $\beta$ simply as the subjective discount rate. Under this convention,
\begin{equation}\label{eq:expected-utility-reduced2}
\mathfrak{U}
=
\mathbb{E}\left[
\int_0^\tau e^{-\beta t}u(c_t,l_t)\,dt
+
\int_\tau^\infty e^{-\beta t}u(c_t,\bar L)\,dt
\right].
\end{equation}

\begin{rem}
Since mortality is exogenous and independent of financial shocks, absorbing the hazard rate into the discount rate is economically natural in the present setting. This reformulation is especially convenient here because the model does not include life insurance or a bequest motive.
\end{rem}

\subsection*{Financial Market}

The financial market consists of one risk-free asset and one risky asset. The risk-free rate is constant and equal to $r>0$. The risky asset price process $(S_t)_{t\ge 0}$ satisfies
\begin{equation}
\frac{dS_t}{S_t}=\mu\,dt+\sigma\,dB_t,
\end{equation}
where $\mu>r$, $\sigma>0$, and $(B_t)_{t\ge 0}$ is a standard Brownian motion on a complete filtered probability space
$$
(\Omega,\mathcal{F}_\infty,\mathbb{P};\mathbb{F}),
\qquad
\mathbb{F}=(\mathcal{F}_t)_{t\ge 0}.
$$

\subsection*{Wealth Dynamics and Admissible Strategies}

We assume that the wage rate is constant and equal to $\zeta>0$. Hence, before retirement, labor income is given by
$$
\zeta(\bar L-l_t).
$$

Let $(\pi_t)_{t\ge 0}$ denote the dollar amount invested in the risky asset. For a given strategy $(c,l,\pi,\tau)$, the corresponding wealth process $(W_t^{c,l,\pi,\tau})_{t\ge 0}$ evolves according to
\begin{equation}\label{eq:wealth-dynamics}
dW_t^{c,l,\pi,\tau}
=
\left[
rW_t^{c,l,\pi,\tau}
+(\mu-r)\pi_t
-c_t
+\zeta(\bar L-l_t){\bf 1}_{\{t<\tau\}}
\right]dt
+\sigma \pi_t\,dB_t,
\qquad
W_0^{c,l,\pi,\tau}=w.
\end{equation}

Given an initial wealth level $w\ge 0$, we call a quadruple $(c,l,\pi,\tau)$ admissible if the following conditions hold:
\begin{itemize}
\item[(a)]
$c_t>0$, $l_t>0$, and $\pi_t$ are $\mathbb{F}$-progressively measurable processes such that, for every $t\ge 0$,
\begin{align*}
\int_0^t c_s\,ds<\infty,
\qquad
\int_0^t l_s\,ds<\infty,
\qquad
\int_0^t \pi_s^2\,ds<\infty
\quad \text{a.s.};
\end{align*}

\item[(b)]
$\tau\in\mathcal{S}(0,T)$, where $\mathcal{S}(t,T)$ denotes the set of all $\mathbb{F}$-stopping times taking values in $[t,T]$;

\item[(c)]
$l_t$ satisfies \eqref{eq:condition-leisure};

\item[(d)]
the wealth process in \eqref{eq:wealth-dynamics} satisfies the borrowing constraint 
\begin{equation}\label{eq:wealth-constraint}
W_t^{c,l,\pi,\tau}\ge 0,
\qquad \forall t\ge 0
\quad \text{a.s.}
\end{equation}
\end{itemize}
We denote the set of all admissible strategies by $\mathcal{A}(w)$.

\begin{rem}
Condition \eqref{eq:wealth-constraint} imposes a non-negative wealth constraint, which can be interpreted as a no-borrowing condition. Since the model does not allow labor income to be collateralized through a natural borrowing limit, this formulation provides a transparent and economically standard solvency restriction.
\end{rem}

We may now state the agent's optimization problem.

\begin{pr}[Primal Problem]\label{pr:main}
Given $w\ge 0$, define
\begin{equation}
V(w)
=
\sup_{(c,l,\pi,\tau)\in\mathcal{A}(w)}
\mathbb{E}\left[
\int_0^\tau e^{-\beta t}u(c_t,l_t)\,dt
+
\int_\tau^\infty e^{-\beta t}u(c_t,\bar L)\,dt
\right].
\end{equation}
\end{pr}

To ensure that the post-retirement value is well defined, we impose the following assumption.

\begin{as}\label{as:K1-positive}
Let
\begin{equation}\label{eq:theta-def}
\theta:=\frac{\mu-r}{\sigma}
\end{equation}
denote the market price of risk, and define
\begin{equation}\label{eq:K1-def}
K_1
:=
r+\frac{\beta-r}{\gamma_1}
+\frac{\gamma_1-1}{2\gamma_1^2}\theta^2.
\end{equation}
We assume that
\[
K_1>0.
\]
\end{as}

\section{Optimization Problem}\label{sec:optimization}

We now reformulate the agent's problem in a way that is convenient for the subsequent dual analysis. Since retirement is irreversible and post-retirement leisure is fixed at the maximal level $\bar L$, it is natural to separate the post-retirement continuation problem from the pre-retirement optimization problem.

Once the agent retires, labor income disappears and leisure is permanently fixed at $\bar L$. The continuation problem after retirement is therefore the standard infinite-horizon consumption--investment problem with felicity $u(c,\bar L)$:
\begin{equation}\label{eq:retired-problem}
V_R(w)
:=
\sup_{(c,\pi)}
\mathbb{E}\left[
\int_0^\infty e^{-\beta t}u(c_t,\bar L)\,dt
\right],
\end{equation}
subject to
\begin{equation}\label{eq:retired-wealth}
dW_t
=
\bigl(rW_t+(\mu-r)\pi_t-c_t\bigr)\,dt
+\sigma \pi_t\,dB_t,
\qquad
W_0=w,
\end{equation}
and the non-negativity constraint
\begin{equation}\label{eq:retired-constraint}
W_t\ge 0
\qquad \text{for all } t\ge 0.
\end{equation}

Recall from Section~\ref{sec:model} that
\[
\gamma_1:=1-\alpha(1-\gamma),
\qquad
u(c,l)=\frac{c^{1-\gamma_1}l^{\gamma_1-\gamma}}{1-\gamma_1}.
\]
Hence
\[
u(c,\bar L)=\frac{c^{1-\gamma_1}\bar L^{\gamma_1-\gamma}}{1-\gamma_1}.
\]

By Assumption~\ref{as:K1-positive}, the post-retirement value function is
\begin{equation}\label{eq:VR}
V_R(w)
=
\frac{1}{1-\gamma_1}
K_1^{-\gamma_1}
\bar L^{\gamma_1-\gamma}
w^{1-\gamma_1}.
\end{equation}

Its convex dual is given by
\begin{equation}\label{eq:JR-def}
J_R(\xi)
:=
\sup_{w\ge 0}\bigl(V_R(w)-w\xi\bigr),
\qquad \xi>0,
\end{equation}
and a direct calculation yields
\begin{equation}\label{eq:JR}
J_R(\xi)
=
\frac{\bar L^{\frac{\gamma_1-\gamma}{\gamma_1}}}{K_1}
\frac{\gamma_1}{1-\gamma_1}
\xi^{-\frac{1-\gamma_1}{\gamma_1}}.
\end{equation}

The dual minimizer associated with wealth level \(w\) is characterized by
\[
V_R'(w)=\xi
\quad\Longleftrightarrow\quad
\xi
=
K_1^{-\gamma_1}\bar L^{\gamma_1-\gamma}w^{-\gamma_1},
\]
or equivalently,
\begin{equation}\label{eq:retired-dual-primal-map}
w
=
-J_R'(\xi)
=
\frac{\bar L^{\frac{\gamma_1-\gamma}{\gamma_1}}}{K_1}
\xi^{-1/\gamma_1}.
\end{equation}
Accordingly, the optimal post-retirement consumption and portfolio rules take the familiar Merton form
\begin{equation}\label{eq:retired-controls}
c_t^R = K_1 W_t,
\qquad
\pi_t^R=\frac{\theta}{\sigma\gamma_1}W_t.
\end{equation}

Before retirement, the agent chooses both consumption and leisure. Since labor income is equal to \(\zeta(\bar L-l)\), it is convenient to define the effective offset term
\begin{equation}\label{eq:eps-def}
\epsilon(l):=\zeta(\bar L-l),
\qquad 0<l\le L<\bar L.
\end{equation}
Equivalently,
\[
c+\epsilon(l)=c+\zeta(\bar L-l)=c+\zeta\bar L-\zeta l.
\]

In the dual formulation, this leads to the static maximization problem
\begin{equation}\label{eq:utilde-def}
\tilde{u}(\xi)
:=
\sup_{c\ge 0,\;0<l\le L}
\Bigl(
u(c,l)-\xi(c+\zeta l)+\zeta\bar L \xi
\Bigr),
\qquad \xi>0.
\end{equation}
Using the first-order conditions together with the upper constraint \(l\le L\), we obtain the threshold
\begin{equation}\label{eq:xitilde-def}
\tilde{\xi}
:=
\left(
\frac{\gamma_1-\gamma}{\zeta(1-\gamma_1)}
\right)^{\gamma_1}
L^{-\gamma},
\end{equation}
and the maximizers
\begin{equation}\label{eq:chat}
\hat{c}(\xi)
=
\begin{cases}
\left(\dfrac{\gamma_1-\gamma}{\zeta(1-\gamma_1)}\right)^{\frac{\gamma_1-\gamma}{\gamma}}
\xi^{-1/\gamma},
& \text{if } \xi>\tilde \xi, \\[3mm]
L^{\frac{\gamma_1-\gamma}{\gamma_1}}
\xi^{-1/\gamma_1},
& \text{if } 0<\xi\le \tilde \xi,
\end{cases}
\end{equation}
and
\begin{equation}\label{eq:lhat}
\hat{l}(\xi)
=
\begin{cases}
\left(\dfrac{\gamma_1-\gamma}{\zeta(1-\gamma_1)}\right)^{\frac{\gamma_1}{\gamma}}
\xi^{-1/\gamma},
& \text{if } \xi>\tilde \xi, \\[3mm]
L,
& \text{if } 0<\xi\le \tilde \xi.
\end{cases}
\end{equation}
Substituting these expressions into \eqref{eq:utilde-def}, we obtain
\begin{equation}\label{eq:utilde}
\begin{aligned}
\tilde{u}(\xi)
&=
\left[
\frac{\zeta\gamma}{\gamma_1-\gamma}
\left(
\frac{\gamma_1-\gamma}{\zeta(1-\gamma_1)}
\right)^{\frac{\gamma_1}{\gamma}}
\xi^{-\frac{1-\gamma}{\gamma}}
+\zeta\bar L \xi
\right]
{\bf 1}_{\{\xi>\tilde \xi\}}
\\
&\quad
+
\left[
\frac{\gamma_1}{1-\gamma_1}
L^{\frac{\gamma_1-\gamma}{\gamma_1}}
\xi^{-\frac{1-\gamma_1}{\gamma_1}}
-\zeta L \xi+\zeta\bar L \xi
\right]
{\bf 1}_{\{0<\xi\le \tilde \xi\}}.
\end{aligned}
\end{equation}

For later dynamic arguments, it is convenient to introduce the continuation value from an arbitrary date \(t\in[0,T]\). Let \(\mathcal A(t,w)\) denote the natural continuation analogue of \(\mathcal A(w)\), namely the set of admissible controls starting from time \(t\) with current wealth \(w\). Then the continuation value is
\begin{equation}\label{eq:continuation-value}
V(t,w)
=
\sup_{(c,l,\pi,\tau)\in\mathcal A(t,w)}
\mathbb{E}_t\left[
\int_t^\tau e^{-\beta(s-t)}u(c_s,l_s)\,ds
+
e^{-\beta(\tau-t)}V_R\bigl(W_\tau^{c,l,\pi,\tau}\bigr)
\right].
\end{equation}
In particular, the original primal problem in Section~\ref{sec:model} is recovered as \(V(w)=V(0,w)\).

We now develop a dual formulation for the continuation problem \eqref{eq:continuation-value}. The dual martingale method is particularly convenient here because it allows us to replace the pathwise wealth constraint by a single static budget constraint involving a non-increasing multiplier process.

Let \(H\) denote the state-price density process,
\begin{equation}\label{eq:spd}
H_t
:=
\exp\left(
-r t-\frac{1}{2}\theta^2 t-\theta B_t
\right),
\qquad t\ge 0,
\end{equation}
and, for \(0\le s_1\le s_2\le T\), let \(\mathcal{NI}(s_1,s_2)\) be the set of all \(\mathbb{F}\)-adapted, non-negative, non-increasing, right-continuous processes with left limits on \([s_1,s_2]\) satisfying \(D_{s_1-}=1\).

The following proposition converts the dynamic non-negativity constraint on wealth into a static budget inequality.

\begin{pro}\label{pro:static-budget}
Fix \(t\in[0,T]\) and \(w\ge 0\).

\begin{itemize}
\item[(a)]
For any admissible strategy \((c,l,\pi,\tau)\in\mathcal A(t,w)\),
\begin{equation}\label{eq:static-budget-a}
\sup_{D\in\mathcal{NI}(t,T)}
\mathbb{E}_t\left[
\int_t^\tau H_s D_s\bigl(c_s-\zeta(\bar L-l_s)\bigr)\,ds
+
H_\tau W_\tau^{c,l,\pi,\tau}D_{\tau-}
\right]
\le wH_t.
\end{equation}

\item[(b)]
Conversely, let \(\tau\in\mathcal S(t,T)\), let \(c\) and \(l\) be non-negative progressively measurable processes satisfying the admissibility conditions on \([t,\tau)\), and let \(\mathfrak B\) be a non-negative \(\mathcal F_{\tau-}\)-measurable random variable such that
\begin{equation}\label{eq:static-budget-b}
\sup_{D\in\mathcal{NI}(t,T)}
\mathbb{E}_t\left[
\int_t^\tau H_s D_s\bigl(c_s-\zeta(\bar L-l_s)\bigr)\,ds
+
H_\tau \mathfrak B D_{\tau-}
\right]
=
wH_t.
\end{equation}
Then there exists a portfolio process \(\pi\) such that \((c,l,\pi,\tau)\in\mathcal A(t,w)\) and
\[
W_\tau^{c,l,\pi,\tau}\ge \mathfrak B.
\]
\end{itemize}
\end{pro}

\begin{proof}
See \citet[Proposition~1]{JeonKimYang2026}.
\end{proof}
Fix \(t\in[0,T]\), \(w\ge 0\), and \(\tau\in\mathcal S(t,T)\). For a Lagrange multiplier \(\xi>0\), define
\begin{equation}\label{eq:Xi-process}
\Xi_s^{t,\xi}
:=
\xi e^{\beta(s-t)}\frac{H_s}{H_t},
\qquad s\ge t.
\end{equation}
Then the Lagrangian associated with \eqref{eq:static-budget-a} is
\begin{equation}\label{eq:Lagrangian}
\begin{aligned}
\mathfrak L
:=
&\;
\mathbb{E}_t\left[
\int_t^\tau e^{-\beta(s-t)}u(c_s,l_s)\,ds
+
e^{-\beta(\tau-t)}V_R\bigl(W_\tau^{c,l,\pi,\tau}\bigr)
\right]
\\
&\;
+\xi\left(
w-
\sup_{D\in\mathcal{NI}(t,T)}
\frac{1}{H_t}
\mathbb{E}_t\left[
\int_t^\tau H_s D_s\bigl(c_s-\zeta(\bar L-l_s)\bigr)\,ds
+
H_\tau W_\tau^{c,l,\pi,\tau}D_{\tau-}
\right]
\right).
\end{aligned}
\end{equation}
Using the definition of \(\tilde u\) in \eqref{eq:utilde-def} and the post-retirement dual transform \(J_R\), we obtain
\begin{equation}\label{eq:Lagrangian-bound}
\mathfrak L
\le
\inf_{D\in\mathcal{NI}(t,T)}
\mathbb{E}_t\left[
\int_t^\tau e^{-\beta(s-t)}
\tilde u\bigl(\mathcal Z_s^{t,\xi,D}\bigr)\,ds
+
e^{-\beta(\tau-t)}J_R\bigl(\mathcal Z_{\tau-}^{t,\xi,D}\bigr)
\right]
+\xi w,
\end{equation}
where
\begin{equation}\label{eq:Z-def}
\mathcal Z_s^{t,\xi,D}
:=
\Xi_s^{t,\xi}D_s,
\qquad
\mathcal Z_t^{t,\xi,D}=\xi.
\end{equation}

Here the running term in \eqref{eq:Lagrangian-bound} is simply
\(\tilde u(\mathcal Z)\), because the static maximization defining
\(\tilde u\) already incorporates both leisure choice and labor income.

For fixed \(\xi>0\) and \(D\in\mathcal{NI}(t,T)\), the corresponding candidate controls are
\begin{equation}\label{eq:candidate-controls-pre}
c_s=\hat c\bigl(\mathcal Z_s^{t,\xi,D}\bigr),
\qquad
l_s=\hat l\bigl(\mathcal Z_s^{t,\xi,D}\bigr),
\qquad s\in[t,\tau),
\end{equation}
and the candidate terminal wealth at retirement is
\begin{equation}\label{eq:candidate-terminal-wealth}
W_\tau^{c,l,\pi,\tau}
=
-J_R'\bigl(\mathcal Z_{\tau-}^{t,\xi,D}\bigr)
=
\frac{\bar L^{\frac{\gamma_1-\gamma}{\gamma_1}}}{K_1}
\bigl(\mathcal Z_{\tau-}^{t,\xi,D}\bigr)^{-1/\gamma_1}.
\end{equation}

Since \(D\) is non-increasing, the process \(\mathcal Z^{t,\xi,D}\) evolves as a singularly controlled diffusion:
\begin{equation}\label{eq:Z-dynamics}
d\mathcal Z_s^{t,\xi,D}
=
(\beta-r)\mathcal Z_s^{t,\xi,D}\,ds
-
\theta \mathcal Z_s^{t,\xi,D}\,dB_s
+
\Xi_s^{t,\xi}\,dD_s,
\qquad s\ge t.
\end{equation}

Motivated by \eqref{eq:Lagrangian-bound}, define
\begin{equation}\label{eq:game-value}
\mathcal J(t,z;D,\tau)
:=
\mathbb{E}_t\left[
\int_t^\tau e^{-\beta(s-t)}
\tilde u\bigl(\mathcal Z_s^{t,z,D}\bigr)\,ds
+
e^{-\beta(\tau-t)}J_R\bigl(\mathcal Z_{\tau-}^{t,z,D}\bigr)
\right],
\end{equation}
where \(\mathcal Z_t^{t,z,D}=z\). Then, for any admissible strategy,
\begin{equation}\label{eq:weak-duality-raw}
V(t,w)
\le
\inf_{\xi>0}
\left(
\sup_{\tau\in\mathcal S(t,T)}
\inf_{D\in\mathcal{NI}(t,T)}
\mathcal J(t,\xi;D,\tau)
+
\xi w
\right).
\end{equation}
This is the weak duality relation underlying our approach.

The right-hand side of \eqref{eq:weak-duality-raw} leads naturally to a two-person zero-sum game. The controller, acting as the minimizer, chooses a process \(D\in\mathcal{NI}(t,T)\), while the stopper, acting as the maximizer, chooses a stopping time \(\tau\in\mathcal S(t,T)\). Their common payoff functional is \(\mathcal J(t,z;D,\tau)\).

Define the lower and upper values
\begin{equation}\label{eq:lower-upper-values}
\underline J(t,z)
:=
\sup_{\tau\in\mathcal S(t,T)}
\inf_{D\in\mathcal{NI}(t,T)}
\mathcal J(t,z;D,\tau),
\qquad
\overline J(t,z)
:=
\inf_{D\in\mathcal{NI}(t,T)}
\sup_{\tau\in\mathcal S(t,T)}
\mathcal J(t,z;D,\tau).
\end{equation}
If \(\underline J(t,z)=\overline J(t,z)\), we say that the game has a value and denote the common value by \(J(t,z)\).

\begin{pr}[Dual Problem]\label{pr:dual}
If the stopping--control game admits a value, then the dual value function is defined by
\begin{equation}\label{eq:dual-problem}
J(t,z)
=
\sup_{\tau\in\mathcal S(t,T)}
\inf_{D\in\mathcal{NI}(t,T)}
\mathcal J(t,z;D,\tau)
=
\inf_{D\in\mathcal{NI}(t,T)}
\sup_{\tau\in\mathcal S(t,T)}
\mathcal J(t,z;D,\tau).
\end{equation}
\end{pr}

With this notation, weak duality takes the compact form
\begin{equation}\label{eq:weak-duality}
V(t,w)\le \inf_{\xi>0}\bigl(J(t,\xi)+\xi w\bigr).
\end{equation}

By the dynamic programming principle, one expects that the dual value function satisfies both the max--min and min--max forms of the Hamilton--Jacobi--Bellman quasi-variational inequality. Writing
\begin{equation}\label{eq:dual-operator}
\mathcal L
:=
\frac{\theta^2}{2}z^2\partial_{zz}
+
(\beta-r)z\partial_z
-
\beta,
\end{equation}
on the domain
\[
\mathcal R_T:=\{(t,z)\mid 0\le t<T,\; z>0\},
\]
the candidate HJBQVI is
\begin{align}
\begin{cases}\label{eq:HJBQVI1}
\max\left\{
\min\left\{
\partial_t J+\mathcal L J+\tilde u(z),
\;
-\partial_z J
\right\},
\;
J_R(z)-J(t,z)
\right\}=0,
& \text{in } \mathcal R_T,
\\[2mm]
J(T,z)=J_R(z),
& z>0,
\end{cases}
\end{align}
and equivalently
\begin{align}
\begin{cases}\label{eq:HJBQVI2}
\min\left\{
\max\left\{
\partial_t J+\mathcal L J+\tilde u(z),
\;
J_R(z)-J(t,z)
\right\},
\;
-\partial_z J
\right\}=0,
& \text{in } \mathcal R_T,
\\[2mm]
J(T,z)=J_R(z),
& z>0.
\end{cases}
\end{align}

The term \(J_R(z)-J(t,z)\) corresponds to the stopper's retirement decision, while the gradient constraint \(-\partial_z J\) reflects the minimizer's singular control through the non-increasing process \(D\). In the next section, we show that this HJBQVI can be recast as a double obstacle problem and solved via the associated free-boundary characterization.

For convenience, the definitions of the function spaces used throughout the paper are collected in Appendix~\ref{sec:appendix-function-spaces}.

To obtain a formulation that captures both parabolic HJBQVIs \eqref{eq:HJBQVI1} and \eqref{eq:HJBQVI2} simultaneously, we introduce the following parabolic variational inequality with an obstacle constraint and a gradient constraint.

\begin{pr}\label{eq:main3}
	Find a unique solution
	\[
	\mathcal{Q}\in W^{1,2}_{p,\mathrm{loc}}(\mathcal{D}_T)\cap C^{1,1}(\mathcal{D}_T)\cap C^{1,2}(\mathcal{D}_T\backslash\{\mathcal{Q}=J_R\})
	\]
	for some \(p\ge 3\), to the following parabolic variational inequality:
	\begin{equation}\label{obs:Q}
		\left\{\begin{aligned}
			&\partial_t \mathcal{Q}+\mathscr{L}\mathcal{Q}+\tilde{u}(z) =0,
			\quad && \text{if }\partial_z\mathcal{Q}< 0\text{ and }\mathcal{Q}>J_R,\\
			&\partial_t \mathcal{Q}+\mathscr{L}\mathcal{Q}+\tilde{u}(z)  \geq0,
			\quad && \text{if }\partial_z\mathcal{Q}= 0,\\
			&\partial_t \mathcal{Q}+\mathscr{L}\mathcal{Q}+\tilde{u}(z) \leq0,
			\quad && \text{if }\mathcal{Q}=J_R ,\\
			&\mathcal{Q}\geq J_R\ \text{ and }\ \partial_z\mathcal{Q}\leq 0,
			\quad&&\text{in }\mathcal{D}_T,\\
			&\mathcal{Q}(T, z) = J_R(z),
			\quad && \text{for }z>0,
		\end{aligned}\right.
	\end{equation}
\end{pr}
where \(\mathcal{D}_T:=\{(t,z)\mid 0\le t<T,\;0<z<\infty\}\), and the differential operator \(\mathscr{L}\) is given by
\[
\mathscr{L}:= \frac{\theta^2}{2}z^2\partial_{zz}+(\beta-r)z\partial_z - \beta.
\]

\section{Rigorous Analysis of the Variational Inequality \eqref{obs:Q}}

We begin by introducing the change of variables
\[
\tau=T-t,\qquad z=e^x,\qquad \text{and} \qquad Y(\tau,x)=\mathcal{Q}(T-\tau,e^x).
\]
Under this transformation, problem \eqref{obs:Q} becomes
\begin{equation}\label{obs1}
\begin{cases}
\partial_\tau Y(\tau,x)-\mathcal{L}Y(\tau,x)=f(x)
\quad \text{for } (\tau,x)\in\Omega_T \text{ such that } \partial_x Y(\tau,x)<0 \text{ and } Y(\tau,x)>\psi(x),
\\
\partial_\tau Y(\tau,x)-\mathcal{L}Y(\tau,x)\le f(x)
\quad \text{for } (\tau,x)\in\Omega_T \text{ such that } \partial_x Y(\tau,x)=0,
\\
\partial_\tau Y(\tau,x)-\mathcal{L}Y(\tau,x)\ge f(x)
\quad \text{for } (\tau,x)\in\Omega_T \text{ such that } Y(\tau,x)=\psi(x),
\\
Y(0,x)=\psi(x),
\quad \text{for } x\in\mathbb{R},
\end{cases}
\end{equation}
where
\[
\mathcal{L}:=\frac{\theta^2}{2}\partial_{xx}
+\left(\beta-r-\frac{\theta^2}{2}\right)\partial_x
-\beta,
\qquad
\Omega_T:=(0,T)\times\mathbb{R},
\]
\begin{align*}
\psi(x)
=J_R(e^x)
=&\frac{\bar{L}^{\frac{\gamma_1-\gamma}{\gamma_1}}}{K_1}\frac{\gamma_1}{1-\gamma_1}e^{\left(1-\frac{1}{\gamma_1}\right)x},
\\
f(x)
=\tilde{u}(e^x)=&
\left[
\frac{\gamma_1}{1-\gamma_1}L^{\frac{\gamma_1-\gamma}{\gamma_1}}e^{\left(1-\frac{1}{\gamma_1}\right)x}
+\zeta(\bar{L}-L)e^x
\right]
{\bf 1}_{\{x\le \tilde{x}\}}
\\
&\quad
+
\left[
\frac{\zeta\gamma}{\gamma_1-\gamma}
\left(
\frac{\gamma_1-\gamma}{\zeta(1-\gamma_1)}
\right)^{\frac{\gamma_1}{\gamma}}
e^{\left(1-\frac{1}{\gamma}\right)x}
+\zeta\bar{L}e^x
\right]
{\bf 1}_{\{x>\tilde{x}\}},
\end{align*}
and
\begin{equation*}
\tilde{x}=\ln(\tilde{\xi})
:=
\ln\left[\left(
\frac{\gamma_1-\gamma}{\zeta(1-\gamma_1)}
\right)^{\gamma_1}
L^{-\gamma}\right].
\end{equation*}
\begin{rem}\label{f_1 f_2}
Write
\[
f(x)=f_1(x){\bf 1}_{\{x\le \tilde{x}\}}+f_2(x){\bf 1}_{\{x>\tilde{x}\}},
\]
where
\[
f_1(x)=\frac{\gamma_1}{1-\gamma_1}L^{\frac{\gamma_1-\gamma}{\gamma_1}}e^{\left(1-\frac{1}{\gamma_1}\right)x}
+\zeta(\bar{L}-L)e^x
\]
and
\[
f_2(x)=
\frac{\zeta\gamma}{\gamma_1-\gamma}
\left(
\frac{\gamma_1-\gamma}{\zeta(1-\gamma_1)}
\right)^{\frac{\gamma_1}{\gamma}}
e^{\left(1-\frac{1}{\gamma}\right)x}
+\zeta\bar{L}e^x.
\]
A direct calculation shows that \(f_1(\tilde{x})=f_2(\tilde{x})\), that
\(f_1'(x)\ge f_2'(x)\) for all \(x\le \tilde{x}\), and that
\(f_1'(x)\le f_2'(x)\) for all \(x\ge \tilde{x}\). Hence,
\[
f=\max\{f_1,f_2\}.
\]
Moreover, since \(f_1'(\tilde{x})=f_2'(\tilde{x})\), it follows that
\(f\in C^{1+\alpha}_{\mathrm{loc}}(\mathbb{R})\) for all \(\alpha\in(0,1)\).
\end{rem}

\subsection{Existence of a strong solution}
To find the solution to \eqref{obs1}, we first consider a sequence of functions $\{Y_n\}_{n=1}^\infty$ in $\O_T^n:=(0,T)\times(-n,n)$, where $Y_n$ is defined to be the solution to the following variational inequality.
    \begin{equation}\label{obs1:bdd}
        \begin{cases}
            \partial_\tau Y_n(\tau,x) -\mathcal{L}Y_n(\tau,x)= f(x)
            \quad \text{for }\ (\tau,x)\in\O_T^n
            \ \text{ with }\ Y_n(\tau,x)>\psi(x)
            \ \text{ and }\ \pat_xY_n(\tau,x)<0,
            \\
            \partial_\tau Y_n(\tau,x) -\mathcal{L}Y_n(\tau,x)\ge f(x)
            \quad \text{for }\ (\tau,x)\in\O_T^n
            \ \text{ with }\ Y_n(\tau,x)=\psi(x)
            \\
            \partial_\tau Y_n(\tau,x) -\mathcal{L}Y_n(\tau,x)\le f(x)
            \quad \text{for }\ (\tau,x)\in\O_T^n
            \ \text{ with }\ \pat_xY_n(\tau,x)=0,
            \\
            \pat_x Y_n(\tau,-n)= \psi'(-n)
            \quad \text{and} \quad 
            \pat_x Y_n(\tau,n)= \psi'(n)
            \quad \text{for all }\ \tau\in(0,T),
            \\
            Y_n(0,x)=\psi(x) \quad \text{for }\ x\in[-n,n],
        \end{cases}
    \end{equation}
We then approximate the above problem by the following penalized problem:
    \begin{equation}\label{obs1:pen}
        \begin{cases}
            \partial_\tau Y_{n,\e} -\mathcal{L}Y_{n,\e}
            =f+\b_{l,\e}(Y_{n,\e}-\psi)+\b_{u,\e}(\pat_x Y_{n,\e}-\varphi_\e)
            \quad \text{in }\ \Omega_T^n,
            \\
            \pat_xY_{n,\e}(\tau,-n)= \psi'(-n)
            \quad \text{and}\quad 
            \pat_x Y_{n,\e}(\tau,n)= \psi'(n)
            \quad \text{for \ all }\ \tau\in(0,T),
            \\
            Y_{n,\e}(0,x)=\psi(x) \quad \text{for all }\ x\in[-n,n],
        \end{cases}
    \end{equation}
    where $\b_{l,\e}(\cdot)\in C^\infty(\mathbb{R})$ and
    $\b_{u,\e}(\cdot)\in C^\infty(\mathbb{R})$
    are functions such that
    \begin{equation*}
        \begin{cases}
            \b_{l,\e}(s)\ge 0 \quad \text{for \ all }\ s\in\mathbb{R},
            \\
            \b_{l,\e}(s)=0 \quad \text{for \ all }\ s\ge \e,
            \\
            \b_{l,\e}'(s)\le 0
            \quad \text{for \ all }\ s\in\mathbb{R},
            \\
            \b_{l,\e}(0)=\frac{\g_1}{1-\g_1}(\bar{L}^\gam-L^\gam)e^{(1+\frac{1}{\g_1})n}-\zeta(\bar{L}-L)e^n,
        \end{cases}
        \begin{cases}
            \b_{u,\e}(s)\le 0 \quad \text{for \ all }\ s\in\mathbb{R},
            \\
            \b_{u,\e}(s)=-\frac{2}{\e}s-1
            \quad \text{for all }\ s\ge 0,
            \\
            \b_{u,\e}(s)=0 \quad \text{for \ all }\ s\le -\e,
            \\
            \b_{u,\e}'(s),\ \b_{u,\e}''(s)\le 0
            \quad \text{for \ all }\ s\in\mathbb{R},
        \end{cases}
    \end{equation*}
    % \begin{equation*}
    % \b_{u,\e}(\l)=
    %     \begin{cases}
    %         0, \quad &\text{if }\ \l\le -\e,
    %         \\
    %         -(\l+\e)^2, \quad &\text{if }\ \l>-\e.
    %     \end{cases}
    % \end{equation*}
and \[\varphi_\e(x)=-\zeta\bar{L}\e e^x
\quad \text{for all }\ x\in\mathbb{R}.\]
% In particular, since the function $f$ is not sufficiently regular, we have introduced the smoothened function $f$
% which is defined by
% \[f(x)=m_\e(f_2(x)-f_1(x))+f_1(x)
% \quad \text{for all }\ x\in \mathbb{R}.\]
% Here, $m_\e(\cdot)\in C^\infty(\mathbb{R})$ is the function which satisfies
% \begin{equation*}
%     \begin{cases}
%         m_\e(s)\to s^+ \quad \text{as }\ \e\to 0^+,
%         \\
%         m_\e(s)=0 \quad \text{if } \ s\le -\e,
%         \quad
%         m_\e(s)=s \quad \text{if } \ s\ge \e,
%         \\
%         s^+\le m_\e(s)\le (s+\e)^+
%         \quad \forall \ s\in\mathbb{R},
%         \\
%         0\le m'(s)\le 1
%         \quad \forall \ s\in\mathbb{R}.
%     \end{cases}
% \end{equation*}
% Thus, $f(x)\le f(x)$
% for all $x\in\mathbb{R}$.
% Moreover, since $f'_i(x)\le \z\bar{L}e^x$ for $i=1,2$,
% we also deduce from $0\le m'_\e(\cdot)\le 1$ that $f'(x)\le \z\bar{L}e^x$ for all $x\in\mathbb{R}$.
\begin{lem}\label{Y n,e}
    Let $n\in\mathbb{N}$. Then for each $\e>0$ any $\a\in(0,1)$, there exists a unique solution $Y_{n,\e}\in C^{\frac{3+\a}{2},3+\a}({\Omega_T^n})$ to the problem \eqref{obs1:pen}.
    Furthermore, $Y_{n,\e}$ satisfies the following properties:
\begin{itemize}
    \item [(1)] \textbf{Constraints:} $Y_{n,\e}(\tau, x) \ge \psi(x)$ and $\partial_x Y_{n,\e}(\tau, x) \le 0$ for all $(\tau, x) \in \Omega_T^n$.
    \item [(2)] \textbf{Monotonicity:} $Y_{n,\e}$ is non-decreasing in time, i.e., $\partial_\tau Y_{n,\e}(\tau, x) \ge 0$ for all $(\tau, x) \in \Omega_T^n$.
    \item [(3)] \textbf{Growth Condition:} There exist constants $C > 0$ and $d > 0$ such that 
    \[
    |Y_{n,\e}(\tau, x)| \le Ce^{d|x|} \quad \text{for all } (\tau, x) \in \Omega_T^n,
    \]
    where the constants $C$ and $d$ are independent of $\e>0$ and $n\in \mathbb{N}$.
\end{itemize}
\end{lem}

\begin{proof}
{ See Appendix~\ref{app:proof:Y-n-e}.}
\end{proof}
\begin{lem}\label{Yn}
   For each $n \in \mathbb{N}$, there exists a solution $Y_n \in W^{1,2}_p(\Omega_T^n) \cap C(\overline{\Omega_T^n})$ for $1 < p < \infty$ to the problem \eqref{obs1:bdd}. 
\end{lem}

\begin{proof}
{ See Appendix~\ref{app:proof:Yn}.}
\end{proof}
\begin{lem}\label{Y}
    There exists a solution $Y\in W^{1,2}_{p,{\rm loc}}(\Omega_T)\cap C(\overline{\Omega_T})$, $1<p<\infty$ to the problem \eqref{obs1}.
\end{lem}

\begin{proof}
{ See Appendix~\ref{app:proof:Y}.}
\end{proof}
\subsection{Free boundary analysis}
{
Before entering the technical details, we summarize the logic of this subsection. The obstacle contact set first determines the retirement boundary, while the gradient constraint is analyzed through an auxiliary obstacle problem for the spatial derivative. The central boundary for the verification/uniqueness step is the binding boundary associated with the gradient constraint: in the transformed variables it is \(\mathcal X^B(\tau)\), and in the original variables it becomes \(z_B(t)=\exp(\mathcal X^B(T-t))\), which is monotone increasing in calendar time \(t\).

{
\[
\resizebox{0.93\textwidth}{!}{%
\begin{tikzpicture}[
    box/.style={
        rectangle,
        rounded corners,
        draw=black,
        thick,
        align=center,
        inner sep=7pt,
        minimum width=3.45cm,
        minimum height=1.15cm,
        text=black,
        font=\small
    },
    widebox/.style={
        rectangle,
        rounded corners,
        draw=black,
        thick,
        align=center,
        inner sep=8pt,
        minimum width=10.9cm,
        minimum height=0.95cm,
        text=black,
        font=\small
    },
    arrow/.style={
        -{Latex[length=3.5mm,width=2.5mm]},
        very thick,
        draw=black
    },
    title/.style={
        text=black,
        align=center,
        font=\small\bfseries
    }
]

% ==========================================================
% Step 1: stopping/free-retirement structure
% ==========================================================
\node[title] at (0,2.20) {Step 1: stopping/free-retirement structure};

\node[box] (Y) at (-4.8,0.75) {
Dual value\\[3pt]
$Y$
};

\node[box] (Y0) at (0,0.75) {
Normalized value\\[3pt]
$Y_0=e^{-x}Y$
};

\node[box, minimum width=4.25cm] (R) at (4.8,0.75) {
Stopping contact set\\[3pt]
$\{Y_0=\psi_0\}$\\[3pt]
retirement boundary $\mathcal X^R$
};

\draw[arrow] (Y) -- (Y0);
\draw[arrow] (Y0) -- (R);

% ==========================================================
% Transition: differentiation on the continuation region
% ==========================================================
\node[widebox] (T) at (0,-1.35) {
Differentiate the VI in $x$ on the continuation region
$\mathcal C_T=\{Y>\psi\}$
};

\draw[arrow] (Y0.south) -- (T.north);
\draw[arrow] (T.south) -- ++(0,-0.65);

% ==========================================================
% Step 2: gradient constraint and binding boundary
% ==========================================================
\node[title] at (0,-2.70) {Step 2: gradient constraint and binding boundary};

\node[box] (V) at (-4.8,-4.25) {
Gradient of the VI\\[3pt]
$\partial_xY=V$
};

\node[box] (V0) at (0,-4.25) {
Normalized gradient\\[3pt]
$V_0=e^{-x}V$
};

\node[box, minimum width=4.25cm] (B) at (4.8,-4.25) {
Binding contact set\\[3pt]
$\{V_0=0\}$\\[3pt]
binding boundary $\mathcal X^B$
};

\node[box, minimum width=5.0cm] (zB) at (4.8,-6.80) {
Original time scale\\[3pt]
$z_B(t)=\exp(\mathcal X^B(T-t))$\\[3pt]
is increasing in $t$
};

\draw[arrow] (V) -- (V0);
\draw[arrow] (V0) -- (B);
\draw[arrow] (B) -- (zB);

\end{tikzpicture}%
}
\]
}

Thus, the retirement boundary organizes the obstacle-contact region, whereas the binding boundary organizes the gradient-constrained region. The maximal strong solution is the one compatible with this monotone binding-boundary structure, and the verification theorem later identifies it with the value of the dual stopper--singular-controller game.
}
\subsubsection{Investigation of the contact region}
In this section, we analyze the location of the contact region $\{Y=\psi\}$ and its boundary. 
In order to verify the region where such boundary emerges, we first construct suitable auxiliary functions and apply the comparison principle for variational inequalities \cite[Theorem~1.2]{YangYan2008}.
Since the function $f$ appearing on the right-hand side consists of several terms with different exponential rates, we introduce the following transformed problem:
\begin{equation}\label{obs:Y_0}
        \begin{cases}
            \partial_\tau Y_0(\tau,x) -\widetilde{\cL}Y_0(\tau,x)=f_0(x)
            \quad \text{if }\ Y_0(\tau,x)>\psi_0(x),\text{ and }\ \pat_x Y_0(\tau,x)+Y_0(\tau,x)<0
            \quad  \forall (\tau,x)\in\O_T,
            \\
            \partial_\tau Y_0(\tau,x) -\widetilde{\cL}Y_0(\tau,x)\le f_0(x)
            \quad \text{if }\ \pat_x Y_0(\tau,x)+Y_0(\tau,x)=0
            \quad  \forall (\tau,x)\in\O_T,
            \\
            \partial_\tau Y_0(\tau,x) -\widetilde{\cL}Y_0(\tau,x)\ge f_0(x)
            \quad \text{if }\ Y_0(\tau,x)=\psi_0(x)
            \quad  \forall (\tau,x)\in\O_T,
            \\
            Y_0(0,x)=\psi_0(x) \quad \text{for \ all }\ x\in \mathbb{R},
        \end{cases}
    \end{equation}
where $\widetilde{\cL}:=\frac{\t^2}{2}\pat_{xx}+(\b-r+\frac{\t^2}{2})\pat_x-r$, and
\[Y_0(\tau,x)=e^{-x}Y(\tau,x),\quad  f_0(x)=e^{-x}f(x), \quad \psi_0(x)=e^{-x}\psi(x)
\quad \text{for \ all }\ (\tau,x)\in \O_T.\]
Since the contact set $\{Y=\psi\}$ is preserved under this transformation, that is, $\{Y=\psi\}=\{Y_0=\psi_0\}$,
the free boundaries also coincide.
Consequently, every result obtained for the transformed pair $(Y_0,\psi_0)$ carries over directly to the original functions
$(Y,\psi)$.

% \begin{rem}
%     It is important to verify the points where $\pat_\tau \psi_0 -\mathcal{L}\psi_0 \ge u_0$ holds.
%     Note that 
% \end{rem}

\begin{lem}\label{Y>psi}
    Let $Y_0$ be the solution to the problem \eqref{obs:Y_0}. Then,
    $Y_0>\psi_0$ in $(0,T)\times(x^*,\infty)$
    where
    \begin{equation}\label{x^*}
        x^*=\g_1\ln\left( \frac{\g_1}{1-\g_1}\frac{\bar{L}^{\frac{\g_1-\g}{\g_1}}-L^{\frac{\g_1-\g}{\g_1}}}{\zeta (\bar{L}-L)} \right).
    \end{equation}
\end{lem}

\begin{proof}
{ See Appendix~\ref{app:proof:Y-psi}.}
\end{proof}
\begin{lem}\label{contact_1}
    Let $Y_0$ be the solution to the problem \eqref{obs:Y_0}. Then, there exists $\underline{x}<x^*$ such that
    \[Y_0=\psi_0 \quad \text{in }\ [0,T]\times (-\infty,\underline{x}],\]
    where $x^*$ is the point stated in Lemma~\ref{Y>psi}.
\end{lem}

\begin{proof}
{ See Appendix~\ref{app:proof:contact-1}.}
\end{proof}
\begin{cor}\label{X^R region}
    The free boundary $\partial\{Y_0>\psi_0\}$ of the problem \eqref{obs:Y_0} is contained in the region $(0,T)\times(\underline{x},x^*)$, where $x^*$ and $\underline{x}$ are the constants given in Lemma~\ref{Y>psi} and Lemma~\ref{contact_1}, respectively.
\end{cor}

\begin{lem}\label{contact_2}
    Let $Y_0$ be the solution to the problem \eqref{obs:Y_0}. Then for each $\e>0$, there exists a constant $c_\e>0$ such that 
    \[Y_0=\psi_0 \quad \text{in }\ [0,c_\e]\times(-\infty,x^*-2\e],\]
    where $x^*$ is the point stated in Lemma~\ref{Y>psi}.
\end{lem}

\begin{proof}
{ See Appendix~\ref{app:proof:contact-2}.}
\end{proof}
\subsubsection{Parametrization of the retirement boundary}
In this section, we parametrize the free boundary $\pat\{Y_0>\psi_0\}$ by proving the contact set is monotone in the negative $x$-direction. Moreover, to prepare for the subsequent proof of local Lipschitz regularity of $\pat\{Y_0>\psi_0\}$, we establish two key ingredients. First, we derive a uniform upper bound for $\pat_\tau Y_0$. Second, utilizing the continuity of $\pat_\tau Y_0$ \cite{Blanchet-et-al-2006}, we prove $\pat_{xx}Y_0 > \psi''_0$ at the boundary, which guarantees $\pat_x Y_0 > \psi'_0$ in the adjacent non-contact region.

\begin{lem}\label{Y_tau, Y_x}
    Let $Y_0$ be the solution to the problem \eqref{obs:Y_0}. Then,
    \[0\le \pat_\tau Y_0 \le \zeta\bar{L}
    \quad \text{a.e. in }\ \O_T.\]
\end{lem}

\begin{proof}
{ See Appendix~\ref{app:proof:Y-tau-Y-x}.}
\end{proof}
\begin{lem}\label{X^R parametrization}
     Let $Y_0$ be the solution to the problem \eqref{obs:Y_0}. Then, for each $(\tau_0,x_0)\in \O_T$ satisfying $Y_0(\tau_0,x_0)=\psi_0(x_0)$, we have
     \begin{equation}\label{mnt1}
        Y_0(\tau_0,x)=\psi_0(x)
     \quad \text{for all }\ x\le x_0. 
     \end{equation}
     Thus, we may represent the free boundary $\partial\{Y_0>\psi_0\}$ by a curve $\cX^R:(0,T)\to \mathbb{R}$, defined as
    \begin{equation}\label{mnt2}
        \cX^R(\tau):=\sup\{x:Y_0(\tau,x)=\psi_0(x)\}, \quad \tau\in(0,T).
    \end{equation}
     Furthermore, for each $\tau_0\in (0,T)$,
     there exists $\d_0>0$ such that
     \begin{equation}\label{mnt3}
         \pat_x Y_0(\tau_0,x)-\psi_0'(x)> 0
    \quad \text{for all }\ \cX^R(\tau_0)< x\le \cX^R(\tau_0)+\d_0.
     \end{equation}
\end{lem}

\begin{proof}
{ See Appendix~\ref{app:proof:X-R-parametrization}.}
\end{proof}
From the monotonicity of $\cX^R$,
its continuity follows from a standard argument as in \cite[Theorem 3.1]{F2} and the Hopf lemma.
In particular, continuity of $\cX^R$ implies that 
\[\displaystyle\bigcup_{\tau\in(0,T)}(\tau,\cX^R(\tau))
=\pat\{Y_0>\psi_0\}.\] In other words, $(\tau,\cX^R(\tau))$ constitutes the entire free boundary $\pat\{Y_0>\psi_0\}$. 
Furthermore, combining the monotonicity and continuity of $\cX^R$ with Corollary~\ref{X^R region}, we also deduce that $\cX^R$ is uniformly continuous on $[0,T]$.
Using this information, we proceed to improve the regularity of $\cX^R$ by proving its local Lipschitz property in the following subsubsection.

\subsubsection{Properties and regularity of the retirement boundary}

To establish the local Lipschitz continuity of the free boundary $\mathcal{X}^R$, we derive an estimate showing that $\partial_\tau Y_0$ and $\partial_x Y_0-\psi_0'$ are comparable near it. The strict positivity of $\partial_{xx}Y_0-\psi''_0$ provides a linear lower bound for $\partial_x Y_0-\psi_0'$. Using an auxiliary function and the comparison principle, we obtain a desired estimate and applying this estimate to level curves converging to $\cX^R$ proves the local Lipschitz continuity of $\cX^R$. Finally, the boundary Harnack inequality establishes smoothness.

\begin{lem}\label{ineq 2}
    Let $\cX^R:(0,T)\to \mathbb{R}$ be defined as in \eqref{mnt2}.
    Then for each small $\k>0$ and $\tau\in[2\k,T]$, there exists $\d>0$ and $C=C(\k,\d)>0$ such that 
    \begin{equation}\label{Y0:estimate}
        0 \le \pat_\tau Y_0(\tau,x) \le C\{\pat_x Y_0(\tau,x)-\psi_0'(x)\}
    \quad \text{for all }\ \cX^R(\tau)<x\le \cX^R(\tau)+\d.
    \end{equation}
\end{lem}

\begin{proof}
{ See Appendix~\ref{app:proof:ineq-2}.}
\end{proof}
\begin{thm}[Properties of the retirement boundary]
\label{X^R:properties}
    Let $\cX^R:(0,T)\to \mathbb{R}$ be defined as in \eqref{mnt2}.
    Then, 
    \begin{itemize}
        \item [(1)] $\cX^R$ is locally Lipschitz continuous,
        \item [(2)] $\cX^R$ is strictly decreasing,
        \item [(3)] $\displaystyle\lim_{\tau\to 0^+}\cX^R(\tau)=x^*$.
    \end{itemize}
\end{thm}

\begin{proof}
{ See Appendix~\ref{app:proof:X-R-properties}.}
\end{proof}
Furthermore, by utilizing the boundary Harnack inequality \cite{TorresLatorre2024}, \cite{Kukuljan2022} inductively, we may obtain the smoothness of the free boundary $\cX^R$.

\begin{cor}[Smoothness]\label{X^R:smooth}
    Let $\cX^R:(0,T)\to \mathbb{R}$ be defined as in \eqref{mnt2}.
    Then, $\cX^R\in C^{\infty}((0,T))$.
\end{cor}

\begin{proof}
{ See Appendix~\ref{app:proof:X-R-smooth}.}
\end{proof}
\subsubsection{Auxiliary obstacle problem and its contact region}

To investigate the region $\{\partial_x Y = 0\}$ and its boundary, we consider an auxiliary single obstacle problem which corresponds to the variational inequality that $Y$ satisfies in $\{Y>\psi\}$.
Our main strategy is to identify $\pat_x Y$ with the solution of such auxiliary single obstacle problem.
With such identity, analyzing the free boundary $\pat\{\pat_x Y<0\}$ reduces to investigating the free boundary of the corresponding single obstacle problem.
This allows us to adopt the classical methods to verify the properties of the free boundary as usual.
Our desired auxiliary obstacle problem would be as follows:
 \begin{equation}\label{obs_V}
        \begin{cases}
            \partial_\tau V(\tau,x) -\cL V(\tau,x)=f'(x)
            \quad \text{for }\ (\tau,x)\in\cC_T \ \text{ with }\ V(\tau,x)<0,
            \\
            \partial_\tau V(\tau,x) -\cL V(\tau,x)\le f'(x)
            \quad \text{for }\ (\tau,x)\in\cC_T \ \text{ with }\ V(\tau,x)=0,
            \\
            V(\tau,\cX^R(\tau))=\psi'(\cX^R(\tau))
            \quad \text{for \ all }\ \tau\in(0,T),
            \\
            V(0,x)=\psi'(x) \quad \text{for }\ x\in (\cX^R(0),\infty)
        \end{cases}
    \end{equation}
where $\cC_T:=\{Y>\psi\}$.
Since we will subsequently show that $V=\pat_x Y$ holds in $\cC_T$,
the results obtained by investigating the free boundary $\pat\{V<0\}$ can be viewed as the results for $\pat\{\pat_x Y<0\}$.
For convenience, we consider the following transformed problem: 
\begin{equation}\label{obs:V_0}
        \begin{cases}
            \partial_\tau V_0(\tau,x) -\widetilde{\cL}V_0(\tau,x)=g(x)
            \quad \text{for }\ (\tau,x)\in\cC_T \ \text{ with }\ V_0(\tau,x)<0,
            \\
            \partial_\tau V_0(\tau,x) -\widetilde{\cL}V_0(\tau,x)\le g(x)
            \quad \text{for }\ (\tau,x)\in\cC_T \ \text{ with }\ V_0(\tau,x)=0,
            \\
            V_0(\tau,\cX^R(\tau))=\phi(\cX^R(\tau))
            \quad \text{for \ all }\ \tau\in(0,T),
            \\
            V_0(0,x)=\phi(x) \quad \text{for }\ x\in (\cX^R(0),\infty),
        \end{cases}
    \end{equation}
where $\widetilde{\cL}:=\frac{\t^2}{2}\pat_{xx}+(\b-r+\frac{\t^2}{2})\pat_x-r$,
\[V_0(\tau,x)=e^{-x}V(\tau,x),\quad  g(x)=e^{-x}f'(x), \quad \phi(x)=e^{-x}\psi'(x)
\quad \text{for }\ \tau\in (0,T),\ x\in \mathbb{R}.\]
This transformation terminates the exponential term in $f'$, thereby simplifying the subsequent analysis.
Moreover, it is reasonable to focus on $V_0$, as the free boundary $\pat\{V_0<0\}$ coincides with $\pat\{V<0\}$.

To obtain the existence and uniqueness of $V_0$, we introduce a penalized problem on the bounded domain $\cC_T^n=\{(\tau,x)\in\cC_T:\cX^R(\tau)\le x \le n\}$:
\begin{equation}\label{obsV:pen}
    \begin{cases}
        \pat_\tau V^0_{n,\e} -\widetilde{\cL} V^0_{n,\e}
        =g+\b_{2,\e}(V^0_{n,\e})
        \quad \text{in }\ \cC_T^n,
        \\
        V^0_{n,\e}(\tau,\cX^R(\tau))= \phi(\cX^R(\tau))
        \quad \text{and} \quad 
        V^0_{n,\e}(\tau,n)=\r(\tau)
        \quad \text{for }\ \tau\in(0,T),
        \\
        V^0_{n,\e}(0,x)=\phi(x)
        \quad \text{for }\  x\in(\cX^R(0),n),
    \end{cases}
\end{equation}
where $\r(\tau)=\min\{2\zeta\bar{L}\tau+\phi(n),0\}$ is imposed for compatibility condition. The penalty function $\b_{2,\e}\in C^\infty(\mathbb{R})$ satisfies $\b_{2,\e}\le 0$, $\b'_{2,\e}\le 0$, with $\b_{2,\e}(s)=0$ for $s\le -\e$ and $\b_{2,\e}(0)=-2\zeta\bar{L}$.

Since the left lateral boundary is smooth, the Schauder fixed-point theorem \cite[Theorem 8.3]{Lieberman} guarantees the existence of $V^0_{n,\e}$. The structure of $\b_{2,\e}$ yields uniform $W^{1,2}_p(\cC_T^n)$ bounds for $1<p<\infty$, allowing us to extract a weak limit $V^0_n\in W^{1,2}_p(\cC_T^n)\cap C(\overline{\cC_T^n})$ as $\e \to 0^+$. This limit $V^0_n$ satisfies the corresponding obstacle problem on $\cC_T^n$ with the same boundary conditions.

Next, applying the interior $W^{1,2}_p$-estimate \cite[Theorem 7.30]{Lieberman} on $\cC_T^R:=\{(\tau,x)\in\cC_T:x<R\}$, we deduce that for any $n>2R$, there exists a constant $C>0$ such that
\[\Vert V^0_n\Vert_{W^{1,2}_p(\cC_T^R)}\le C\left(\Vert V^0_n\Vert_{L^p(\cC_T^{2R})}+\Vert g \Vert_{L^p(\cC_T^{2R})}+\Vert \phi \Vert_{W^{1,2}_p(\cC_T^{2R})}\right).\]
By the comparison principle \cite[Theorem 1.1]{YangYan2008}, $V^0_{n,\e}$ (and thus $V^0_n$) is uniformly bounded from below by an $n$-independent auxiliary function, making the right-hand side above independent of $n$. We can therefore extract a weak limit $V_{\rm lim} \in W^{1,2}_p$ as $n\to\infty$ on compact subsets of $\cC_T$, which satisfies \eqref{obs:V_0}. The comparison principle on the unbounded domain \cite[Theorem 1.2]{YangYan2008} yields $V_{\rm lim}=V_0$. Finally, the Sobolev embedding theorem ensures the local uniform convergence $V^0_n \to V_0$ and $\pat_x V^0_n \to \pat_x V_0$, guaranteeing that $V_0$ inherits the boundary conditions properly.

In the following lemma, we aim to verify the region in which the contact between $V$ and its obstacle $0$ occurs.

\begin{lem}\label{contact:3}
    Let $V$ be the solution to \eqref{obs_V}. Then, for each $\e>0$, there exists $c_\e>0$ such that
    \[V=0 \quad \text{in }\ [\e,T]\times[c_\e,\infty).\]
\end{lem}

\begin{proof}
{ See Appendix~\ref{app:proof:contact-3}.}
\end{proof}

\subsubsection{Parametrization of the wealth-binding boundary}
Since $g$ is strictly increasing, we note from Lemma~\ref{contact:3} that
if $V_0(\tau_0,x_0)=0$ for some $(\tau_0,x_0)\in \cC_T$,
then $V_0(\tau_0,x)=0$ for all $x\ge x_0$.
This can be obtained by following the same procedure as in Lemma~\ref{X^R parametrization}.
Thus, it is natural to consider the following parametrization:
\begin{equation}\label{X^W}
    \cX^B(\tau):=\inf\{x:V_0(\tau,x)=0\}
\quad \text{for \ all }\ \tau\in(0,T).  
\end{equation}
Note that the set $\{x<\cX^B(\tau)\}=\{V_0<0\}$ is open and hence using the parametrization $\cX^B$,
we aim to verify that the solution $V$ to the auxiliary problem \eqref{obs_V} coincides with the spatial derivative of the solution to \eqref{obs1}.
That is, $V=\partial_x Y$ in $\cC_T$
and this justifies that $\cX^B$ is a suitable parametrization of the free boundary $\pat\{\pat_x Y<0\}$.

\begin{lem}\label{V,dY}
    Let $Y$ and $V$ be the solution to \eqref{obs1} and \eqref{obs_V} respectively.
    Then, $\pat_x Y = V$ in $\cC_T$.
\end{lem}

\begin{proof}
{ See Appendix~\ref{app:proof:V-dY}.}
\end{proof}
Note that $V\in W^{1,2}_{p,\loc}(\cC_T)$ and thus $\partial_{xx}Y=\partial_x V\in C^{\frac{\a}{2},\a}(\cC_T)$ for some $\a\in(0,1)$, induced by the embedding theorem.
Moreover, the continuity of $\pat_\tau Y$ on $\pat\{Y>\psi\}$ induced from the result that $\pat_\tau Y\ge 0$ almost everywhere (see \cite{Blanchet-et-al-2006}) yields the following improvement in regularity of $Y$.
\begin{cor}
    Let $Y\in W^{1,2}_{p,\loc}(\O_T)\cap C(\overline{\O_T})$ be the solution to \eqref{obs1}.
    Then, $Y\in C^{1,1}(\Omega_T)\cap C^{1,2}(\cC_T)$.
\end{cor}
Utilizing the identity $V = \pat_x Y$ in $\cC_T$, 
we next investigate the increasing property of $V_0$ with respect to $\tau$ and $x$ on the parabolic boundary of $\cC_T$ and hence obtain the monotonicity of $\cX^B$.
\begin{lem}\label{Vd:dtau,dx}
    Let $V_0$ be the solution to \eqref{obs:V_0}. Then, we have
    $\pat_\tau V_0 \ge 0$ a.e. 
    and 
    $\pat_x V_0 \ge 0$ in $\cC_T$.
    Consequently, $V_0 \ge \phi$ holds globally in $\cC_T$.
\end{lem}

\begin{proof}
{ See Appendix~\ref{app:proof:Vd-dtau-dx}.}
\end{proof}
\subsubsection{Properties and regularity of the wealth-binding boundary}
In this section, we begin by obtaining the local Lipschitz regularity of the wealth-binding boundary.
To this end, it suffices to show that $\partial_\tau V_0$ and $\partial_x V_0$ remain comparable as they reach the boundary. It is important to note that the standard regularity result in \cite{Blanchet-et-al-2006} cannot be directly applied to the solution $V_0$ of our upper obstacle problem. To overcome this difficulty, we establish the desired comparability estimate through an alternative approach. We achieve this by deriving the corresponding bound for the solution of an approximating penalized problem, and then pass to the limit for our original solution. In this procedure, suitable barrier functions are introduced to fulfill the conditions necessary for the comparison principle.

\begin{lem}\label{V_0:dtau,dx}
    Let $V_0$ be the solution to \eqref{obs:V_0}. Then, there exists some small constant $\k_0>0$ such that for all $\k\in(0,\k_0)$, there exists a constant $C=C(\k)>0$ such that
    \begin{equation}\label{ineq(3)}
        0\le \tau\pat_\tau V_0(\tau,x)\le Ce^{\frac{1}{\g}x}\pat_x V_0(\tau,x)
        \quad \text{for all }\ (\tau,x)\in \{V_0<0\}\cap\cC_T \text{ satisfying }\ 2\k\le\tau\le T.
    \end{equation}
\end{lem}

\begin{proof}
{ See Appendix~\ref{app:proof:V-0-dtau-dx}.}
\end{proof}
Using the estimate \eqref{ineq(3)} and proceeding as in the proof of Theorem~\ref{X^R:properties}, we obtain the following result. 
\begin{thm}[Properties of the Wealth-binding boundary]
\label{X^W:properties}
    Let $\cX^B:(0,T)\to \mathbb{R}$ be defined as in \eqref{X^W}.
    Then, 
    \begin{itemize}
        \item [(1)] $\cX^B$ is locally Lipschitz continuous.
        Thus, $\cX^B$ constitutes the entire free boundary $\pat\{V_0<0\}$.
        \item [(2)] $\cX^B$ is strictly decreasing.
        \item [(3)] $\displaystyle\lim_{\tau\to 0^+}\cX^B(\tau)=\infty$.
    \end{itemize}
\end{thm}

\begin{proof}
{ See Appendix~\ref{app:proof:X-W-properties}.}
\end{proof}
To establish the smoothness of $\cX^B$, we inductively apply the boundary Harnack inequality as in Corollary~\ref{X^R:smooth}. Note, however that the right-hand side $g$ is non-differentiable at $x=\tilde{x}$. Since $\cX^B$ is strictly decreasing, it can intersect this line at most once, say at a unique $\tilde{\tau} \in (0,T)$. If the free boundary passes through this point, its infinite differentiability may fail, which yields the following restrictive result.

\begin{cor}[Smoothness]\label{XW smooth}
    Let $\cX^B:(0,T)\to \mathbb{R}$ be defined as in \eqref{X^W}.
    If $\cX^B$ does not intersect $x=\tilde{x}$, then $\cX^B \in C^\infty((0,T))$.
    Otherwise, if $\cX^B$ intersects $x=\tilde{x}$, then
    $\cX^B \in C^{0,1}_{\rm loc}((0,T)) \cap C^{\infty}((0,T)\setminus\{\tilde{\tau}\})$,
    where $\tilde{\tau}$ is the unique point such that
    $\cX^B(\tilde{\tau})=\tilde{x}$.
\end{cor}

\begin{rem}
The contact set $\{V_0=0\}$ requires $g(x) \ge 0$. Since $g$ is strictly increasing and $g(\tilde{x}) = \zeta \left( \bar{L} - \frac{L}{1-\a} \right)$, the interaction between the free boundary $\cX^B$ and $\tilde{x}$ depends on $\a$. If $\a \ge 1 - L/\bar{L}$, then $g(\tilde{x}) \le 0$, which implies that the contact occurs only for $x \ge \tilde{x}$. Thus, $\cX^B$ does not intersect $x=\tilde{x}$ and remains globally smooth on $(0,T)$. Conversely, if $\a < 1 - L/\bar{L}$, $\cX^B$ may cross $\tilde{x}$, reducing its regularity to $C^{0,1}$ exactly at the crossing point.
\end{rem}

In the following theorem, relying on the structure of the free boundary $\cX^B$, we establish the uniqueness of the strong solution $Y$. We prove that $Y$ coincides with any strong solution whose gradient-constrained region is similarly bounded from the left by a monotone decreasing curve. This specific structure is essential: it allows the maximum to propagate horizontally to the free boundary, while the monotonicity of the free boundary guarantees the interior cylinder condition required for the parabolic Hopf lemma.

\subsection{Uniqueness of the solution \texorpdfstring{$Y$}{Y}}
\begin{thm}\label{thm:maximal-strong-solution}

    Let $Y\in W^{1,2}_{p,{\rm loc}}(\O_T)\cap C(\overline{\Omega_T})$ be the function obtained in Lemma~\ref{Y}, satisfying \eqref{obs1}. Then, $Y$ is the maximal strong solution of \eqref{obs1}. Furthermore, $Y$ coincides with any strong solution of \eqref{obs1} whose gradient-constrained region is of the form $\{(\tau, x): x \ge \cX(\tau)\}$ for some monotone decreasing free boundary curve $\cX$.
\end{thm}

\begin{proof}
{ See Appendix~\ref{app:proof:thm-maximal-strong-solution}.}
\end{proof}
\subsubsection{Strict convexity of the solution \texorpdfstring{$Q$}{Q}}
For the duality theorem, let us establish the strict convexity of $Q$ in $\mathcal{D}_T\setminus \{\partial_z Q =0\}$.
\begin{lem}[Strict convexity]\label{lem:strict-convexity-Q}

    Let $Q\in W^{1,2}_{p,{\rm loc}}(\mathcal{D}_T)\cap C(\overline{\cD_T})$, $1<p<\infty$ be the solution to \eqref{obs:Q}.
    Then, $\partial_{zz}Q>0$ a.e. in $\{\partial_z Q <0\}$. Consequently, $Q$ is strictly convex with respect to $z$ in $\{\partial_z Q <0\}$.
\end{lem}

\begin{proof}
{ See Appendix~\ref{app:proof:lem-strict-convexity-Q}.}
\end{proof}
\subsection{Characterization of \(\mathcal{Q}\) and Its Free Boundaries}

We now translate the results obtained for the transformed variational inequality in the \((\tau,x)\)-variables back to the original formulation in the \((t,z)\)-variables. Since the change of variables
\[
\tau=T-t,\qquad z=e^x,\qquad Y(\tau,x)=\mathcal{Q}(T-\tau,e^x)
\]
is one-to-one, the regularity and free-boundary properties established for \(Y\) immediately yield the corresponding structural properties of \(\mathcal{Q}\). In particular, the retirement region, the working region, and the binding region can be characterized in terms of two free boundaries in the original variables.

\begin{thm}[Structure of the dual variational inequality]\label{thm:Q-structure}
Let \(\mathcal Q\) be the unique solution to Problem~\eqref{eq:main3}. Then there exist two functions
\[
z_R,z_B:(0,T)\to(0,\infty),
\qquad
0<z_R(t)<z_B(t)\quad \text{for all }t\in(0,T),
\]
such that the domain
\[
\mathcal D_T =\{(t,z)\mid 0\le t<T,\; z>0\}
\]
is decomposed into the following three regions:
\begin{align*}
\mathcal{RR}_{\rm dual}
&:=
\{(t,z)\in\mathcal D_T: 0<z\le z_R(t)\},
\\
\mathcal{WR}_{\rm dual}
&:=
\{(t,z)\in\mathcal D_T: z_R(t)<z<z_B(t)\},
\\
\mathcal{BR}_{\rm dual}
&:=
\{(t,z)\in\mathcal D_T: z\ge z_B(t)\}.
\end{align*}
More precisely, the following assertions hold.

\begin{enumerate}
\item[\rm (i)]
\textbf{Characterization of the three regions.}
\begin{align*}
\mathcal{RR}_{\rm dual}
&=
\{(t,z)\in\mathcal D_T:\mathcal Q(t,z)=J_R(z)\},
\\
\mathcal{WR}_{\rm dual}
&=
\{(t,z)\in\mathcal D_T:\mathcal Q(t,z)>J_R(z),\ \partial_z\mathcal Q(t,z)<0\},
\\
\mathcal{BR}_{\rm dual}
&=
\{(t,z)\in\mathcal D_T:\mathcal Q(t,z)>J_R(z),\ \partial_z\mathcal Q(t,z)=0\}.
\end{align*}
Hence,
\[
\mathcal D_T
=
\mathcal{RR}_{\rm dual}\cup \mathcal{WR}_{\rm dual}\cup \mathcal{BR}_{\rm dual},
\]
with the union being disjoint up to the free boundaries.

\item[\rm (ii)]
\textbf{Free boundaries.}
The two free boundaries are given by
\[
\partial\mathcal{RR}_{\rm dual}
=
\{(t,z_R(t)):0<t<T\},
\qquad
\partial\mathcal{BR}_{\rm dual}
=
\{(t,z_B(t)):0<t<T\}.
\]
Equivalently, if \(\mathcal X^R\) and \(\mathcal X^B\) denote the free boundaries in the transformed \((\tau,x)\)-variables with \(\tau=T-t\) and \(z=e^x\), then
\[
z_R(t)=\exp\!\big(\mathcal X^R(T-t)\big),
\qquad
z_B(t)=\exp\!\big(\mathcal X^B(T-t)\big).
\]

\item[\rm (iii)]
\textbf{Convexity and derivative bounds.}
For each \(t\in[0,T)\), the map \(z\mapsto \mathcal Q(t,z)\) is convex on \((0,\infty)\) and is strictly convex on \((0,z_B(t))\).
Moreover,
\[
J_R'(z)\le \partial_z \mathcal Q(t,z)\le 0,
\qquad (t,z)\in \mathcal D_T,
\]
or equivalently,
\[
0\le -\partial_z\mathcal Q(t,z)\le -J_R'(z),
\qquad (t,z)\in \mathcal D_T.
\]
In particular,
\[
\partial_z\mathcal Q(t,z)=0
\quad \text{for } z\ge z_B(t),
\]
and
\[
\partial_z\mathcal Q(t,z)<0
\quad \text{for } 0<z<z_B(t).
\]

\item[\rm (iv)]
\textbf{Regularity and monotonicity of the free boundaries.}
Both \(z_R\) and \(z_B\) are continuous on \((0,T)\) and strictly increasing in \(t\).
More precisely, \(z_R\) is \(C^\infty\) on \((0,T)\).
For \(z_B\), let \(\tilde z:=e^{\tilde x}\). If \(z_B\) does not intersect \(\tilde z\), then \(z_B\in C^\infty((0,T))\). Otherwise, there exists a unique \(\tilde t\in(0,T)\) such that \(z_B(\tilde t)=\tilde z\), and
\[
z_B\in C^{0,1}_{\rm loc}((0,T))\cap C^\infty((0,T)\setminus\{\tilde t\}).
\]

\item[\rm (v)]
\textbf{Terminal-time limits.}
As \(t\to T-\),
\[
z_R(t)\longrightarrow z_R^\ast,
\qquad
z_B(t)\longrightarrow \infty,
\]
where
\[
z_R^\ast:=e^{x^\ast}
=
\left(
\frac{\gamma_1}{1-\gamma_1}
\frac{\bar L^{\frac{\gamma_1-\gamma}{\gamma_1}}-L^{\frac{\gamma_1-\gamma}{\gamma_1}}}
{\zeta(\bar L-L)}
\right)^{\gamma_1}.
\]
\end{enumerate}
\end{thm}

{Theorem~\ref{thm:Q-structure} should be read together with Theorem~\ref{thm:maximal-strong-solution}: the solution \(\mathcal Q\) is obtained from the maximal strong solution of the transformed variational inequality, and this solution is unique among strong solutions whose gradient-constrained region is generated by a monotone binding boundary. In the original \((t,z)\)-variables this means that the binding boundary \(z_B(t)\), associated with the gradient constraint \(\partial_z\mathcal Q=0\), is monotone increasing in time. This boundary is not imposed exogenously; it is generated by the maximal strong solution and is later verified to implement the dual game.}

Theorem~\ref{thm:Q-structure} shows that the dual domain \(\mathcal D_T\) is partitioned by the two free boundaries \(z_R(t)\) and \(z_B(t)\) into three qualitatively different regions. Figure~\ref{fig:dual_free_boundaries} visualizes this decomposition in the original \((t,z)\)-variables. The lower region,
\(\mathcal{RR}_{\rm dual}\), is the retirement region, where the solution coincides with the retirement payoff, that is,
\(\mathcal Q(t,z)=J_R(z)\). The middle region,
\(\mathcal{WR}_{\rm dual}\), is the working region, where
\(\mathcal Q(t,z)>J_R(z)\) and \(\partial_z\mathcal Q(t,z)<0\). The upper region,
\(\mathcal{BR}_{\rm dual}\), is the binding region, where
\(\mathcal Q(t,z)>J_R(z)\) and \(\partial_z\mathcal Q(t,z)=0\).
Accordingly, the lower free boundary \(z_R(t)\) separates the retirement and working regions, whereas the upper free boundary \(z_B(t)\) separates the working and binding regions.

The figure also illustrates the monotonicity properties stated in Theorem~\ref{thm:Q-structure}. Both boundaries are strictly increasing in time, so the retirement and binding thresholds move upward as \(t\) approaches the terminal date \(T\). In particular, the numerical plot shows that \(z_R(t)\) increases relatively slowly, while \(z_B(t)\) rises much more rapidly near maturity. This is consistent with the terminal-time behavior in Theorem~\ref{thm:Q-structure}(v), namely,
\[
z_R(t)\to z_R^\ast
\qquad\text{and}\qquad
z_B(t)\to\infty
\quad\text{as } t\to T-.
\]
Thus, near the terminal time, the retirement boundary remains bounded, while the binding boundary is pushed upward without bound.

From the viewpoint of the geometry of the solution, the figure should be read together with the derivative and convexity properties in Theorem~\ref{thm:Q-structure}(iii). For each fixed \(t\), the interval \(0<z<z_B(t)\) corresponds to the range in which the solution remains strictly decreasing and, in fact, strictly convex in \(z\), while the region \(z\ge z_B(t)\) corresponds to the flat part where the gradient constraint binds and \(\partial_z\mathcal Q(t,z)=0\). Hence, the two free boundaries in Figure~\ref{fig:dual_free_boundaries} do not merely separate colors in the diagram; they encode the transitions between the three distinct regimes of the dual variational inequality.

\begin{figure}[!htb]
    \centering
    \includegraphics[width=0.65\textwidth]{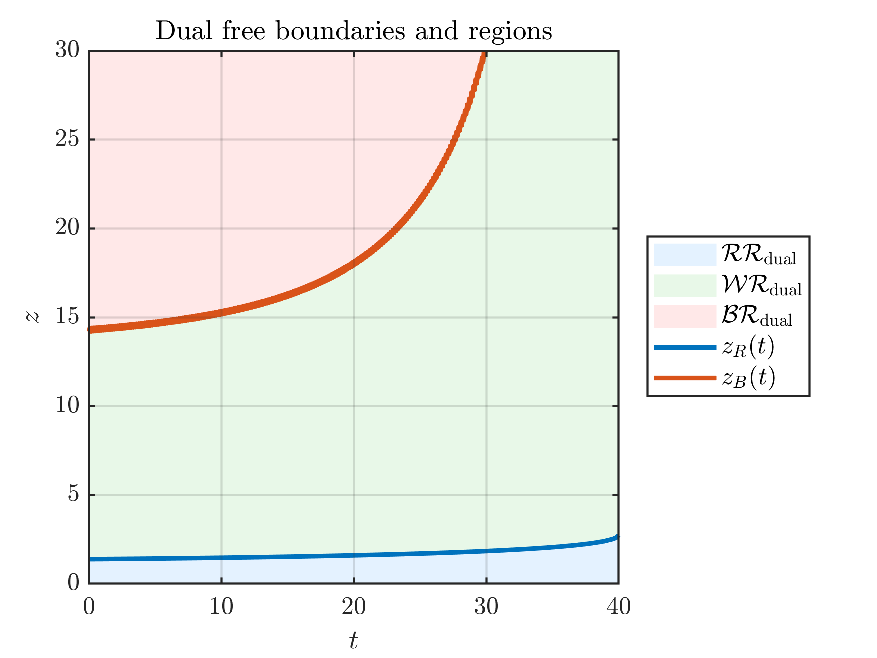}
    \caption{Dual free boundaries \(z_R(t)\) and \(z_B(t)\), and the regions they induce in the dual domain.}
    \label{fig:dual_free_boundaries}
\end{figure}

\section{Verification and Optimal Strategies}

{In this section, we verify the bridge from analysis back to control. The maximal strong solution \(\mathcal Q\) of the dual variational inequality, characterized in particular by the gradient-constrained binding boundary \(z_B(t)\) that is monotone increasing in time, is shown to coincide with the value of the stopper--singular-controller game. This identification then yields the primal-dual relation and the optimal primal strategy.}

For a given initial state \((t,z)\in\mathcal D_T\), let \(z_R\) and \(z_B\) denote the retirement and binding free boundaries obtained in the previous section. For each \(D\in\mathcal{NI}(t,T)\), define the retirement hitting time by
\begin{equation}\label{eq:tauD-def}
\tau_t^{D}(z)
:=
\inf\{s\in[t,T)\mid \mathcal Z_s^{t,z,D}\le z_R(s)\}\wedge T.
\end{equation}
Next, define the barrier strategy \(D^B(t,z)=\{D_s^B(t,z)\}_{s\in[t,T]}\) by
\begin{equation}\label{eq:DB-dual}
D_s^B(t,z)
:=
\min\left\{
1,\,
\inf_{t\le \eta\le s}\frac{z_B(\eta)}{\Xi_\eta^{t,z}}
\right\},
\qquad s\in[t,T],
\end{equation}
where
\begin{equation}\label{eq:Xi-uncontrolled}
\Xi_s^{t,z}
=
z e^{\beta(s-t)}\frac{H_s}{H_t},
\qquad s\ge t.
\end{equation}
For notational convenience, we write
\begin{equation}\label{eq:tauR-def}
\tau_R(t,z):=\tau_t^{D^B(t,z)}(z).
\end{equation}
Since \(z_B\) is continuous and strictly positive, it follows that
$$
D^B(t,z)\in\mathcal{NI}(t,T),
\qquad
\tau_R(t,z)\in\mathcal S(t,T).
$$

The next theorem shows that the pair \(\bigl(D^B(t,z),\tau_R(t,z)\bigr)\) is optimal for the min--max problem and that the game value coincides with \(\mathcal Q\).

\begin{thm}\label{thm:min-max}
For every \((t,z)\in\mathcal D_T\), the pair
$$
\bigl(D^B(t,z),\tau_R(t,z)\bigr)\in\mathcal{NI}(t,T)\times\mathcal S(t,T)
$$
solves the min--max problem
$$
\overline J(t,z)
=
\inf_{D\in\mathcal{NI}(t,T)}
\sup_{\tau\in\mathcal S(t,T)}
\mathcal J(t,z;D,\tau)
=
\mathcal J\bigl(t,z;D^B(t,z),\tau_R(t,z)\bigr).
$$
Moreover,
$$
\overline J(t,z)=\mathcal Q(t,z),
\qquad (t,z)\in\mathcal D_T.
$$
Consequently, the stopping--control game admits a value and
$$
J(t,z)=\underline J(t,z)=\overline J(t,z)=\mathcal Q(t,z),
\qquad (t,z)\in\mathcal D_T.
$$
\end{thm}

\begin{proof}
{ See Appendix~\ref{app:proof:thm-min-max}.}
\end{proof}
Having identified the value of the stopping--control game with \(\mathcal Q\), we now return to the primal problem. The next theorem establishes the duality relation between the primal and dual value functions and recovers the optimal primal strategy from the optimal dual state process.

\begin{thm}[Duality theorem and optimal strategies]\label{thm:main}
Fix \(t\in[0,T]\) and \(w\ge 0\).

\begin{itemize}
\item[(a)]
The primal value function \(V(t,w)\) and the dual value function \(J(t,\xi)\) satisfy
\begin{equation}\label{eq:duality}
V(t,w)=\inf_{\xi>0}\bigl(J(t,\xi)+\xi w\bigr)
=\inf_{\xi>0}\bigl(\mathcal Q(t,\xi)+\xi w\bigr).
\end{equation}
Moreover, there exists a unique \(\xi^*=\xi^*(t,w)\in(0,z_B(t)]\) such that
\begin{equation}\label{eq:xi-star}
w=-\partial_zJ(t,\xi^*)=-\partial_z\mathcal Q(t,\xi^*).
\end{equation}

\item[(b)]
For \(s\in[t,T]\), define
$$
D_s^*
:=
D_s^B(t,\xi^*)
=
\min\left\{
1,\,
\inf_{t\le \eta\le s}\frac{z_B(\eta)}{\Xi_\eta^{t,\xi^*}}
\right\},
\qquad
\mathcal Z_s^*
:=
\mathcal Z_s^{t,\xi^*,D^*}
=
\Xi_s^{t,\xi^*}D_s^*,
$$
and let
\begin{equation}\label{eq:tau-star}
\tau^*
:=
\tau_R(t,\xi^*)
=
\inf\{s\in[t,T)\mid \mathcal Z_s^*\le z_R(s)\}\wedge T.
\end{equation}
Then the optimal pre-retirement controls are given by
\begin{equation}\label{eq:optimal-controls-pre}
c_s^*=\hat c(\mathcal Z_s^*),
\qquad
l_s^*=\hat l(\mathcal Z_s^*),
\qquad
W_s^*=-\partial_z\mathcal Q(s,\mathcal Z_s^*),
\qquad
\pi_s^*=\frac{\theta}{\sigma}\mathcal Z_s^*\,\partial_{zz}\mathcal Q(s,\mathcal Z_s^*),
\quad s\in[t,\tau^*),
\end{equation}
where \(W^*\) denotes the associated wealth process. At retirement,
\begin{equation}\label{eq:optimal-terminal-wealth}
W_{\tau^*}^*
=
-J_R'(\mathcal Z_{\tau^*-}^*)
=
\frac{\bar L^{\frac{\gamma_1-\gamma}{\gamma_1}}}{K_1}
(\mathcal Z_{\tau^*-}^*)^{-1/\gamma_1}.
\end{equation}
After retirement, the optimal continuation controls are
\begin{equation}\label{eq:optimal-controls-post}
c_s^*=K_1W_s^*,
\qquad
l_s^*=\bar L,
\qquad
\pi_s^*=\frac{\theta}{\sigma\gamma_1}W_s^*,
\quad s\ge \tau^*.
\end{equation}
Consequently, the strategy \((c^*,l^*,\pi^*,\tau^*)\in\mathcal A(t,w)\) is optimal for the primal problem.
\end{itemize}
\end{thm}

\begin{proof}
{ See Appendix~\ref{app:proof:thm-main}.}
\end{proof}
As an immediate consequence of Theorem~\ref{thm:main}, the optimal retirement rule can be reformulated directly in the primal \((t,w)\)-variables.

\begin{cor}[Optimal retirement threshold in the primal domain]\label{cor:primal-retirement-threshold}
For each \(t\in[0,T)\), define the wealth-domain retirement threshold by
\begin{equation}\label{eq:WR-def}
\mathcal W_R(t)
:=
-\partial_z\mathcal Q\bigl(t,z_R(t)\bigr)
=
-J_R'\bigl(z_R(t)\bigr)
=
\frac{\bar L^{\frac{\gamma_1-\gamma}{\gamma_1}}}{K_1}\,z_R(t)^{-1/\gamma_1}.
\end{equation}
Then \(\mathcal W_R\) is continuous on \([0,T)\), and the primal retirement and working regions are given by
\[
\mathcal{RR}_{\rm primal}
:=
\{(t,w)\in[0,T)\times[0,\infty)\mid w\ge \mathcal W_R(t)\},
\]
and
\[
\mathcal{WR}_{\rm primal}
:=
\{(t,w)\in[0,T)\times(0,\infty)\mid 0<w<\mathcal W_R(t)\}.
\]
Moreover, under the optimal strategy of Theorem~\ref{thm:main}, the optimal retirement time admits the equivalent representation
\begin{equation}\label{eq:tau-star-primal}
\tau^*
=
\inf\{s\in[t,T)\mid W_s^*\ge \mathcal W_R(s)\}\wedge T.
\end{equation}
\end{cor}

\begin{proof}
{ See Appendix~\ref{app:proof:cor-primal-retirement-threshold}.}
\end{proof}
\begin{figure}[!htb]
    \centering
    \includegraphics[width=0.8\textwidth]{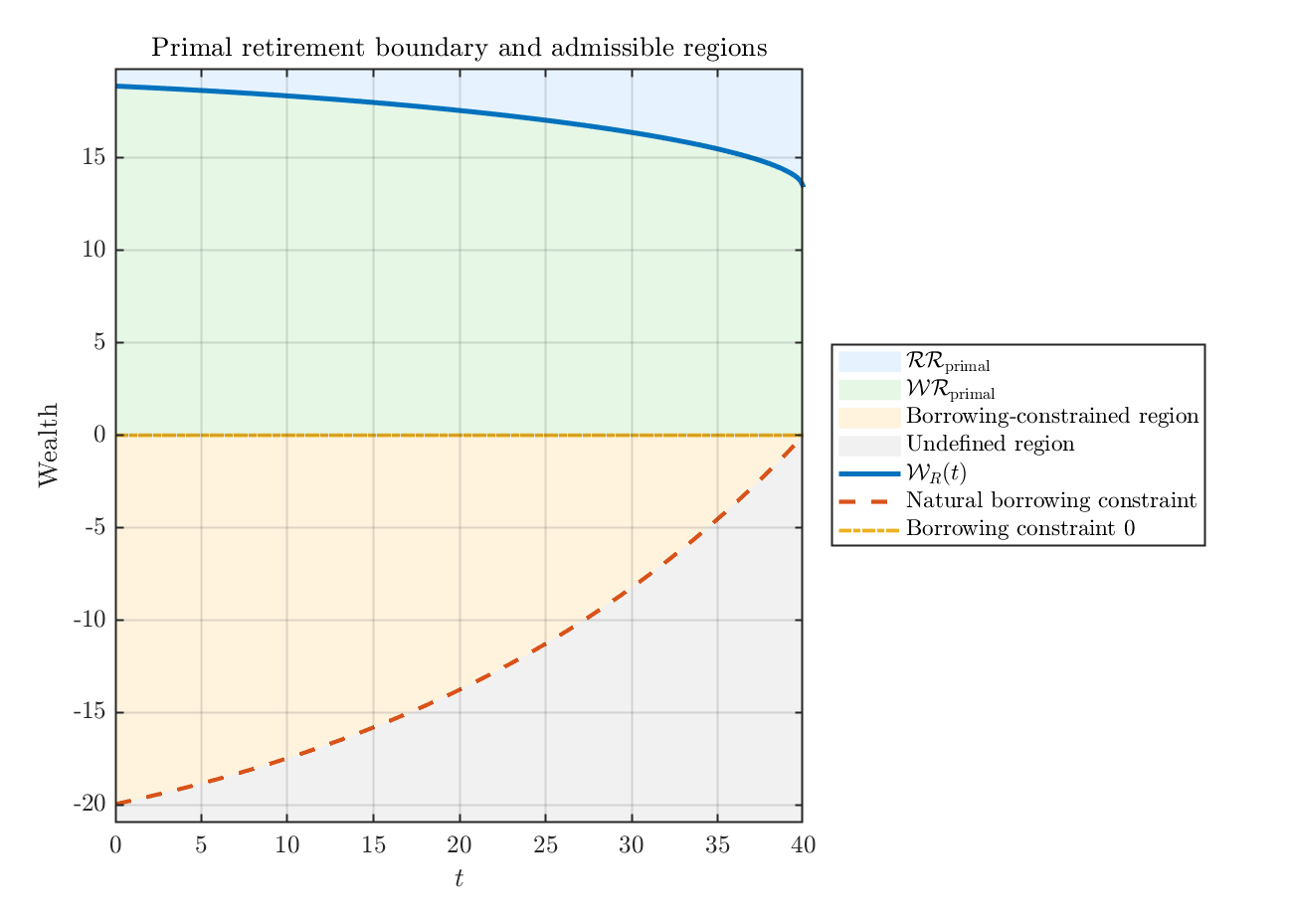}
    \caption{Primal wealth domain, the optimal retirement threshold \(\mathcal W_R(t)\), and the regions induced by the borrowing constraints.}
    \label{fig:wealth-domain}
\end{figure}

Figure~\ref{fig:wealth-domain} illustrates the geometry of the primal wealth domain implied by Corollary~\ref{cor:primal-retirement-threshold}. The solid curve represents the optimal retirement threshold \(\mathcal W_R(t)\). As established in \eqref{eq:tau-star-primal}, retirement is optimally triggered when the wealth process first reaches this boundary from below. Hence, the region above \(\mathcal W_R(t)\) corresponds to the primal retirement region \(\mathcal{RR}_{\rm primal}\), while the strip between \(0\) and \(\mathcal W_R(t)\) corresponds to the working region \(\mathcal{WR}_{\rm primal}\).

The figure also compares the imposed borrowing constraint \(w\ge 0\) with the natural borrowing constraint
\[
-\frac{1-e^{-r(T-t)}}{r}\bar L.
\]
Since the imposed borrowing constraint is stricter than the natural one, the interval between the natural borrowing constraint and \(0\) is excluded from the admissible wealth domain, even though it would be feasible under the natural debt limit alone. For this reason, we label this strip as the borrowing-constrained region. The area below the natural borrowing constraint is marked as undefined, since wealth levels in that region are not economically meaningful for the optimization problem.

The figure therefore provides a direct visualization of the primal state-space decomposition: the agent continues working while wealth remains in \(\mathcal{WR}_{\rm primal}\), retires upon reaching \(\mathcal W_R(t)\), and is prevented by the borrowing constraint from entering the region below zero wealth.

\section{Numerical Results}\label{sec:numerical}

We now present numerical results for the model. The computations are based on the numerical implementation developed in \citet{JKY_int}. Since the purpose of the present paper is not to develop a new numerical method, we use this implementation as a computational tool and focus on the economic implications of labor supply flexibility under borrowing constraints.

All numerical results below report the computed free boundaries and the corresponding optimal consumption, leisure, portfolio, and retirement policies. In particular, we illustrate the shapes of the two free boundaries and examine their dependence on the model parameters.

\subsection{Implications}

We now examine the implications of the model under a benchmark parameter specification. Unless otherwise stated, the baseline parameters are given by
\[
T=40,\qquad
\beta=0.04,\qquad
r=0.04,\qquad
\mu=0.08,\qquad
\sigma=0.2,\qquad
\gamma=3,
\]
\[
\alpha=0.5,\qquad
\zeta=1,\qquad
L=0.5,\qquad
\bar L=1.
\]

Figure~\ref{fig:alpha-effect} reports the comparative statics with respect to the consumption weight parameter \(\alpha\) in the working region. Since a larger \(\alpha\) places relatively more weight on consumption and less weight on leisure in the Cobb--Douglas preference specification, it is natural to expect a reallocation from leisure toward current consumption. Consistent with this intuition, the figure shows that, for a given wealth level, a higher \(\alpha\) is associated with higher consumption and lower leisure. In the benchmark calibration, it is also associated with a smaller risky position. Taken together, these patterns suggest that a stronger preference for consumption tends to make the agent consume more, enjoy less leisure while working, and hold a more conservative risky position.

\begin{figure}[ht]
	\centering
	\subfigure[Consumption as a function of wealth for different values of \(\alpha\).]{
		\includegraphics[scale=0.4]{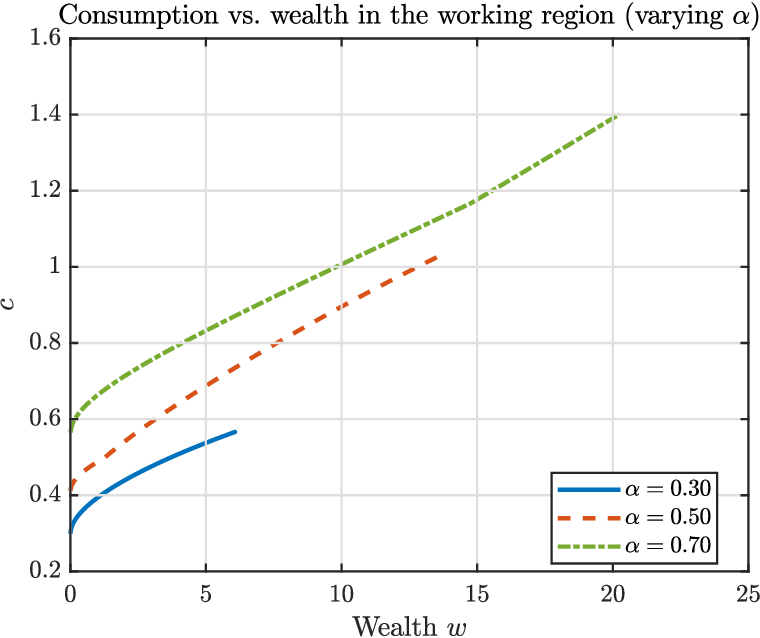}
	}
	\subfigure[Leisure as a function of wealth for different values of \(\alpha\).]{
		\includegraphics[scale=0.4]{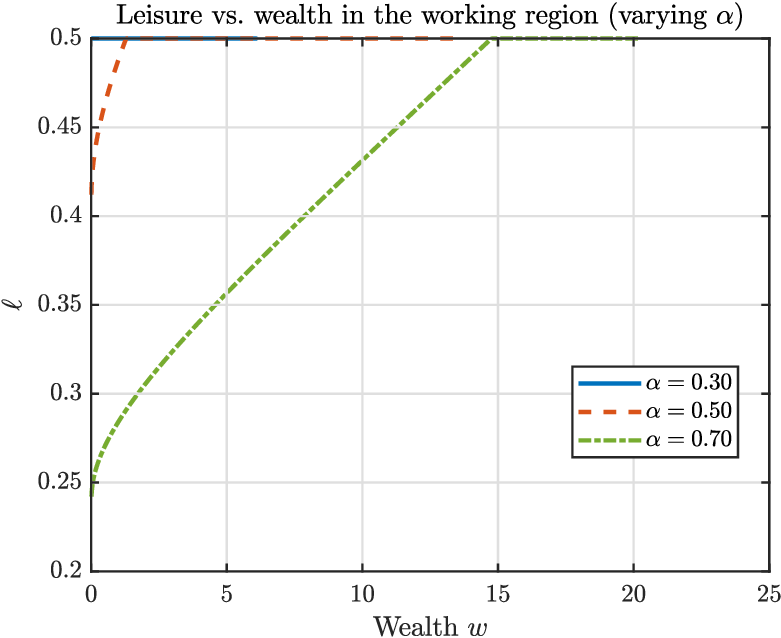}
	}
    \subfigure[Portfolio position as a function of wealth for different values of \(\alpha\).]{
		\includegraphics[scale=0.4]{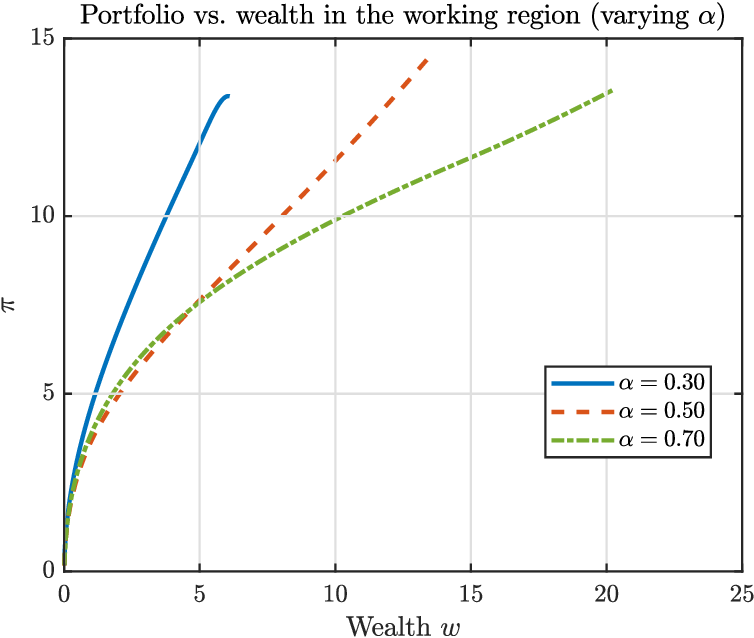}
	}
	\caption{Comparative statics with respect to \(\alpha\) in the working region.}
    \label{fig:alpha-effect}
\end{figure}

Figure~\ref{fig:L-effect} reports the comparative statics with respect to the upper leisure bound \(L\) in the working region. A larger \(L\) relaxes the upper bound on leisure before retirement and therefore gives the agent greater flexibility in labor supply while continuing to work. In the benchmark calibration, a higher \(L\) is associated with lower consumption, higher leisure, and a smaller risky position for a given wealth level. These patterns suggest that greater leisure flexibility may reduce labor effort and, through this channel, lower both current spending and the demand for risky investment.

\begin{figure}[ht]
	\centering
	\subfigure[Consumption as a function of wealth for different values of \(L\).]{
		\includegraphics[scale=0.4]{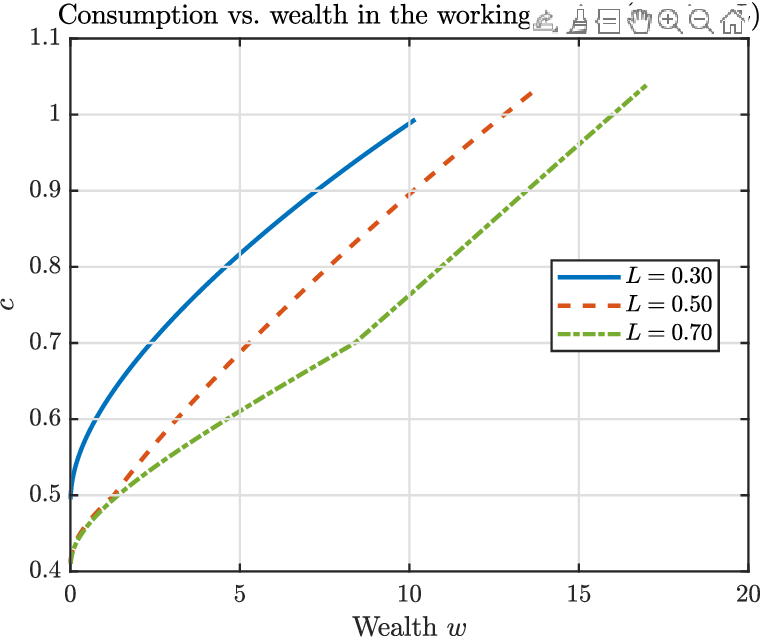}
	}
	\subfigure[Leisure as a function of wealth for different values of \(L\).]{
		\includegraphics[scale=0.4]{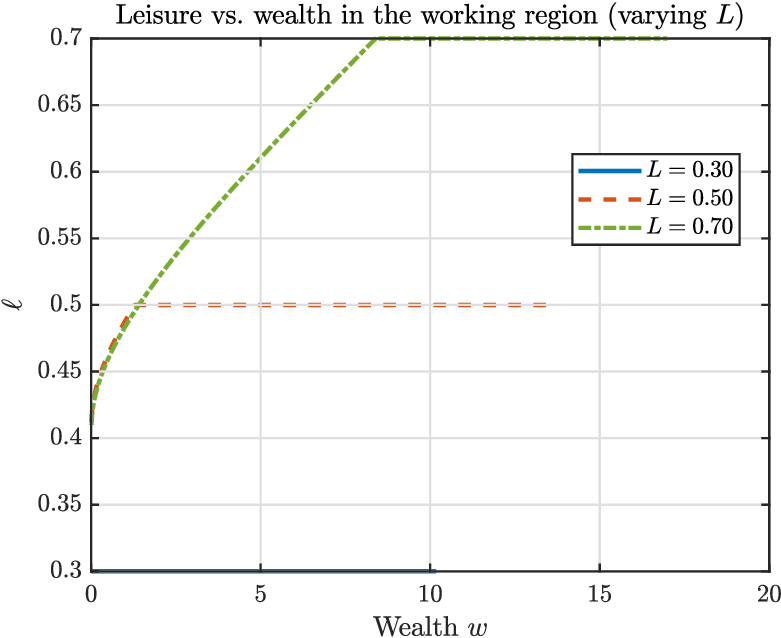}
	}
    \subfigure[Portfolio position as a function of wealth for different values of \(L\).]{
		\includegraphics[scale=0.4]{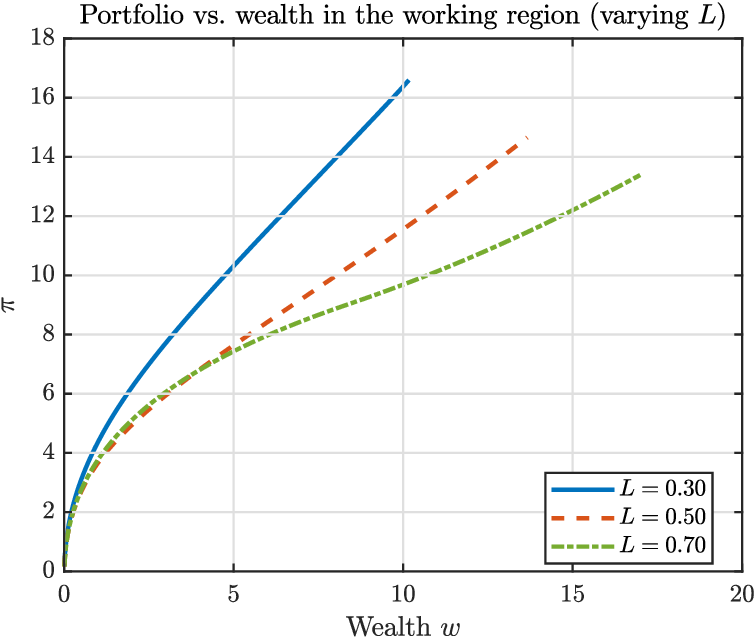}
	}
	\caption{Comparative statics with respect to \(L\) in the working region.}
    \label{fig:L-effect}
\end{figure}

Figure~\ref{fig:Lbar-effect} reports the comparative statics with respect to the full-leisure level \(\bar L\) in the working region. The parameter \(\bar L\) determines the post-retirement leisure level and, through the labor-income term \(\zeta(\bar L-l)\), also affects the effective scale of labor income while the agent is still working. In the benchmark calibration, a higher \(\bar L\) is associated with higher consumption, higher leisure, and a larger risky position for a given wealth level. A possible interpretation is that a larger \(\bar L\) raises the value of future retirement and improves effective lifetime resources, thereby shifting the working-region policies upward.

\begin{figure}[ht]
	\centering
	\subfigure[Consumption as a function of wealth for different values of \(\bar L\).]{
		\includegraphics[scale=0.4]{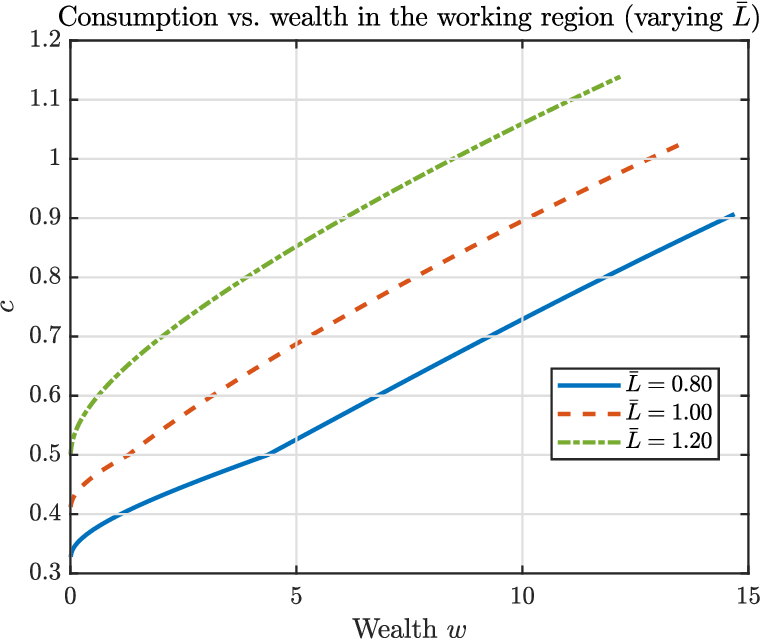}
	}
	\subfigure[Leisure as a function of wealth for different values of \(\bar L\).]{
		\includegraphics[scale=0.4]{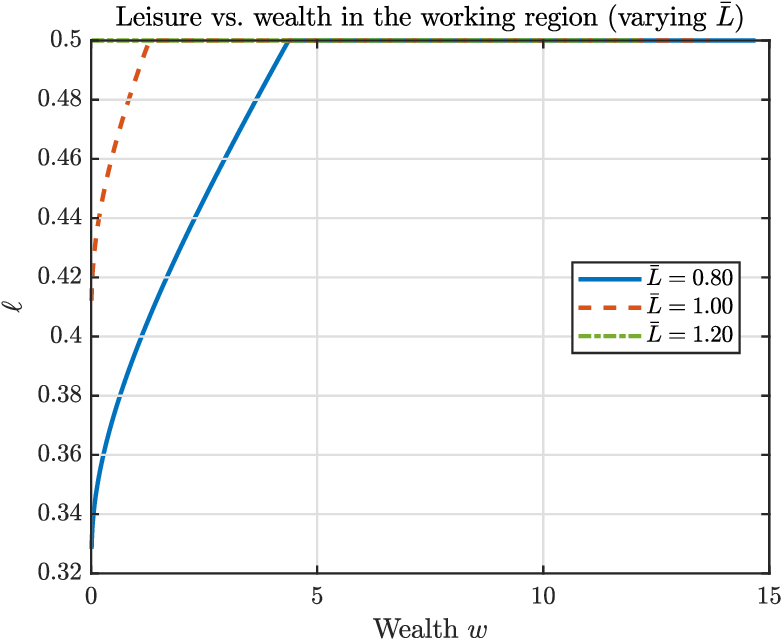}
	}
    \subfigure[Portfolio position as a function of wealth for different values of \(\bar L\).]{
		\includegraphics[scale=0.4]{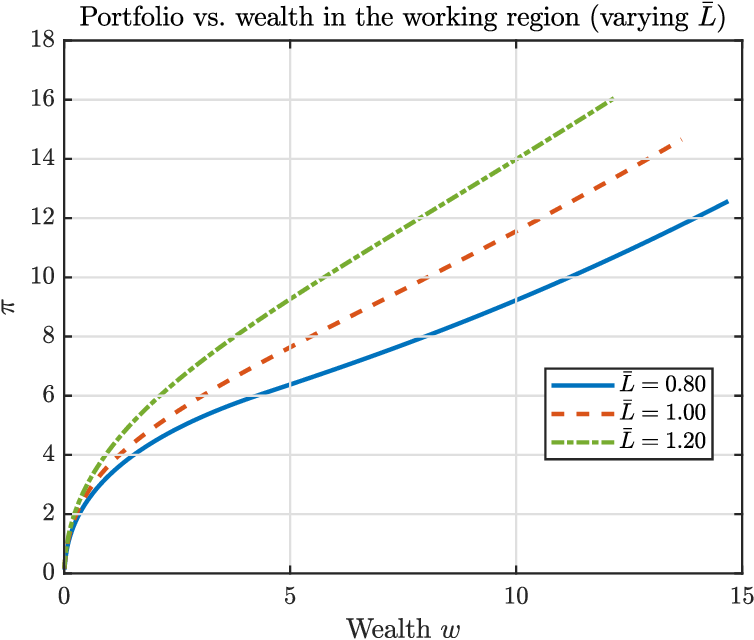}
	}
	\caption{Comparative statics with respect to \(\bar L\) in the working region.}
    \label{fig:Lbar-effect}
\end{figure}

Figure~\ref{fig:T-effect} reports the horizon effect in the working region. In the benchmark calibration, a longer horizon is associated with higher consumption, higher leisure, and a larger risky position for a given wealth level. One possible interpretation is that a longer horizon increases the value of future labor-supply flexibility and makes continued work more attractive, thereby relaxing the effective intertemporal resource constraint while the agent remains in the working region. Under this interpretation, the upward shift in consumption, leisure, and risky investment can be viewed as a consequence of the greater value of future adjustment opportunities.

\begin{figure}[ht]
	\centering
	\subfigure[Consumption as a function of wealth for different values of \(T\).]{
		\includegraphics[scale=0.4]{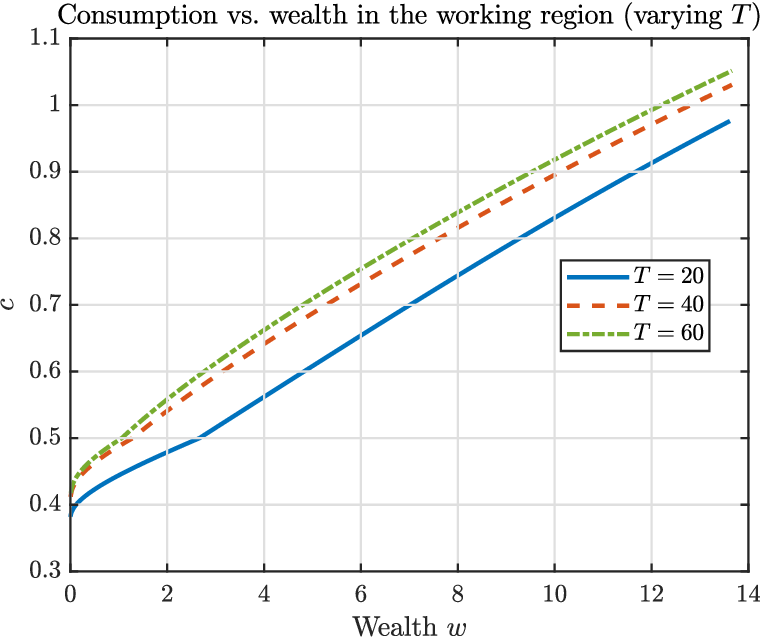}
	}
	\subfigure[Leisure as a function of wealth for different values of \(T\).]{
		\includegraphics[scale=0.4]{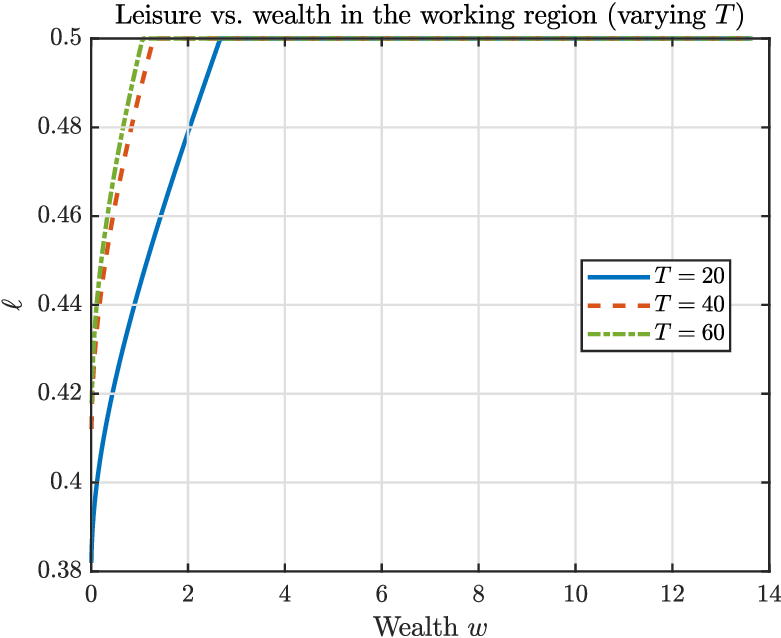}
	}
    \subfigure[Portfolio position as a function of wealth for different values of \(T\).]{
		\includegraphics[scale=0.4]{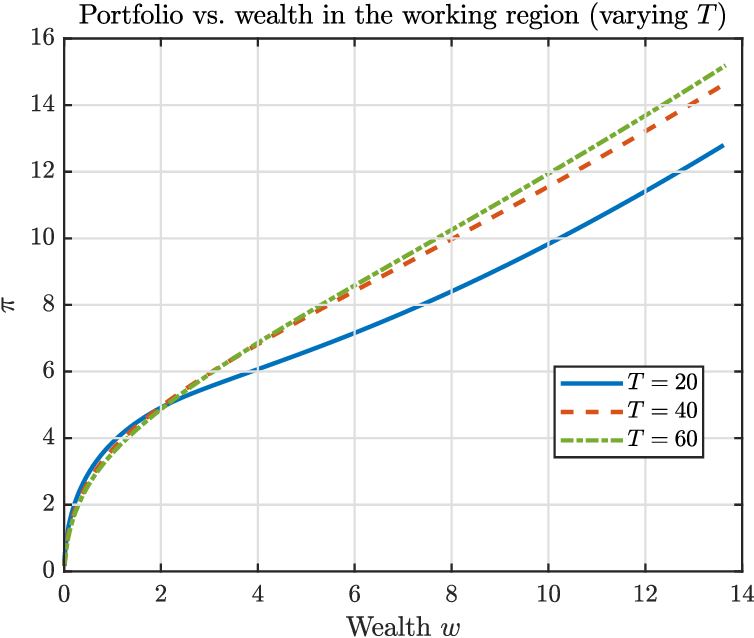}
	}
	\caption{Comparative statics with respect to the horizon \(T\) in the working region.}
    \label{fig:T-effect}
\end{figure}

\section{Concluding Remarks}

This paper analyzed a finite-horizon consumption--investment problem with endogenous labor supply and irreversible early retirement under a borrowing constraint. The agent chooses consumption, portfolio allocation, leisure, and the retirement time, subject to the requirement that wealth remain nonnegative.

{Using a dual martingale method, we transformed the primal problem into a zero-sum stopper--singular-controller game and characterized the dual object by a min--max parabolic variational inequality with obstacle and gradient constraints. The maximal strong solution is selected by the gradient-constrained structure and generates the binding boundary \(z_B(t)\), which is monotone increasing in time; the verification theorem identifies this solution with the dual game value; and the duality theorem recovers the optimal retirement, consumption, leisure, and portfolio decisions.}

The numerical results complement the theoretical analysis by illustrating how labor supply flexibility affects the free boundaries and the associated optimal policies. Our results provide a rigorous framework for studying the joint effects of labor supply flexibility, portfolio choice, and retirement timing under borrowing constraints. Extensions to more general preferences, stochastic labor income, and alternative retirement regimes remain interesting directions for future research.
\vspace{1cm}
\appendix 

\begin{center}
    {\LARGE {\bf Appendix}}
\end{center}

\section{{Proofs}}\label{sec:proofs}
{This appendix collects the proofs of the results stated in the main text.}

\subsection{{Proof of Lemma~\ref{Y n,e}}}\label{app:proof:Y-n-e}
The existence of the solution $Y_{n,\e}$ in $W^{1,2}_p(\O_T^n)$, $1<p<\infty$ is guaranteed by the Leray-Schauder fixed point theorem, and the embedding theorem \cite{Ladyzhenskaya} implies $Y_{n,\e}\in C^{\frac{1+\a}{2},1+\a}(\O_T^n)$ for some $\a\in(0,1)$.
Since $f\in C_{\rm loc}^{1+\a}(\mathbb{R})$ the right hand side of the equation is in $C^{\frac{\a}{2},\a}(\O_T^n)$.
Thus, we deduce by the Schauder theory that $Y_{n,\e}\in C^{1+\frac{\a}{2},2+\a}(\O_T^n)$.
Again, since the right hand side is in $C^{\frac{1+\a}{2},1+\a}(\O_T^n)$,
the Schauder theory ensures that $Y_{n,\e}\in C^{\frac{3+\a}{2},3+\a}(\Omega_T^n)$.

Let us establish the inequality $Y_{n,\e}\ge \psi$.
Observe that $\psi'(x)<0$ and $\displaystyle\lim_{\e\to 0^+} \varphi_\e(x)= 0$ for all $x\in[-n,n]$. 
Hence, for sufficiently small $\d>0$ satisfying $\psi'(x)<-2\d$ for all $x\in[-n,n]$,
there exists $\e_0\in (0,\d)$ such that
$\vert \varphi_\e \vert < \d$ in $[-n,n]$
for all $0<\e<\e_0$.
This implies that for all $0<\e<\e_0$,
\[\psi'(x)-\varphi_\e(x)\le -\d<-\e 
\quad \text{for all }\ x\in [-n,n].\]
This yields $\b_{u,\e}(\psi'-\varphi_\e)=0$ for $\e\in (0,\e_0)$ in $\O_T^n$.
Moreover, it can be seen by direct computation that
the function $e^{-x}(f(x)+\cL \psi(x))$ is monotone increasing in $\mathbb{R}$.
Therefore, for all $n\in\mathbb{N}$ satisfying $-n<\tilde{x}$, we have
\begin{align*}
    e^{n}\big(f(-n)+\cL \psi(-n)\big)
    &=\frac{\g_1}{1-\g_1}(L^\gam-\bar{L}^\gam)e^{\frac{n}{\g_1}}+\zeta(\bar{L}-L)
    \\
    &\le e^{-x}\big(f(x)+\cL \psi(x)\big)
\end{align*}
for all $x\in[-n,n]$.
Since $\frac{\g_1}{1-\g_1}(L^\gam-\bar{L}^\gam)<0$, we see that $f(-n)+{\cL}\psi(-n)$ is negative for sufficiently large $n\in\mathbb{N}$.
Consequently, 
\[f(x)+\cL \psi(x)
\ge e^{x+n}\{f(-n)+\cL \psi(-n)\}
\ge e^{2n}\{f(-n)+\cL \psi(-n)\}
\quad \text{for all }\ x\in [-n,n].\]
% Moreover, for sufficiently large $n\in\mathbb{N}$, we see that $u_0(-n)+\widetilde{\cL}\psi_0(-n)$ is negative and this implies
% \[e^x\{u_0(-n)+\widetilde{\cL} \psi_0(-n)\}\ge e^n\{u_0(-n)+\widetilde{\cL} \psi_0(-n)\} \quad \text{for all }\ x\in[-n,n].\]
Since it is prescribed that $\b_{l,\e}(0)=-e^{2n}\{f(-n)+\cL \psi(-n)\}$, we combine all the results to see
\[-\cL \psi-\b_{l,\e}(0)\le f
\quad \text{in }\ \O_T^n.\]
Therefore, $\pat_\tau \psi-\cL \psi-\b_{l,\e}(0)+\b_{u,\e}(\psi'-\varphi_\e)\le f$ in $\O_T^n$ and $\pat_x Y=\psi'$ and $Y=\psi$ on the lateral boundaries and the initial boundary respectively.
Thus, by the comparison principle \cite[Theorem II.1]{BARLES1999191},
it follows that $Y_{n,\e}\ge \psi$ in $\O_T^n$.

To prove $\pat_x Y_{n,\e}\le 0$, 
we differentiate the equation in \eqref{obs1:pen} to see that
$\pat_x Y_{n,\e}$ satisfies
\begin{equation}\label{dx Y(n,e)}
    \pat_\tau (\pat_x Y_{n,\e})-\cL(\pat_x Y_{n,\e})
    =f'+\b'_{l,\e}(\cdots)(\pat_x Y_{n,\e}-\psi')+\b'_{u,\e}(\pat_x Y_{n,\e}-\varphi_\e)\pat_x(\pat_x Y_{n,\e}-\varphi_\e)
\end{equation}
in $\O_T^n$.
Suppose that $\pat_x Y_{n,\e}$ attains its strict positive maximum at some point $(\tau_0,x_0) \in [0,T] \times [-n,n]$. Then, since $\pat_xY_{n,\e}=\psi'<0$ on the parabolic boundary of $\O_T^n$, $(\tau_0,x_0)$ is contained in $(0,T]\times(-n,n)$.
Therefore,
\[\pat_{\tau} (\partial_x Y_{n,\e})(\tau_0,x_0)\ge 0,
\quad 
\pat_{x} (\partial_x Y_{n,\e})(\tau_0,x_0)= 0,
\quad \text{and}\quad 
\pat_{xx} (\partial_x Y_{n,\e})(\tau_0,x_0)\le 0.\]
This implies $\pat_\tau (\pat_xY_{n,\e})(\tau_0,x_0) -\cL (\pat_xY_{n,\e})(\tau_0,x_0)\ge 0.$
Moreover, at the point $(\tau_0,x_0)$, the right hand side of \eqref{dx Y(n,e)} becomes
\begin{align*}
    f'(x_0)+\b'_{l,\e}(\cdots)&\{\pat_x Y_{n,\e}(\tau_0,x_0)-\psi'(x_0)\}
    +\b'_{u,\e}\big(\pat_x Y_{n,\e}(\tau_0,x_0)-\varphi_\e(x_0)\big)(\pat_{xx}Y_{n,\e}(\tau_0,x_0)-\varphi'_\e(x_0))
    \\
    &\le 
     f'(x_0)-\b'_{u,\e}\big(-\varphi_\e(x_0)\big)\varphi'_\e(x_0)
     \\
     &\le f'(x_0)+\frac{2}{\e}\varphi'_\e(x_0)
     =f'(x_0)-2\zeta\bar{L}e^{x_0}< 0,
\end{align*}
where the first inequality follows from $\pat_x Y_{n,\e}(\tau_0,x_0)>0$, $\psi'(x_0)<0$, $\b'_{l,\e}(\cdot)\le 0$ and $\b''_{u,\e}(\cdot)\le 0$, and the second inequality follows from $\pat_{xx} Y_{n,\e}(\tau_0,x_0)=0$, $\varphi_\e(\cdot)<0$ and $\b'_{u,\e}(s)=-\frac{2}{\e}$ for all $s\ge 0$.
Finally, the last strict inequality holds since $f'(x) \le \zeta\bar{L}e^x$ for all $x \in [-n,n]$.
This contradicts the fact that the left-hand side of \eqref{dx Y(n,e)} satisfies $\partial_\tau (\partial_x Y_{n,\varepsilon})(\tau_0,x_0) - \mathcal L (\partial_x Y_{n,\varepsilon})(\tau_0,x_0) \ge 0$.
Therefore, $\partial_x Y_{n,\varepsilon}$ cannot attain a positive maximum value in $\Omega_T^n$,
and we conclude that
$\partial_x Y_{n,\varepsilon} \le 0$ in $\Omega_T^n$.

To prove the time monotonicity of $Y_{n,\e}$, define $Y_{n,\e}^\d$ by
\[Y_{n,\e}^\d(\tau,x)=Y_{n,\e}(\tau+\d,x)
\quad \text{for all }\ (\tau,x)\in \O_{T-\d}^n:=(0,T-\d)\times(-n,n).\]
Then, $Y_{n,\e}^\d$ satisfies
\begin{equation}
        \begin{cases}
            \partial_\tau Y_{n,\e}^\d -\mathcal{L}Y_{n,\e}^\d
            =f+\b_{l,\e}(Y_{n,\e}^\d-\psi)+\b_{u,\e}(\pat_x Y_{n,\e}^\d-\varphi_\e)
            \quad \text{in }\ \Omega_{T-\d}^n,
            \\
            \pat_x Y_{n,\e}^\d(\tau,-n)= \psi'(-n)
            \quad \text{and}\quad 
            \pat_x Y_{n,\e}^\d(\tau,n)= \psi'(n)
            \quad \text{for \ all }\ \tau\in(0,T-\d),
            \\
            Y_{n,\e}^\d(0,x)=Y_{n,\e}(\d,x)\ge \psi(x) \quad \text{for all }\ x\in[-n,n],
        \end{cases}
    \end{equation}
    By the comparison principle \cite[Theorem II.1]{BARLES1999191}, we deduce that $Y_{n,\e}^\d\ge Y_{n,\e}$ in $\O_{T-\d}^n$. Since $\d>0$ is arbitrary, we conclude $\pat_\tau Y_{n,\e}\ge 0$ in $\O_T^n$.

To prove the last inequality,
we define $\overline{Y}$ by 
\begin{equation*}
\overline{Y}(x)=
    \begin{cases}
        \frac{\zeta\bar{L}}{r}e^x+\frac{2\bar{L}^\gam}{\min\{1,K_1\}}\frac{\g_1}{1-\g_1}e^{(1-\frac{1}{\g_1})x}+\frac{\zeta\g}{\b(\g_1-\g)}\left(\frac{\g_1-\g}{\zeta(1-\g_1)}\right)^{\frac{\g_1}{\g}}e^{(1-\frac{1}{\g})(\tilde{x})}
        \quad &\text{for }\ 0<\g<1,
        \\
        \frac{\zeta\bar{L}}{r}e^x + e^{\lambda_2 x}
        \quad &\text{for }\ \g>1,
    \end{cases}
\end{equation*}
where $\lambda_2 < 0$ is the negative root of the characteristic equation $\frac{\t^2}{2}\l^2 + (\b-r-\frac{\t^2}{2})\l - \b = 0$.
It can be seen by the comparison principle \cite[Theorem 17, p.53]{FriedmanParabolicPDE} that $Y_{n,\e}<\overline{Y}$ in $\O_T^n$.
Therefore, the growth condition of $Y_{n,\e}$ follows from the inequality $\psi\le Y_{n,\e}\le \overline{Y}$.

\subsection{{Proof of Lemma~\ref{Yn}}}\label{app:proof:Yn}
By the $W^{1,2}_p$-estimate for $1<p<\infty$ \cite[Theorem 7.35]{Lieberman}, we see that $Y_{n,\e}$ satisfies
    \begin{equation}\label{Y_{n,e}: W,p estimate}
        \Vert Y_{n,\e} \Vert_{W^{1,2}_p(\Omega_T^n)}
        \le C\Big(\Vert f+\b_{l,\e}(Y_{n,\e}-\psi)+\b_{u,\e}(\pat_x Y_{n,\e}-\varphi_\e)\Vert_{L^p(\Omega_T^n)}
        +\Vert \psi \Vert_{W^{1,2}_p(\O_T^n)}\Big)
    \end{equation}
    for all small $\e>0$.
    Moreover, from that $Y_{n,\e}\ge \psi$, $\pat_x Y_{n,\e}\le 0$ and the monotonicity of $\b_{l,\e}$ and $\b_{u,\e}$, we have \[0 \le \b_{l,\e}(Y_{n,\e}-\psi)\le \b_{l,\e}(0)
    \quad \text{and} \quad
    \b_{u,\e}(-\varphi_\e)\le \b_{u,\e}(\pat_x Y_{n,\e}-\varphi_\e)\le 0.\]
    Since $\b_{u,\e}(-\varphi_\e(x))=-2\zeta\bar{L}e^x-1$ for all $x\in[-n,n]$, we also see that $\b_{u,\e}(\pat_x Y_{n,\e}-\varphi_\e)$ is bounded below by $-2\zeta\bar{L}e^n-1$ in $\O_T^n$.
    Therefore, the estimate \eqref{Y_{n,e}: W,p estimate} reduces to
    \begin{equation*}
        \Vert Y_{n,\e} \Vert_{W^{1,2}_p(\Omega_T^n)}
        \le C\Big(\Vert f\Vert_{L^p(\Omega_T^n)}
        +\b_{l,\e}(0)+2\zeta\bar{L}e^n+1
        +\Vert \psi \Vert_{W^{1,2}_p(\O_T^n)}\Big).
    \end{equation*}
    Since the right hand side is independent of $\e$, there exists a subsequence $\{Y_{n,\e_j}\}_{j=1}^\infty$ that converges weakly in $W^{1,2}_p(\Omega_T^n)$ for all $1<p<\infty$, to a function say $Y_n$. 
    By the Sobolev embedding theorem, 
    $Y_{n,\e_j}\to Y_n$ 
    and $\pat_x Y_{n,\e_j}\to \pat_x Y_n$
    uniformly in $\overline{\Omega_T^n}$ as $j\to\infty$.
    Therefore, $Y_n\ge \psi$ and $\pat_x Y_n\le 0$ in $\overline{\Omega_T^n}$ and the boundary condition of \eqref{obs1:pen} carries over to \eqref{obs1:bdd}.

    To prove that $Y_n$ satisfies \eqref{obs1:bdd},
    let us fix a non-negative $\eta\in C_c^\infty(\{Y_n>\psi\})$. Then, there exists $\d>0$ such that
    $\min\limits_{\text{supp}(\eta)}(Y_n-\psi)>2\d$.
    Since $Y_{n,\e_j}$ converges uniformly to $Y_n$,
    we may choose $j_0\in \mathbb{N}$ large so that $\e_j<\d$ and $\vert Y_{n,\e_j}-Y_n \vert<\d$ for all $j\ge j_0$. 
    Then, $Y_{n,\e_j}-\psi>\d$ and this implies $\b_{l,\e_j}(Y_{n,\e_j}-\psi)\eta=0$ for all $j\ge j_0$.
    Moreover, since $\b_{u,\e_j}\le 0$,
    \[\int_{\Omega_T^n} (\partial_\tau Y_{n,\e_j} -\mathcal{L}Y_{n,\e_j})\eta\ dxd\tau 
    = 
    \int_{\Omega_T^n} f\eta+\b_{u,\e_j}(\pat_xY_{n,\e_j}-\varphi_{\e_j})\eta \ dxd\tau
    \le \int_{\Omega_T^n}f\eta \ dxd\tau
    \quad \text{in }\ \Omega_T^n.\]
    Take $j\to\infty$ and the weak convergence $Y_{n,\e_j}\rightharpoonup Y_n$ in $W^{1,2}_p(\Omega_T^n)$ implies
    \[\int_{\Omega_T^n}(\partial_\tau Y_n -\mathcal{L}Y_n)\eta \ dxd\tau
    \le \int_{\Omega_T^n} f \eta \ dxd\tau
    \quad \text{in }\ \Omega_T^n.\]
    Since $\eta$ is an arbitrary non-negative function in $C_c^\infty(\{Y_n>\psi\})$, we deduce that
    $\partial_\tau Y_n -\mathcal{L}Y_n \le f
    \quad \text{a.e. in }\ \{Y_n>\psi\}$.
    Let us fix $\eta\in C_c^\infty(\{\pat_x Y_n<0\})$.
    Then, there exists $\d>0$ such that
    $\max\limits_{\text{supp}(\eta)}\pat_x Y_n<-2\d$.
    Choose $j_0\in\mathbb{N}$ large so that $\e_j<\d$ and $-\d<\varphi_{\e_j}(\cdot)<0$ for all $j\ge j_0$.
    Then, we have $\pat_x Y_{n,\e_j}-\varphi_{\e_j}<-\d<-{\e_j}$.
    Therefore, $\b_{u,\e_j}(\pat_x Y_{n,\e_j}-\varphi_{\e_j})\eta=0$ for all $j\ge j_0$.
    Note $\b_{l,\e_j}(\cdot)\ge 0$ to see that
    \[\int_{\Omega_T^n} (\partial_\tau Y_{n,\e_j} -\mathcal{L}Y_{n,\e_j})\eta\ dxd\tau 
    = 
    \int_{\Omega_T^n} f\eta+\b_{l,\e_j}(Y_{n,\e_j}-\psi)\eta \ dxd\tau
    \ge \int_{\Omega_T^n}f\eta \ dxd\tau
    \quad \text{in }\ \Omega_T^n.\]
    Take $j\to\infty$ and the weak convergence $Y_{n,\e_j}\rightharpoonup Y_n$ in $W^{1,2}_p(\Omega_T^n)$ implies
    \[\int_{\Omega_T^n}(\partial_\tau Y_n -\mathcal{L}Y_n)\eta \ dxd\tau
    \ge \int_{\Omega_T^n} f \eta \ dxd\tau
    \quad \text{in }\ \Omega_T^n.\]
    Since $\eta$ is an arbitrary non-negative function in $C_c^\infty(\{\pat_xY_n<0\})$, we deduce that
    $\partial_\tau Y_n -\mathcal{L}Y_n \ge f
    \quad \text{a.e. in }\ \{\pat_xY_n<0\}$. 
    Combining the two results, we also deduce that
    $\pat_\tau Y_n -\cL Y_n =f
    \quad \text{a.e. in }\ \{Y_n>\psi\}\cap\{\pat_x Y_n<0\}$.
    This proves that $Y_n$ is a solution of \eqref{obs1:bdd}.

\subsection{{Proof of Lemma~\ref{Y}}}\label{app:proof:Y}
Since $Y_n \in W^{1,2}_p(\O_T^n)$, we have
\[
    \pat_\tau Y_n -\cL Y_n = f\mathbf{1}_{\{Y_n>\psi,\ \pat_x Y_n<0\}}+ \left(\pat_\tau Y_n -\cL Y_n\right)(\mathbf{1}_{\{Y_n=\psi\}}+\mathbf{1}_{\{\pat_x Y_n=0\}})
    \quad \text{a.e. in }\ \O_T^n.
\] 
    Fix any $R>0$ and consider $Y_n$ in the subregion 
    $\O_T^{2R}:=(0,T)\times(-2R,2R)$ for each 
    $n\in \mathbb{N}$ satisfying $n>2R$.
    In the set $\{\pat_x Y_n=0\}$, we have $\pat_\tau Y_n-\cL Y_n\le f$ and $\pat_x Y_n = \pat_{xx} Y_n = 0$ almost everywhere.
    This implies 
    \[\pat_\tau Y_n+\b Y_n=\pat_\tau Y_n-\cL Y_n\le f
    \quad \text{a.e. in }\ \{\pat_x Y_n=0\}.\]
    Moreover, taking the limit $\varepsilon \to 0^+$ in the inequality within Lemma~\ref{Y n,e}, we obtain
    $\partial_\tau Y_n \ge 0$ a.e. in $\Omega_T^{n}$.
    Together with the previous inequality and that $Y_n\ge \psi$, we deduce
    $\b\psi \le \pat_\tau Y_n +\b Y_n\le f$ a.e. in $\{\pat_x Y_n=0\}$
    and hence the norm $\Vert (\pat_\tau Y_n -\cL Y_n)\mathbf{1}_{\{\pat_x Y_n=0\}} \Vert_{L^p(\O_T^{2R})}$ is bounded by a constant depending only on $R$.
    In the set $\{Y_n=\psi\}$, we have $\pat_\tau Y_n -\cL Y_n=-\cL \psi$ and thus, $\Vert (\pat_\tau Y_n -\cL Y_n)\mathbf{1}_{\{Y_n=\psi\}} \Vert_{L^p(\O_T^{2R})}$ depends only on $R$.
    Therefore, we conclude that $\Vert\pat_\tau Y_n -\cL Y_n\Vert_{L^p(\O_T^{2R})}$ is bounded by a constant depending only on $R$.
    
    Applying the interior $W^{1,2}_p$-estimate \cite[Theorem 7.35]{Lieberman} to $Y_n$ for $1<p<\infty$ yields
    \begin{align*}
        \Vert Y_n \Vert_{W^{1,2}_p(\Omega_T^R)}
        &\le C\left( \Vert Y_n \Vert_{L^p(\Omega_T^{2R})}
        +\Vert \pat_\tau Y_n -\cL Y_n\Vert_{L^p(\Omega_T^{2R})}
        +\Vert \psi \Vert_{W^{1,2}_{p}(\O_T^{2R})} \right)
        \\
        &\le C\left( \Vert \overline{Y}+\psi \Vert_{L^p(\Omega_T^{2R})}
        +\Vert \pat_\tau Y_n -\cL Y_n\Vert_{L^p(\Omega_T^{2R})}
        +\Vert \psi \Vert_{W^{1,2}_{p}(\O_T^{2R})} \right),
    \end{align*}
    for some constant $C>0$.
    % where we have used the estimate $\psi\le Y_{n}\le \overline{Y}$ at the second inequality,
    % which is derived by taking $\e\to 0^+$ at the inequality in Lemma \ref{Y n,e}.
    Since the right hand side of the above estimate is bounded by a constant independent of $n$, there exists a subsequence $\{Y_{n_j}\}_{j=1}^\infty$
    that converges weakly in $W^{1,2}_p(\O_T^R)$ for any $1<p<\infty$.
    By a standard diagonal argument,
    we may construct a diagonal subsequence $\{Y_{n_{k,k}}\}_{k=1}^\infty$ that converges weakly to a globally defined limit in $\Omega_T$.
    Moreover, by applying arguments similar to those in Lemma \ref{Yn}, we can verify that this limit $Y$ is a solution to \eqref{obs1}.
    Furthermore, the embedding theorem implies $Y\in C(\overline{\O_T^R})$ for each $R>0$. This yields $Y\in C(\overline{\O_T})$.

\subsection{{Proof of Lemma~\ref{Y>psi}}}\label{app:proof:Y-psi}
Define $\cF_0(x):=f_0(x)+\widetilde{\cL} \psi_0(x)$.
    By direct computation, it can be checked that $\cF_0$ is strictly increasing. 
    In particular, 
    \begin{equation}\label{H'}
        \cF_0'(x)\ge \frac{\bar{L}^{\frac{\g_1-\g}{\g_1}}-L^{\frac{\g_1-\g}{\g_1}}}{1-\g_1}e^{-\frac{1}{\g_1}x} \quad \text{for all }\ x\in\mathbb{R}.    
    \end{equation}
    Observe that $\cF_0(-n)<0$ for sufficiently large $n>0$ and
    $\displaystyle\lim_{x\to\infty}\cF_0(x)=\zeta\bar{L}>0$.
    Therefore, there exists a unique point $x^*$ satisfying
    $\cF_0(x^*)=f_0(x^*)+\widetilde{\cL} \psi_0(x^*)=0$.
   In particular, direct evaluation verifies that $x^*<\tilde{x}$, which yields \eqref{x^*} (see \cite[Lemma A.1]{JeonOh2023}).
    
    Since $Y_0\in W^{1,2}_{p,{\rm loc}}(\O_T)$ for $1<p<\infty$,
    $\pat_\tau \psi_0(x) -\widetilde{\cL} \psi_0(x) = \pat_\tau Y_0(\tau,x) -\widetilde{\cL} Y_0(\tau,x) \ge f_0(x)$
    holds for almost every point $(\tau,x)\in \O_T$ satisfying $Y_0(\tau,x)=\psi_0(x)$.
    This implies the inclusion that $\{Y_0=\psi_0\}\subset \{(\tau,x)\in\O_T: x\le x^*\}$.
    From this result, we obtain that $Y_0>\psi_0$ in $(0,T)\times(x^*,\infty)$.

\subsection{{Proof of Lemma~\ref{contact_1}}}\label{app:proof:contact-1}
% For each $(\tau,x)\in(0,T)\times(-\infty,x^*)$ and $\e>0$, define a function $\Gamma_\e$ by
    % \[\Gamma_\e(x):=\e^3\left(x+\frac{1}{\e^2}\right)^2\mathbf{1}_{\{x\ge-\frac{1}{\e^2}\}}.
    % \]
    % It can be checked that $\Gamma_\e$ belongs to $W^{1,2}_{p,{\rm loc}}((0,T)\times(-\infty,x^*-1))$.
    % Thus, for all $(\tau,x)\in(0,T)\times(-\infty,x^*-1)$, we compute by
    % \begin{equation*}
    % \begin{aligned}
    %     \pat_\tau\Gamma_\e-\widetilde{\cL}\Gamma_\e 
    %     &=\e^3\left[-\t^2-2(\b-r+\frac{\t^2}{2})(x+\frac{1}{\e^2})+r(x+\frac{1}{\e^2})^2\right]\mathbf{1}_{\{x\ge-\frac{1}{\e^2}\}}
    %     \\
    %     &\ge \e^3\left[-\t^2-(2\b+\t^2)(x+\frac{1}{\e^2})\right]\mathbf{1}_{\{x\ge-\frac{1}{\e^2}\}}.
    % \end{aligned}
    % \end{equation*}
For each $x\in(-\infty,x^*-1]$, define
$\Pi_\e(x):=\psi_0(x)+\e^3\left(x+\frac{1}{\e^2}\right)^2\mathbf{1}_{\{x\ge-\frac{1}{\e^2}\}}$
and set $p_\e(x):=\partial_\tau\Pi_\e(x)-\widetilde{\cL}\Pi_\e(x)$.
Then, since $-\widetilde{\cL}\psi_0(x)\ge f_0$ for $x\le x^*$,
$\Pi_\e$ satisfies the following variational inequality:
\begin{equation*}
    \begin{cases}
        \pat_\tau \Pi_\e(x) -\widetilde{\cL} \Pi_\e(x) = p_\e(x)
        \quad \text{for }\ (\tau,x)\in(0,T)\times(-\infty,x^*-1) \ \text{ with }\ \Pi_\e(x)>\psi_0(x),
        \\
        \pat_\tau \Pi_\e(x) -\widetilde{\cL} \Pi_\e(x) \ge f_0(x)
        \quad \text{for }\ (\tau,x)\in(0,T)\times(-\infty,x^*-1) \ \text{ with }\ \Pi_\e(x)=\psi_0(x),
        \\
        \Pi_\e(x^*-1)=\psi_0(x^*-1)+\e^3(x^*-1+\frac{1}{\e^2})^2\ge Y_0(\tau,x^*-1)
        \quad \text{for \ all }\ \tau\in(0,T),
        \\
        \Pi_\e(x)=\psi_0(x)+\Pi_\e(x)\ge \psi_0(x)=Y_0(0,x)
        \quad \text{for \ all }\ x\in(-\infty,x^*-1].
    \end{cases}
\end{equation*}
Since $Y_0$ is bounded above (see Lemma~\ref{Y n,e}), choosing $\e>0$ small enough yields $\Pi_\e(x^*-1)> Y_0(\tau,x^*-1)$ for all $\tau\in(0,T)$.
Moreover, observe that $\Pi_\e(x)>\psi_0(x)$ holds for all $(\tau,x)\in(0,T)\times(-\infty,x^*-1)$ and for sufficiently small $\e>0$,
\begin{align*}
    \pat_\tau \Pi_\e(x) -\widetilde{\cL} \Pi_\e(x) -f_0(x)
    &\ge \frac{\g_1}{1-\g_1}(\bar{L}^{\frac{\g_1-\g}{\g_1}}-L^{\frac{\g_1-\g}{\g_1}})e^{-\frac{1}{\g_1}(x^*-1)}-\zeta(\bar{L}-L)
    \\
    &\quad -\e^3\left[\t^2+(2\b+\t^2)(x^*-1+\frac{1}{\e^2})\right]\mathbf{1}_{\{x\ge-\frac{1}{\e^2}\}}
    \\
    &=\left(e^{\frac{1}{\g_1}}-1\right)\zeta(\bar{L}-L)
    -\e^3\left[\t^2+(2\b+\t^2)(x^*-1+\frac{1}{\e^2})\right]\mathbf{1}_{\{x\ge-\frac{1}{\e^2}\}}>0.
\end{align*}
Therefore, we deduce that $p_\e(x)\ge f_0(x)$ in $(\tau,x)\in(0,T)\times(-\infty,x^*-1)$.
% Moreover, in the set $\{\Pi_\e=\psi_0\}$, we have $\pi_\e(x)\ge f_0(x)$.
% Therefore, $\pi_\e\ge f_0$ in $(0,T)\times(-\infty,x^*-1)$.
For each $(\tau,x)\in \O_T$, we denote by $h(\tau,x):=\pat_\tau Y_0(\tau,x)-\widetilde{\cL} Y_0(\tau,x)$.
    Then, $Y_0$ satisfies
    \begin{equation*}
    \begin{cases}
        \pat_\tau Y_0(\tau,x) -\widetilde{\cL} Y_0(\tau,x) = h(\tau,x)
        \quad \text{for }\ (\tau,x)\in(0,T)\times(-\infty,x^*-1) \ \text{ with }\ Y_0(\tau,x)>\psi_0(x),
        \\
        \pat_\tau Y_0(\tau,x) -\widetilde{\cL} Y_0(\tau,x) \ge f_0(x)
        \quad \text{for }\ (\tau,x)\in(0,T)\times(-\infty,x^*-1) \ \text{ with }\ Y_0(\tau,x)=\psi_0(x),
        \\
        Y_0(0,x)=\psi_0(x)
        \quad \text{for \ all }\ x\in(-\infty,x^*-1].
    \end{cases}
\end{equation*}
Since 
    $h=\pat_\tau Y_0-\widetilde{\cL} Y_0 = f_0\mathbf{1}_{\{\pat_\tau Y_0+Y_0<0\}}+(\pat_x Y_0-\widetilde{\cL} Y_0)\mathbf{1}_{\{\pat_x Y_0+Y_0 =0\}}\le f_0$
    in the set $\{Y_0>\psi_0\}$, it follows from the comparison principle for the variational inequalities \cite[Theorem 1.2]{YangYan2008} that $\Pi_\e\ge Y_0$ in $(0,T)\times(-\infty,x^*-1]$.
    Moreover, from that $\Pi_\e(\tau,x)=\psi_0(x)$ for all $(\tau,x)\in(0,T)\times (-\infty,-\frac{1}{\e^2}]$ and $Y_0\ge \psi_0$ in $\overline{\O_T}$, we conclude $Y_0(\tau,x)=\psi_0(x)$ for all $(\tau,x)\in [0,T]\times(-\infty,-\frac{1}{\e^2}]$.
    % Since $Y_0=e^{-x}Y$ and $\psi_0=e^{-x}J$, the result directly follows.

\subsection{{Proof of Lemma~\ref{contact_2}}}\label{app:proof:contact-2}
Fix any $\e>0$. 
    Then, from the definition of $x^*$ and the strict increasing property of $\cF_0$, there exists $\nu_\e>0$ such that $-\cF_0(x)=-f_0(x)-\widetilde{\cL}\psi_0(x)>\nu_\e$ for all $x< x^*-\e$. Choose $\d>0$ small so that $\nu_\e>\d(\t^2+(2\b+\t^2)\e)$.
    Since $Y_0(0,x)=\psi_0(x)$ for all $x\in\mathbb{R}$
    and $Y_0\in C(\overline{\Omega_T})$, 
    there exists a constant $c_\e>0$ such that 
    \[\psi_0(x)\le Y_0(\tau,x)<\psi_0(x)+\d\e^2
    \quad \text{for \ all }\ (\tau,x)\in[0,c_\e)\times(x^*-2\e,x^*+2\e).\]
    With such $\e>0$, $\d>0$ and $c_\e>0$, let us define a function $\Phi_\e \in W^{1,2}_{p,{\rm loc}}((0,c_\e)\times(-\infty,x^*-\e))$ by
    \begin{equation*}
        \Phi_\e(x)=\psi_0(x)+\d(x-x^*+2\e)^2\mathbf{1}_{\{x\ge x^*-2\e \}}.
    \end{equation*}
    Then for all $x\in(-\infty,x^*-\e)$, $\Phi_\e$ satisfies
    \begin{align*}
        \pat_\tau \Phi_\e(x)-\widetilde{\cL}\Phi_\e(x)
        &=\pat_\tau \psi_0(x) -\widetilde{\cL} \psi_0(x)
        \\
        &\quad +\d\left[-\t^2-2(\b-r
        +\frac{\t^2}{2})(x-x^*+2\e)+r(x-x^*+2\e)^2\right]
        \mathbf{1}_{\{x\ge x^*-2\e\}}
        \\
        & > f_0(x)+\nu_\e-\d\Big(\t^2+(2\b+\t^2)\e\Big)
        \mathbf{1}_{\{x\ge x^*-2\e\}}
        \\
        &>f_0(x).
    \end{align*}
    Moreover, $\Phi_\e(x^*-\e)=\psi_0(x^*-\e)+\d\e^2\ge Y_0(\tau,x^*-\e)$
    for all $\tau\in[0,c_\e]$ and $\Phi_\e(x)\ge \psi_0(x)=Y_0(0,x)$ for all $x\le x^*-\e$.
    
    Define $q_\e(x):=\pat_\tau \Phi_\e(x)-\widetilde{\cL}\Phi_\e(x)$ for all $x\in (-\infty,x^*-\e)$.
    Then, $\Phi_\e$ satisfies the following variational inequality:
    \begin{equation*}
        \begin{cases}
            \partial_\tau \Phi_\e(x) -\widetilde{\cL} \Phi_\e(x) = q_\e(x) 
            \quad \text{for }\ (\tau,x)\in(0,c_\e)\times(-\infty,x^*-\e)\ \text{ with }\ \Phi_\e(x)>\psi_0,
            \\
            \partial_\tau \Phi_\e(x) -\widetilde{\cL} \Phi_\e(x) \ge f_0(x) 
            \quad \text{for }\ (\tau,x)\in(0,c_\e)\times(-\infty,x^*-\e)\ \text{ with }\ \Phi_\e(x)=\psi_0,
            \\
            \Phi_\e(x^*-\e)\ge Y_0(\tau,x^*-\e)
            \quad \text{for \ all }\ \tau\in[0,c_\e],
            \\
            \Phi_\e(x)\ge \psi_0(x)=Y_0(0,x)
            \quad \text{for \ all }\ x\le x^*-\e.
        \end{cases}
    \end{equation*}
    Since $q_\e\ge f_0$,
    applying the comparison principle for the variational inequalities \cite[Theorem 1.2]{YangYan2008} yields $\Phi_\e \ge Y_0$ in $[0,c_\e]\times(-\infty,x^*-\e]$.
    Since $\Phi_\e(x)=\psi_0(x)$ for all $(\tau,x)\in[0,c_\e]\times (-\infty,x^*-2\e]$, we conclude $Y_0(\tau,x)=\psi_0(x)$ for all $(\tau,x)\in [0,c_\e]\times(-\infty,x^*-2\e]$.

\subsection{{Proof of Lemma~\ref{Y_tau, Y_x}}}\label{app:proof:Y-tau-Y-x}
To prove $\pat_\tau Y_0 \le \zeta\bar{L}$,
    we consider the following penalized problem analogous to \eqref{obs1:pen}:
    \begin{equation}\label{Y pen2}
        \begin{cases}
            \partial_\tau Y_{n,\e} -\mathcal{L}Y_{n,\e}
            =f_\e+\b_{l,\e}(Y_{n,\e}-\psi)+\b_{u,\e}(\pat_x Y_{n,\e}-\varphi_\e)
            \quad \text{in }\ \Omega_T^n,
            \\
            \partial_xY_{n,\e}(\tau,-n)= \psi'(-n)
            \quad \text{and}\quad 
            \partial_xY_{n,\e}(\tau,n)= \psi'(n)
            \quad \text{for \ all }\ \tau\in(0,T),
            \\
            Y_{n,\e}(0,x)=\psi(x)+\e \quad \text{for all }\ x\in[-n,n].
        \end{cases}
    \end{equation}
Here, we have introduced the smoothened function $f_{\e}$
which is defined by
$f_{\e}(x):=m_\e(f_2(x)-f_1(x))+f_1(x)$
for all $x\in \mathbb{R}$,
in order to obtain sufficient regularity of $Y_{n,\e}$,
and
$m_\e(\cdot)\in C^\infty(\mathbb{R})$ is a function which satisfies
\begin{equation}\label{mollifier}
    \begin{cases}
        m_\e(s)\to s^+ \quad \text{as }\ \e\to 0^+,
        \\
        m_\e(s)=0 \quad \text{if } \ s\le -\e,
        \quad
        m_\e(s)=s \quad \text{if } \ s\ge \e,
        \\
        s^+\le m_\e(s)\le (s+\e)^+
        \quad \forall \ s\in\mathbb{R},
        \\
        0\le m'(s)\le 1
        \quad \forall \ s\in\mathbb{R}.
    \end{cases}
\end{equation}
    By differentiating the equation in \eqref{Y pen2} with respect to $\tau$, we see that $\pat_\tau Y_{n,\e}$ satisfies
    \begin{equation}
        \begin{cases}
            \partial_\tau (\pat_\tau Y_{n,\e}) -{\cL}_\b (\pat_\tau Y_{n,\e})
            =0
            \quad \text{in }\ \Omega_T^n,
            \\
            \pat_{x\tau} Y_{n,\e}(\tau,-n)= 0
            \quad \text{and}\quad 
            \pat_{x\tau} Y_{n,\e}(\tau,n)= 0
            \quad \text{for \ all }\ \tau\in(0,T),
            \\
            \pat_\tau Y_{n,\e}(0,x)=f_{\e}(x)+{\cL} \psi(x)-\b\e \quad \text{for }\ x\in(-n,n),
        \end{cases}
    \end{equation}
    where ${\cL}_\b:={\cL} 
    +\b'_{u,\e}(\cdots)\pat_x 
    +\b'_{l,\e}(\cdots)$
    and the initial condition is obtained by taking $\tau\to 0^+$ at the equation in \eqref{Y pen2}, which is justified by the regularity result in \cite[Theorem 19, p.321]{FriedmanParabolicPDE}.
    
    Since $f_0+\widetilde{\cL}\psi_0=\cF_0$ is strictly increasing with $\displaystyle\lim_{x\to \infty}\cF_0(x)=\zeta\bar{L}$,
    we deduce that $f+\cL \psi = e^x (f_0+\widetilde{\cL}\psi_0)\le \zeta\bar{L}e^x$ for all $x\in\mathbb{R}$.
    To apply the comparison principle under the Neumann boundary conditions, we introduce a function 
    $W_n(x) := \zeta\bar{L}e^x + \e + C_n e^{\lambda_2 x}$, 
    where $\lambda_2 < 0$ is the negative root of the characteristic equation $\frac{\theta^2}{2}\lambda^2 + (\beta-r-\frac{\theta^2}{2})\lambda - \beta = 0$ (so that $\mathcal{L}(e^{\lambda_2 x}) = 0$), and $C_n := \frac{\zeta\bar{L}}{-\lambda_2}e^{(\lambda_2-1)n} > 0$.
    
    At the boundaries, since $\pat_{x\tau}Y_{n,\e}(\tau, \pm n) = 0$, we have
    \[ \pat_x W_n(-n) = \zeta\bar{L}e^{-n} + C_n \lambda_2 e^{-\lambda_2 n} = 0 \le \pat_{x\tau}Y_{n,\e}(\tau, -n), \]
    \[ \pat_x W_n(n) = \zeta\bar{L}e^{n} + C_n \lambda_2 e^{\lambda_2 n} = \zeta\bar{L}(e^n - e^{(2\lambda_2-1)n}) > 0 = \pat_{x\tau}Y_{n,\e}(\tau, n). \]
    Combined with $\b'_{l,\e}\le 0$ and $\b'_{u,\e}\le 0$, we deduce that
    \[ -{\cL}_\b W_n \ge -{\cL} W_n = -{\cL}(\zeta\bar{L}e^x) - {\cL}(\e) - C_n{\cL}(e^{\lambda_2 x}) = r\zeta\bar{L}e^x + \b\e > 0. \]
    Furthermore, $f_\e\le f+\e$ implies $\pat_\tau Y_{n,\e}(0,x) \le \zeta\bar{L}e^x + \e \le W_n(x)$.
    
    Thus, by the comparison principle \cite[Theorem 17, p.53]{FriedmanParabolicPDE}, we deduce that $\partial_\tau Y_{n,\varepsilon}\le W_n(x)$ in $\Omega^n_T$. 
    Taking the limit $\e\to 0^+$ yields $\pat_\tau Y_n \le \zeta\bar{L}e^x + C_n e^{\lambda_2 x}$.
    Since $\lambda_2 < 0$, we have $\lambda_2 - 1 < -1 < 0$. Thus, $C_n \to 0$ as $n \to \infty$. For any fixed $x \in \mathbb{R}$, taking $n\to \infty$ yields $\pat_\tau Y\le \zeta\bar{L}e^x$. Therefore, $e^{-x}\pat_\tau Y=\pat_\tau Y_0\le \z\bar{L}$.

\subsection{{Proof of Lemma~\ref{X^R parametrization}}}\label{app:proof:X-R-parametrization}
Suppose to the contrary that there exists some $\tau'\in(0,T)$ and $x_1,x_2\in\mathbb{R}$ 
with $x_1<x_2$ satisfying
\[Y_0(\tau',x_1)>\psi_0(x_1)
\quad \text{and} \quad
Y_0(\tau',x_2)=\psi_0(x_2).\]
Existence of $x_2\in \mathbb{R}$ is indeed guaranteed by Lemma~\ref{contact_1}.
Then, by continuity, there exists an open interval $I\subset \mathbb{R}$ containing $x_1$, such that $Y_0(\tau',x)> \psi_0(x)$ holds for all $x\in I$.
This allows us to define an interval containing $x_1$, where the strict inequality is preserved. Specifically, we define its endpoints by
\[x_l := \sup \{ x < x_1 : Y_0(\tau', x) = \psi_0(x) \}
\quad \text{and} \quad
x_u := \inf \{ x > x_1 : Y_0(\tau', x) = \psi_0(x) \}.\]
Again, Lemma~\ref{contact_1} guarantees that $x_l>-\infty$.
We then have $Y_0=\psi_0$ and $\pat_xY_0=\pat_x\psi_0$
at both points $(\tau',x_l)$ and $(\tau',x_u)$.
Moreover, by the monotonicity of $Y_0$ with respect to $\tau$, we see that $Y_0(\tau,x_l)=\psi_0(x_l)$ and $Y_0(\tau,x_u)=\psi_0(x_u)$ hold for all $0\le \tau\le \tau'$.
Consider the cylinder $Q:=(0,\tau')\times(x_l,x_u)$.
Again by continuity, we see that the set $\{Y_0>\psi_0\}\cap Q$ is nonempty and furthermore $Y_0=\psi_0$ and $\pat_x Y_0=\psi'_0$ on the parabolic boundary of $\{Y_0>\psi_0\}\cap Q$.
By differentiating the equation $\pat_\tau (Y_0-\psi_0)-\widetilde{\cL}(Y_0-\psi_0)=f_0+\widetilde{\cL}\psi_0$, which holds in $\{Y_0>\psi_0\}$,
we deduce that
\[\pat_\tau (\pat_xY_0-\psi'_0)-\widetilde{\cL}(\pat_xY_0-\psi'_0)=f'_0+\widetilde{\cL}\psi'_0=\cF'_0>0
\quad \text{in }\ \{Y_0>\psi_0\}\cap Q.\]
Thus, by the strong maximum principle, it follows that
$\pat_x Y_0 -\psi'_0>0$ for all points that do not lie on the parabolic boundary of $\{Y_0>\psi_0\}\cap Q$.
Consequently, $\pat_x Y_0(\tau',x)-\psi'_0(\tau',x)>0$ for all $x\in (x_l,x_u)$.
However, since $Y_0(\tau',x)-\psi_0(\tau',x)=0$ for $x=x_l$ and $x=x_u$, applying the mean value theorem yields a point $x_m\in(x_l,x_u)$ satisfying 
$\pat_x Y_0(\tau',x_m)-\psi'_0(x_m)=0$, which is a contradiction.

From \eqref{mnt1} combined with the result in Lemma~\ref{contact_1}, the parametrization $X^R(\tau)$ is well defined for all $\tau\in(0,T)$.
Moreover, with $\pat_\tau Y_0 \ge 0$ a.e. in $\O_T$, we can easily deduce the monotone decreasing property of $\cX^R$.

Let us prove the strict inequality $\partial_x Y_0 > \psi'_0$ at each point in $\{Y_0 > \psi_0\}$ that is sufficiently close to $\cX^R$.
Fix some $\tau_0\in (0,T)$.
From the previous result that $\pat_\tau Y_0\ge 0$ a.e in $\O_T$, we derive the continuity of $\pat_\tau Y_0$ at each free boundary points by \cite[Theorem 1.2]{Blanchet-et-al-2006}.
Therefore, by taking $x\to \cX^R(\tau_0)$ (from the right) on the equation $\pat_\tau (Y_0(\tau_0,x)-\psi_0(x))-\widetilde{\cL}(Y_0(\tau_0,x)-\psi_0(x))=f_0(x)+\widetilde{\cL} \psi_0(x)=\cF_0(x)$,
we deduce that
\[\displaystyle\lim_{x\to \cX^R(\tau_0)^+}-\frac{\t^2}{2}(\pat_{xx}Y_0(\tau_0,x)-\psi''_0(x))=\cF_0(\cX^R(\tau_0)).\]
Note from the result in Lemma~\ref{contact_1} that the free boundary point $(\tau_0,\cX^R(\tau_0))\in \pat\{Y_0>\psi_0\}$ is contained in the set $\{(\tau,x)\in\O_T:x\le x^*\}$.
In particular, $\cX^R(\tau_0)<x^*$ since otherwise, the monotonicity of $Y_0$ with respect to $\tau$ implies that for $R>x^*$, the cylinder $(0,\tau_0)\times(x^*,R)$ is contained in the region $\{Y_0>\psi_0\}$ with $Y_0=\psi_0$ and $\pat_x Y_0=\psi'_0$ on its left lateral boundary. 
Noting that $\pat_\tau Y_0$ satisfies $\pat_\tau (\pat_\tau Y_0)-\widetilde{\cL}(\pat_\tau Y_0)=0$ in the cylinder and that $\pat_x(\pat_\tau Y_0)=0$ on the left lateral boundary, we deduce by the parabolic Hopf lemma that there exists a point lying in the cylinder, at which $\pat_\tau Y_0=0$ holds.
Thus, by the strong maximum principle, $\pat_\tau Y_0\equiv 0$ in $(0,\tau_0)\times(x^*,R)$ and this
combined with the initial condition $Y_0(0,x)=\psi_0(x)$ and the continuity of $Y_0$ up to the initial boundary yields $Y_0\equiv \psi_0$ in $(0,\tau_0)\times(x^*,R)$.
This contradicts $(0,\tau_0)\times(x^*,R)\subset \{Y_0>\psi_0\}$.
Consequently, by setting $c_0:=-\cF_0(\cX^R(\tau_0))>-\cF_0(x^*)=0$, we conclude
\[\displaystyle\lim_{x\to \cX^R(\tau_0)^+}(\pat_{xx}Y_0(\tau_0,x)-\psi''_0(x))=\frac{2c_0}{\t^2}>0.\]
Choose $\d_0>0$ satisfying
\[\pat_{xx}Y_0(\tau_0,x)-\psi''_0(x)>\frac{c_0}{\t^2}
\quad \text{for all }\ \cX^R(\tau_0)<x<\cX^R(\tau_0)+\d_0.\]
Then, since $\pat_x Y_0(\tau_0,\cX^R(\tau_0))=\psi'_0(\tau_0)$,
integrating both sides with respect to the spatial variable yields
\[\int_{\cX^R(\tau_0)}^{x} \left(\pat_{xx}Y_0(\tau_0,y)-\psi''_0(y)\right)\ dy
=\Big[\pat_{x}Y_0(\tau_0,x)-\psi'_0(x)\Big]>0\]
for all $\cX^R(\tau_0)<x\le \cX^R(\tau_0)+\d_0$.
This yields our desired result.

\subsection{{Proof of Lemma~\ref{ineq 2}}}\label{app:proof:ineq-2}
In order to ensure that our local analysis near the free boundary $\cX^R$ strictly avoids the contact region of the gradient constraint, we first observe the continuity of $\pat_x Y_0 + Y_0$ in $\O_T$, obtained by the Sobolev embedding theorem. 
    Noting that $\pat_x Y_0 + Y_0 = \psi'_0 + \psi_0 = e^{-x}\psi' < 0$ on the bounded free boundary $\cX^R$, uniform continuity guarantees the existence of constants $c>0$ and $\d_0>0$ such that $\pat_x Y_0(\tau,x) + Y_0(\tau,x) < -c$ for all $\tau \in (0,T)$ and $\cX^R(\tau) < x < \cX^R(\tau) + \d_0$.
    
    Fix some $\k>0$ and $\tau_0\in [2\k,T]$.
    Since $\pat_\tau Y_0\ge 0$, the result in \cite{Blanchet-et-al-2006} implies the continuity of $\pat_\tau Y_0$ at the free boundary point $(\tau_0,\cX^R(\tau_0))$.
    Moreover, since $\pat_\tau Y_0(\tau_0,x)-\widetilde{\cL} Y_0(\tau_0,x) =f_0(x)$ for all $\cX^R(\tau_0)<x<\cX^R(\tau_0)+\d_0$,
    taking $x\to \cX^R(\tau_0)$ at the equation 
    $(\partial_\tau-\widetilde{\cL})(Y_0-\psi_0)=f_0+\widetilde{\cL} \psi_0=\cF_0$ implies
    \[\displaystyle\lim_{x\to \cX^R(\tau_0)}
    -\frac{\t^2}{2}\left(\pat_{xx}Y_0(\tau_0,x)-\psi_0''(x)\right)=\cF_0(\cX^R(\tau_0)).\]
    Since the function $\cF_0$ is strictly increasing with $\cF_0(x^*)=0$ and
    $\cX^R(\tau)<x^*$ for all $\tau\in(0,T]$,
    it follows that $\cF_0(\cX^R(\tau_0))$ is strictly negative.
    We define a positive constant $c_\k$ by
    $c_\k=-\frac{2}{\t^2}\cF_0(\cX^R(\k))$.
    Since $0<c_\k\le -\frac{2}{\t^2}\cF_0(\cX^R(\tau))$ for all $\tau\in[\k,T]$, we can choose a uniform constant $\d\in (0,\d_0)$ such that
    \[\pat_{xx}Y_0(\tau,x)-\psi_0''(x)>\frac{1}{2}c_\k
    \quad \text{for }\ \cX^R(\tau)<x<\cX^R(\tau)+\d\]
    holds for all $\tau\in[\k,T]$.
    Then, integrating both sides yields
    \begin{equation}\label{dx Y_0}
        \pat_x Y_0(\tau,x)-\psi_0'(x)=\int_{\cX^R(\tau)}^x \pat_{yy}Y_0(\tau,y)-\psi_0''(y)\ dy>\frac{1}{2}c_\k(x-\cX^R(\tau))  
    \end{equation}
    for all $\cX^R(\tau)<x<\cX^R(\tau)+\d$ and $\tau\in[\k,T]$.
    Noting that $\cX^R$ is monotone decreasing and uniformly continuous,
    we may choose $\eta>0$ less than $\k$ such that 
    $0\le \cX^R(\tau-\eta)-\cX^R(\tau)<\d$ and define a region by $S_\delta:=\{(\tau,x)\in \O_T:\tau_0-\eta<\tau<\tau_0,\ \cX^R(\tau)<x<\cX^R(\tau_0)+\d\}$.
    Then, by the choice of $\eta$ and $\d$, it is clear that $S_\d$ is contained in the non-contact set $\{Y_0>\psi_0\}\cap\{\pat_x Y_0+Y_0<0\}$.
    Define a function $V$ by
    \[V(\tau,x):=(\tau-\tau_0+\eta)\ \pat_\tau Y_0(\tau,x)
    -M\left[1-e^{-k(x-\cX^R(\tau_0))}\right],\]
    where $k>0$ and $M>0$ are sufficiently large constants.
    Then, since $\pat_\tau Y_0(\tau,\cX^R(\tau))=0$ for all $\tau$, which follows from the continuity of $\pat_\tau Y_0$, we see that $V\le 0$ on the left lateral boundary and the initial boundary of $S_\d$.
    Moreover, recall that $\pat_\tau Y_0\le \z\bar{L}$ in $\O_T$.
    By choosing $k>0$ and $M>0$ large enough to satisfy $\eta\z\bar{L}<M(1-e^{-k\d})$, we see that $V\le 0$ on the right lateral boundary of $S_\d$.
    Choose $k>0$ large enough to satisfy $\frac{\t^2}{2}k^2+(\b-r+\frac{\t^2}{2})k-r >0$.
    Then, since $\pat_\tau (\pat_\tau Y_0)-\widetilde{\cL}(\pat_\tau Y_0)=0$, for sufficiently large $M>0$, $V$ satisfies in $S_\d$ that
    \begin{align*}
        \pat_\tau V-\widetilde{\cL} V
        &=\pat_\tau Y_0
        -M\left[e^{-k(x-\cX^R(\tau_0))}\left(\tfrac{\t^2}{2}k^2-(\b-r+\tfrac{\t^2}{2})k-r\right)+r\right]
        \\
        &\le \z\bar{L}-M\left[e^{-k\d}\left(\tfrac{\t^2}{2}k^2-(\b-r+\tfrac{\t^2}{2})k-r\right)+r\right]
        \\
        &\le 0.
    \end{align*}
    Therefore, by applying the comparison principle, we deduce that $V\le 0$ in $\overline{S_\d}$.
    This further implies that $V(\tau_0,x)= \eta\, \pat_\tau Y_0(\tau_0,x)-M\left[1-e^{-k(x-\cX^R(\tau_0))}\right]\le 0$ and thus,
    \[\eta\,\pat_\tau Y_0(\tau_0,x)
    \le M\left[1-e^{-k(x-\cX^R(\tau_0))}\right]
    \le k M(x-\cX^R(\tau_0))\]
    for all $\cX^R(\tau_0)<x<\cX^R(\tau_0)+\d$.
    Combined with \eqref{dx Y_0}, we deduce that
    \[\pat_\tau Y_0(\tau_0,x)\le \frac{2k M}{c_\k\eta}\left(\pat_x Y_0(\tau_0,x)-\psi_0'(x)\right)
    \quad \text{for all }\ \cX^R(\tau_0)<x<\cX^R(\tau_0)+\d.\]
    Since $\tau_0\in [2\k,T]$ is arbitrary, our desired result follows with $C=\frac{2k M}{c_\k\eta}$ depending on $\k$ and $\d$.

\subsection{{Proof of Theorem~\ref{X^R:properties}}}\label{app:proof:X-R-properties}
To establish the local Lipschitz continuity of the free boundary $\cX^R(\tau)$,
    we shall construct a sequence of level curves that converges to $\cX^R$.
    Let us fix some small $\k>0$
and choose $N\in\mathbb{N}$ sufficiently large so that
\begin{equation}\label{X^R, N}
    \frac{1}{N}\le 
    \min_{\tau\in [\k,T]}\{Y_0(\tau,\cX^R(\tau)+\d)-\psi_0(\cX^R(\tau)+\d)\}.
\end{equation}
Here, $\d>0$ is a small constant from the previous lemma, introduced in order to ensure $\pat_x Y_0-\psi'_0>0$ near $(\tau,\cX^R(\tau))$ so that we can establish a sequence of level curves approximating to $\cX^R(\tau)$ without intersecting the set $\{\pat_x Y_0+Y_0=0\}$.
Utilizing $\pat_x Y_0(\tau,x)-\psi'_0(x)>0$ for all $x \in (\cX^R(\tau), \cX^R(\tau)+\d]$, the implicit function theorem uniquely defines a continuously differentiable curve $\cX^R_n$ on $(\k,T)$ for each $n\ge N$, which satisfies
\begin{equation}\label{X^R, n}
Y_0(\tau,\cX^R_n(\tau))-\psi_0(\cX^R_n(\tau))=\frac{1}{n}.
\end{equation}
Moreover, since $Y_0-\psi_0$ is strictly increasing in $x$, the sequence $\{\cX^R_n(\tau)\}_{n=N}^\infty$ is strictly decreasing in $n$.
Being bounded below by $\cX^R(\tau)$, this implies that for all $\tau\in[\k,T]$, $\cX^R_n(\tau)$ converges to $\cX^R(\tau)$ as $n\to\infty$.
Differentiating \eqref{X^R, n} with respect to $\tau$ yields
\begin{equation}\label{eq:fb_derivative}
\partial_\tau Y_0(\tau,\cX^R_n(\tau))
=
\big\{\psi_0'(\cX^R_n(\tau)) - \partial_x Y_0(\tau,\cX^R_n(\tau))\big\}
\left[\frac{d}{d\tau}\cX^R_n(\tau)\right].
\end{equation}
The inequality \eqref{X^R, N} combined with the strict positivity of $\pat_x Y_0-\psi_0'$ near the free boundary point $(\tau,\cX^R(\tau))$ (see \eqref{mnt3})
implies $\cX^R(\tau)< \cX_n^R(\tau)\le \cX^R(\tau)+\d$ for each $n\ge N$.
Therefore, applying the previous estimate \eqref{Y0:estimate} for all $n\ge N$ yields that
\begin{equation}\label{eq:fb_derivative 2}
    \left\vert \frac{d}{d\tau}{\cX_n^R}(\tau)\right\vert
    =
    \left|\frac{\partial_\tau Y_0(\tau,\cX^R_n(\tau))}{
    \psi_0'(\cX^R_n(\tau))-\partial_x Y_0(\tau,\cX^R_n(\tau))
    }\right|
    \le C
    \quad \text{for all }\ \tau\in [2\k,T],
\end{equation}
for some constant $C=C(\k,\d)>0$.
% This implies that
% the sequence $\{\frac{d}{d\tau}\cX^R_n(\tau)\}_{n=1}^\infty$ is uniformly bounded on $[2\k,T]$.
% Moreover, the strict positivity $\pat_x Y_0(\tau,x)-\psi'_0(x)> 0$ on $(\cX^R(\tau), \cX^R(\tau)+\d]$ and the condition \eqref{X^R, N}, together with Corollary \ref{X^R region}, imply that $\underline{x}\le \cX^R(\tau)\le \cX^R_n(\tau)\le x^*+\d$ for all $n\ge N$ and $\tau\in [0,T]$. This establishes the uniform boundedness of $\{\cX^R_n\}_{n=N}^\infty$ on $[2\k,T]$.
% Therefore, by the Arzel\'a-Ascoli theorem,
% we can extract a subsequence $\{\cX^R_{n_k}\}_{k=1}^\infty$ that converges uniformly on $[2\k,T]$.
% Substituting $\cX^R_{n_k}$ into \eqref{X^R, n} and letting $k\to\infty$, we see that
% $\cX^R_{n_k}(\tau)$ converges to $\cX^R(\tau)$ uniformly on $[2\k,T]$.

To obtain the local Lipschitz continuity of $\cX^R$, fix any closed interval $I\subset (0,T)$ and corresponding $\k>0$ such that $I\subset (2\k,T)$.
Choose $\tau_1,\tau_2\in I$ with $\tau_1<\tau_2$.
Applying the mean value theorem to $\cX^R_n$ and using \eqref{eq:fb_derivative 2}, we obtain
\[\left|
\frac{\cX^R_n(\tau_2)-\cX^R_n(\tau_1)}{\tau_2-\tau_1}
\right|
=
\left\vert\frac{d}{d\tau}\cX^R_n(\tau_3)\right\vert
\le C\]
for some $\tau_3\in(\tau_1,\tau_2)$ and a constant $C>0$ depending only on $I$.
Taking $n\to\infty$ and the convergence $\cX^R_n\to \cX^R$ yields the Lipschitz continuity of $\cX^R$ on $I$.

    (2) 
    % From $\pat_\tau Y_0\ge 0$, it is not hard verify the monotone decreasing property of $\cX^R$.
    % Furthermore, strict monotonicity follows from the standard argument, using the Hopf lemma.
    To prove the monotonicity of $\cX^R$, let us fix any $\tau_1, \tau_2 \in (0,T)$ with $\tau_1 < \tau_2$.
Let $x = \cX^R(\tau_2)$. By definition and the continuity of $Y_0$ and $\psi_0$, we have $Y_0(\tau_2, x) = \psi_0(x)$.
Since $Y_0 $ is increasing with respect to $\tau$ and $Y_0 \ge \psi_0$, it follows that
\[ \psi_0(x) \le Y_0(\tau_1, x) \le Y_0(\tau_2, x) = \psi_0(x). \]
Thus, $Y_0(\tau_1, x) = \psi_0(x)$, which implies $x$ belongs to the contact set at time $\tau_1$.
By the definition of the free boundary $\cX^R(\tau_1) = \sup\{x : Y_0(\tau_1, x) = \psi_0(x)\}$, we immediately deduce that
\[ \cX^R(\tau_1) \ge x = \cX^R(\tau_2). \]
This implies that $\cX^R$ is monotone decreasing.

To prove the strictly decreasing property of $\cX^R$, suppose by contradiction that there exist 
$\tau_1<\tau_2$ such that $\cX^R(\tau_1)=\cX^R(\tau_2)$. 
For $r>0$, define
$Q_r := (\tau_1,\tau_2)\times (\cX^R(\tau_1),\, \cX^R(\tau_1)+r)$,
and for $0<\delta<\tau_2-\tau_1$ define
$Y_0^\delta(\tau,x) := Y_0(\tau+\delta,x) - Y_0(\tau,x)$
on the cylinder
$Q_{r,\delta}:=(\tau_1,\tau_2-\delta)\times (\cX^R(\tau_1),\, \cX^R(\tau_1)+r)$.
Then $Y_0^\delta$ satisfies
$\partial_\tau Y_0^\delta - \cL Y_0^\delta = 0$ 
in $Q_{r,\delta}$.
At the left lateral boundary $\{\tau_1<\tau<\tau_2-\delta,\; x=\cX^R(\tau_1)\}$, we have $Y_0=\psi_0$, 
and hence $Y_0^\delta=0$ and $\partial_x Y_0^\delta=0$ there. 
Suppose that there exists no point in $Q_{r,\d}$ satisfying 
$Y_0^\d=0$.
Then, since $Y_0^\d$ attains its strict local minimum at the left lateral boundary of $Q_{r,\d}$ we deduce by the parabolic Hopf lemma that $\pat_x Y_0^\d>0$ on the same boundary, which is a contradiction.

Let us denote the interior minimum point by
$(\hat{\tau},\hat{x})\in Q_{r,\delta}$, satisfying $Y_0^\delta(\hat{\tau},\hat{x})=0$.
Consider the region
$\widetilde Q_r := (0,\hat{\tau})\times(\cX^R(\tau_1),\, \cX^R(\tau_1)+r)$,
and observe that $\widetilde Q_r\cap\{Y_0>\psi_0\}$ is a nonempty, open and connected set. 
Since $Y_0^\delta$ satisfies $\pat_\tau Y_0^\d -\cL Y_0^\d=0$ and attains its minimum $0$ at an interior point $(\hat{\tau},\hat{x})$ of $\widetilde Q_r\cap\{Y_0>\psi_0\}$, and $Y_0^\delta\ge0$ on the parabolic boundary of this set by the previous result that $\pat_\tau Y_0\ge 0$ in Lemma~\ref{Y_tau, Y_x}, the strong maximum principle yields
$Y_0^\delta \equiv 0$ in $\widetilde Q_r\cap\{Y_0>\psi_0\}$.
By Lemma~\ref{contact_2}, for every $x\in(\cX^R(\tau_1),\,\cX^R(\tau_1)+r)$, the vertical segment 
$(0,T)\times\{x\}$ intersects the set $\{Y_0=\psi_0\}$ at some free boundary point $(\tau',x)$.  
Since $Y_0^\delta\equiv 0$ in $\widetilde Q_r\cap\{Y_0>\psi_0\}$, we deduce that
\[
Y_0(\tau,x) = \psi_0(x)
\quad \text{for all }\ \tau\in(\tau',\hat{\tau})
\quad \text{and }\ x\in(\cX^R(\tau_1),\,\cX^R(\tau_1)+r).
\]
Thus $Y_0=\psi_0$ in $\widetilde Q_r\cap\{Y_0>\psi_0\}$, contradicting the definition of the set $\{Y_0>\psi_0\}$.
Therefore, the assumption $\cX^R(\tau_1)=\cX^R(\tau_2)$ must be false, and the free boundary $\cX^R$ is strictly decreasing.
 
    (3) Since $\cX^R$ is bounded and monotone decreasing, the limit $\displaystyle\lim_{\tau\to 0^+}\cX^R(\tau)$ exists.
    Combining Lemma~\ref{Y>psi} and Lemma~\ref{contact_2}, we then deduce that $\displaystyle\lim_{\tau\to 0^+}\cX^R(\tau)=x^*$.

\subsection{{Proof of Corollary~\ref{X^R:smooth}}}\label{app:proof:X-R-smooth}
% See the proof in \cite[Theorem 5.7]{ha2026finitehorizonoptimalconsumptioninvestment}
    % for details.
    Fix a free boundary point $(\tau,\cX^R(\tau))$ and consider its neighboring cylinder $Q_r:=(\tau-r^2,\tau+r^2)\times(\cX^R(\tau)-r,\cX^R(\tau)+r)$
    and denote by $I_r=(\tau-r,\tau+r)$.
    Choose $r>0$ small so that $Q_r\subset (0,T)\times(-\infty,x^*)$.
    Then, there exists a sequence of level curves 
    $\{\cX^R_n(\tau)\}_{n=1}^\infty$ that satisfies 
    $Y_0(\tau,\cX^R_n(\tau))=\psi_0(\cX^R(\tau))+1/n$ and $N\in\mathbb{N}$ such that $\cX^R_n(\tau)$ lies in $\{Y_0>\psi_0\}\cap Q_r$ for all $n\ge N$.
    Observe that $Y_0$ satisfies $\pat_\tau Y_0 -\cL Y_0 =f_0$ in $\{Y_0>\psi_0\}\cap Q_r$.
    In particular, since $f_0$ is smooth in $Q_r$, we obtain by the Schauder theory that $Y_0$ is also smooth in $\{Y_0>\psi_0\}\cap Q_r$.
    By differentiating the equation $\pat_\tau Y_0 -\cL Y_0 =f_0$ with respect to $\tau$ and $x$, we deduce that
    $\pat_\tau Y_0$ and $\pat_x Y_0-\psi_0'$ satisfies
    \begin{equation*}
        \begin{cases}
            \pat_\tau (\pat_\tau Y_0) -\widetilde{\cL} (\pat_\tau Y_0) =0
            \quad \text{in }\ \{Y_0>\psi_0\}\cap Q_r,
            \\
            \pat_\tau Y_0 =0
            \quad \text{on }\ \pat\{Y_0>\psi_0\},
        \end{cases}
    \end{equation*}
    \begin{equation*}
        \begin{cases}
            \pat_\tau (\pat_x Y_0-\psi_0') -\widetilde{\cL} (\pat_x Y_0-\psi_0') =\cF_0'
            \quad \text{in }\ \{Y_0>\psi_0\}\cap Q_r,
            \\
            \pat_x Y_0-\psi_0' =0
            \quad \text{on }\ \pat\{Y_0>\psi_0\},
        \end{cases}
    \end{equation*}
    respectively, where $\cF_0=f_0+\widetilde{\cL} \psi_0$ and each boundary conditions are derived from the continuity of $\pat_x Y_0$ in the entire domain $\O_T$ induced by the embedding theorem and the continuity result for $\pat_\tau Y_0$ up to the free boundary since $\pat_\tau Y_0 \ge 0$ (see \cite{Blanchet-et-al-2006}).
    Since the region $\{Y_0>\psi_0\}$ is a Lipschitz domain, we apply the boundary Harnack inequality \cite[Theorem 1.2]{TorresLatorre2024} to deduce that for some $\d\in(0,1)$,
    there exists a constant $C>0$ that depends on $\Vert \pat_\tau Y_0 \Vert_{L^\infty(Q_r)}$ and $\Vert \pat_x Y_0-\psi_0' \Vert_{L^\infty(Q_r)}$ such that
    \begin{align*}
        \left\Vert \frac{d}{d\tau}\cX^R_n(\tau) \right\Vert_{C^{\frac{\d}{2}}(\{Y_0>\psi_0\}\cap I_{r^2/4})}
    =\left\Vert \frac{\pat_\tau Y_0(\tau,\cX^R_n(\tau))}{\pat_x Y_0(\tau,\cX^R_n(\tau))-\psi'_0(\cX^R_n(\tau))} \right\Vert_{C^{\frac{\d}{2}}(\{Y_0>\psi_0\}\cap I_{r^2/4})}\le C.
    \end{align*}
    Then, there exists a subsequence $\{\frac{d}{d\tau}\cX^R_{n_j}(\tau)\}_{j=1}^\infty$
    that converges to a function in $C^{\frac{\d}{2}}(I_{r^2/4})$.
    This implies that there exists 
    $\displaystyle\lim_{j\to\infty}\left(\tfrac{d}{d\tau}\cX^R_{n_j}(\tau)\right)$
    which is of class $C^{\frac{\d}{2}}$.
    In particular, this also ensures that $\cX^R$ is differentiable and $\frac{d}{d\tau}\cX^R(\tau)$ coincides with $\displaystyle\lim_{j\to\infty}\{\tfrac{d}{d\tau}\cX^R_{n_j}(\tau)\}$ in $I_{r^2/4}$.
    % Since $\cX^R_{n_k}(\tau)$ converges uniformly to $\cX^R(\tau)$, we deduce that $\cX^R$ is differentiable and in particular,
    % $\frac{d}{d\tau}\cX^R(\tau)=\displaystyle\lim_{k\to\infty}\frac{d}{d\tau}\cX^R_{n_k}(\tau)=\cZ(\tau)$ in each compact subsets of $I_{r^2/4}$.
    Therefore, we deduce that $\cX^R$ is a $C^{1+\frac{\d}{2}}$-curve in $I_{r^2/4}$.
    
    As a result, since $\cX^R$ is a $C^{1+\frac{\d}{2}}$-curve, the region $\{Y_0>\psi_0\}$ turns out to be a $C^{1+\frac{\d}{2}}$-domain (see \cite[Definition 1.1]{Kukuljan2022}). 
    Since $\cX^R$ is away from $x=\tilde{x}$, $\cF_0$ is smooth near the free boundary.
    Thus, by applying the higher order boundary Harnack inequality \cite[Theorem 1.2]{Kukuljan2022} in the $C^{1+\frac{\d}{2}}$-domain $\{Y_0>\psi_0\}$ we deduce that the $C^{1+\frac{\d}{2}}$ norm of $\frac{d}{d\tau}\cX^R_n(\tau)$ is uniformly bounded on each compact subsets of $I_{r^2/16}$. We again extract a subsequence that converges in $C^{1+\frac{\d}{2}}$ to obtain the result that $\cX^R$ is in $C^{2+\frac{\d}{2}}$.
    By applying the higher order boundary Harnack inequality in $C^{k+\frac{\d}{2}}$-domain for $k=2,3,...$ inductively, we conclude that $\cX^R$ is smooth.

\subsection{{Proof of Lemma~\ref{contact:3}}}\label{app:proof:contact-3}
Note that $\phi$ and $g$ are monotone increasing and $\displaystyle\lim_{x\to\infty}g(x)=\zeta\bar{L}>0$.
Therefore, there exists $x_u\in\mathbb{R}$ such that
$g(x)\ge \zeta\bar{L}/2$ for all $x\ge x_u$.
Choose $\e>0$ small so that $\max\{x^*,x_u\} \le \frac{1}{\e^2}$, which yields
\[-\e^2<\phi(\frac{1}{\e^2})\le \phi(x)<0
\quad \text{and} \quad
\frac{\zeta\bar{L}}{2}\le g(\frac{1}{\e^2})\le g(x), 
\quad \text{for \ all }\ x\ge\frac{1}{\e^2}.\]
Define $\Lambda_\e$ by
\[\Lambda_\e(\tau,x)=\e(\tau-\e)\mathbf{1}_{\{\tau<\e\}}-\e^3(x-\frac{2}{\e^2})^2
\mathbf{1}_{\{x<\frac{2}{\e^2}\}},
\quad \text{for \ all }\ (\tau,x)\in[0,T]\times[\frac{1}{\e^2},\infty).\]
It can be checked that $\Lambda_\e\in W^{1,2}_{p,{\rm loc}}((0,T)\times(\frac{1}{\e^2},\infty))$ and hence, direct computation yields
\begin{align*}
    \pat_\tau \Lambda_\e(\tau,x) -\widetilde{\cL} \Lambda_\e(\tau,x)
    &= \e(1+rT)\mathbf{1}_{\{\tau<\e\}}+\e^3\left[ \t^2
    +2(\b-r+\frac{\t^2}{2})(x-\frac{2}{\e^2})
    -r(x-\frac{2}{\e^2})^2 \right]
    \mathbf{1}_{\{x<\frac{2}{\e^2}\}}
    \\
    &\le \e(1+rT)\mathbf{1}_{\{\tau<\e\}}
    +\e^3\left( \t^2
    +\frac{2r}{\e^2} \right)\mathbf{1}_{\{x<\frac{2}{\e^2}\}}
    \\
    &\le \frac{\zeta\bar{L}}{2}\le g(x)
\end{align*}
for all $(\tau,x)\in (\frac{1}{\e^2},\infty)$.
Moreover, 
$\Lambda_\e(\tau,\frac{1}{\e^2})\le -\frac{1}{\e}+\e T \le V_0(\frac{1}{\e^2})$
for all $\tau\in(0,T)$
and
$\Lambda_\e(0,x)\le -\e^2 \le \phi(x)$
for all $x\ge \frac{1}{\e^2}$.
Denote by $r_\e(\tau,x)=\pat_\tau\Lambda_\e(\tau,x)-\widetilde{\cL} \Lambda_\e(\tau,x)$ for all $(\tau,x)\in(0,T)\times(\frac{1}{\e^2},\infty)$. Then, $\Lambda_\e$ satisfies the following variational inequality:
\begin{equation*}
    \begin{cases}
        \pat_\tau \Lambda_\e(\tau,x) -\widetilde{\cL} \Lambda_\e(\tau,x) = r_\e(\tau,x)
        \quad \text{for }\ (\tau,x)\in(0,T)\times(\frac{1}{\e^2},\infty)\ \text{ with }\ \Lambda_\e(\tau,x)<0,
        \\
        \pat_\tau \Lambda_\e(\tau,x) -\widetilde{\cL} \Lambda_\e(\tau,x) \le g(x)
        \quad \text{for }\ (\tau,x)\in(0,T)\times(\frac{1}{\e^2},\infty)\ \text{ with }\ \Lambda_\e(\tau,x)=0,
        \\
        \Lambda_\e(\tau,\frac{1}{\e^2})\le V_0(\tau,\frac{1}{\e^2})
        \quad \text{for \ all }\ \tau\in(0,T),
        \\
        \Lambda_\e(0,x)\le V_0(0,\frac{1}{\e^2})
        \quad \text{for \ all }\ x\ge \frac{1}{\e^2}.
    \end{cases}
\end{equation*}
Since $r_\e\le g$ holds for each small $\e>0$, we apply the comparison principle \cite[Theorem 1.2]{YangYan2008} to see
$\Lambda_\e\le V_0$ in $(0,T)\times[\frac{1}{\e^2},\infty)$.
In particular, since $\Lambda_\e=0$ in $[\e,T]\times[\frac{2}{\e^2},\infty)$ and $V_0\le 0$ in $(0,T)\times[\frac{1}{\e^2},\infty)$, we deduce that 
$V_0=0$ in $[\e,T]\times[\frac{2}{\e^2},\infty)$.
Since $V_0=e^{-x}V$, the same conclusion holds for $V$.

\subsection{{Proof of Lemma~\ref{V,dY}}}\label{app:proof:V-dY}
For all $(\tau,x)\in \overline{\O_T}$, define
    \begin{equation*}
        \tilde{Y}(\tau,x) = \psi(x)\mathbf{1}_{\{x\le \cX^R(\tau)\}} + \left( \psi(\cX^R(\tau)) + \int_{\cX^R(\tau)}^x V(\tau,y)\ dy \right)\mathbf{1}_{\{x>\cX^R(\tau)\}}.
    \end{equation*}
    
    We first establish $Y \ge \tilde{Y}$ in $\cC_T$ by proving $\pat_x Y \ge V$. Since $Y \in C^{\frac{3+\a}{2},3+\a}(\{\pat_x Y<0\}\cap \cC_T)$ by Schauder theory, differentiating the equation yields that $\mathcal{P} := \pat_x Y - V$ satisfies $\pat_\tau\mathcal{P} - \cL\mathcal{P} \ge 0$ a.e. in $\{\partial_x Y<0\}\cap \cC_T$. On the parabolic boundary of this region (including $\cX^R(\tau)$, $\tau=0$, and $\pat\{\pat_x Y<0\}$), it is easy to check that $\mathcal{P} \ge 0$. Applying the comparison principle \cite[Theorem 1.1]{YangYan2008} yields $\mathcal{P} \ge 0$ in $\{\pat_x Y<0\}\cap\cC_T$. Since $\mathcal{P} = -V \ge 0$ is trivial in $\{\pat_x Y=0\}\cap\cC_T$, we have $\pat_x Y \ge V$ throughout $\cC_T$. Integrating this inequality from $\cX^R(\tau)$ to $x$ immediately gives $Y \ge \tilde{Y}$ in $\cC_T$.

    Next, we prove $\tilde{Y}\ge Y$ in the region $\cC' :=\{(\tau, x) \in \cC_T : \cX^R(\tau)<x<\cX^B(\tau)\}$. Applying the operator $\pat_\tau - \cL$ to $\tilde{Y}$ and using integration by parts and Leibniz's rule along with $\pat_\tau V - \cL V = f'$, we obtain that 
    \begin{align*}
        (\pat_\tau - \cL) \tilde{Y}(\tau, x) 
        &= \int_{\cX^R(\tau)}^x \pat_\tau V(\tau,y)\, dy - \tfrac{\t^2}{2} \pat_x V(\tau,x) - \left(\b-r-\tfrac{\t^2}{2}\right) V(\tau,x) + \b \tilde{Y}(\tau,x) 
        \\
        &= \int_{\cX^R(\tau)}^x \left( \pat_\tau V - \cL V \right) dy - \tfrac{\t^2}{2} \pat_x V(\tau, \cX^R(\tau)) - \left(\b-r-\tfrac{\t^2}{2}\right) V(\tau, \cX^R(\tau)) + \b \psi(\cX^R(\tau)) 
        \\
        &= \int_{\cX^R(\tau)}^x f'(y)\, dy - \tfrac{\t^2}{2} \pat_x V(\tau, \cX^R(\tau)) - \left(\b-r-\tfrac{\t^2}{2}\right) \psi'(\cX^R(\tau)) + \b \psi(\cX^R(\tau)) 
        \\
        &= f(x) - f(\cX^R(\tau)) - \tfrac{\t^2}{2} \pat_x V(\tau, \cX^R(\tau)) - \left(\b-r-\tfrac{\t^2}{2}\right) \psi'(\cX^R(\tau)) + \b \psi(\cX^R(\tau)).
    \end{align*}
    Using the continuity of $\pat_\tau Y$, taking the limit of $\pat_\tau Y - \cL Y = f$ as $x \downarrow \cX^R(\tau)$  and substituting it into the equation for $\tilde{Y}$ cancels the lower-order terms, yielding $\pat_\tau \tilde{Y}- \cL \tilde{Y} = f(x) + h(\tau)$ in $\cC'$, where $h(\tau) := \frac{\t^2}{2} (\pat_{xx} Y - \pat_x V)(\tau, \cX^R(\tau))$. Since $\pat_x Y \ge V$ for $x > \cX^R(\tau)$ and $\pat_x Y = V$ at $x = \cX^R(\tau)$, the difference $\pat_x Y - V$ is increasing at the free boundary, implying $h(\tau) \ge 0$. Thus, $Z := \tilde{Y} - Y$ satisfies $\pat_\tau Z - \cL Z \ge h(\tau) \ge 0$ in $\cC'$.

    Now, since the domain $\cC'$ is unbounded, we employ a standard barrier argument (cf. \cite[Lemma 2.3]{YangYan2008}). Given the exponential growth bound $|Z| \le Ce^{d|x|}$ (see Lemma~\ref{Y n,e}), we introduce the barrier $\Phi_{M}(\tau,x) = Ce^{dM(K\tau-1)+2dx}$, where $K>0$ is chosen large enough to ensure $(\partial_\tau -\cL)\Phi_{M} \ge 0$. Let $\widehat{Z} = Z + \Phi_{M}$. For a fixed step $\k<1/K$ and an arbitrary point $(\tau_0,x_0)\in \cC'$ with $0<\tau_0\le\k$, we choose $M>x_0$.
    Suppose that $\widehat{Z}$ attains its strict minimum on the closure of the domain $\cC' \cap \{0 \le \tau \le \k,\ x \le M\}$.
    On the free boundary $\cX^B$, $\partial_x \widehat{Z} = \pat_x \Phi_M > 0$ (since $\pat_x Y \ge V = 0$), prevents the minimum from occurring there due to the continuity of $\partial_x \widehat{Z}$.
    Since it is prescribed that $\widehat{Z}\ge 0$ on the remaining boundaries, the strict negative minimum of $\widehat{Z}$ is attained at the interior of $\cC'$, which is a contradiction (see \cite[Lemma 2.1]{YangYan2008}).
    Therefore, $\widehat{Z}(\tau_0,x_0) \ge 0$ and letting $M \to \infty$ forces $\Phi_{M}$ to vanish, yielding $Z(\tau,x) \ge 0$ for all $(\tau,x)\in \cC'$ with $\tau \in [0, \k]$. Proceeding inductively over intervals of size $\k$ extends $\tilde{Y} \ge Y$ globally in $\cC'$.
    Combining this with $Y \ge \tilde{Y}$, we conclude $\tilde{Y} = Y$, and therefore $\pat_x Y = V$ in $\cC'$. Finally, in the remaining region $\cC_T \setminus \cC'$, $V = 0$ and $V\le \pat_xY\le 0$ forces $V=\pat_x Y = 0$, completing the proof.

\subsection{{Proof of Lemma~\ref{Vd:dtau,dx}}}\label{app:proof:Vd-dtau-dx}
First, we establish the increasing properties of $V_0$ on the parabolic boundary of $\cC_T$. Recall that $\pat_x Y = V$ from the previous lemma, and that $V_0 - \phi = e^{-x}(\pat_x Y - \psi') > 0$ on the right neighborhood of $\cX^R$ (see \eqref{mnt3}). Since $V_0 = \phi$ on $\cX^R$, it directly follows that $\pat_x V_0 \ge \phi' \ge 0$ on the left lateral boundary of $\cC_T$.
    
    Next, differentiating the identity $V(\tau,\cX^R(\tau)) = \psi'(\cX^R(\tau))$ with respect to $\tau$ yields
    \[
        \pat_\tau V = -\Big(\pat_x V - \psi''\Big)\frac{d\cX^R}{d\tau} \quad \text{on }\ \cX^R \ \text{ from the right}.
    \]
    Since $\cX^R$ is strictly decreasing and $\pat_x V - \psi''= \pat_x(\pat_x Y-\psi') \ge 0$, multiplying this by $e^{-x}$ implies that $\pat_\tau V_0 \ge 0$ on $\cX^R$. Furthermore, taking the limit as $\tau \to 0^+$ in the equation $\pat_\tau V_0 - \widetilde{\cL}V_0 = g$ provides the initial condition $\pat_\tau V_0(0,x) = \cF_0(x) + \cF_0'(x) > 0$ for $x \ge \cX^R(0) = x^*$.
    
    With these boundary behaviors, we now consider $V_0(\tau, x+\d)$ for some sufficiently small $\d > 0$. Previous facts show that $V_0$ is increasing in $x$ on the left lateral boundary, and $V_0=\phi$ is increasing on the initial boundary of $\cC_T$. Combined with $g' > 0$, we can apply the comparison principle \cite[Theorem 1.2]{YangYan2008} to deduce that $V_0(\tau, x+\d) \ge V_0(\tau, x)$ for all $(\tau,x)\in \cC_T$, which implies $V_0$ is monotone increasing in $x$.
    Similarly, from that $\pat_\tau V_0 \ge 0$ on both the lateral and initial boundaries of $\cC_T$ as described above, the same comparison argument yields that $V_0$ is monotone increasing in $\tau$.
    Finally, combining these monotonicity properties with the boundary condition $V_0 = \phi$ on the parabolic boundary of $\cC_T$ directly yields the conclusion that $V_0 \ge \phi$ globally in $\cC_T$.

\subsection{{Proof of Lemma~\ref{V_0:dtau,dx}}}\label{app:proof:V-0-dtau-dx}
To obtain the estimate \eqref{ineq(3)}, we differentiate the equation and apply the comparison principle to the derivatives of solution $V^0_{n,\e}$ of \eqref{obsV:pen}. Since $g$ is not differentiable, smoothing $g$ using the function $m_\e$ in \eqref{mollifier} justifies this differentiation and the same result carries over to that of $g$.
    We omit this smoothing procedure for simplicity.
    Moreover, since $V^0_{n,\e}\le \r$ and $V^0_{n,\e}$ inherits the increasing properties of $V_0$ at the boundary, the comparison principle yields $\pat_\tau V^0_{n,\e}\ge 0$ and $\pat_x V^0_{n,\e}\ge 0$ in $\cC_T^n$.
    
    Fix a small $\k\in(0,\k_0)$ (with $\k_0$ to be determined) and an evaluation point $(\tau_0,x_0)\in \{V_0<0\}$ such that $\tau_0\in [2\k,T]$. Choose $n\ge 2x_0$. On $\cC_T^n$, we introduce the auxiliary functions $\Upsilon(\tau;\tau_0)=-(\tau-\tau_0)\mathbf{1}_{\{\tau\le \tau_0\}}$ and $\Phi(\tau,x;x_0)=e^{k\tau}(e^{\frac{2}{\g}\vert x-x_0 \vert}-\frac{2}{\g}\vert x-x_0 \vert-1)$, where $k=2(\frac{2}{\g^2}+\frac{1}{\g})\t^2+\frac{4}{\g}(\b+r)-r+2$. Since $-\b'_{2,\e}\ge 0$, direct computation gives that for the operator $\widetilde{\cL}_\b:=\widetilde{\cL}+\b'_{2,\e}(\cdots)$,
    \begin{equation}\label{aux:ineq}
        \begin{aligned}
            \pat_\tau \Phi - \widetilde{\mathcal{L}}_{\b} \Phi
        \ge 
        \pat_\tau \Phi - \widetilde{\mathcal{L}} \Phi
        & = e^{k \tau} \Big[
        (k + r)\big(  e^{\tfrac{2}{\g}|x-x_0|} - \tfrac{2}{\g}|x-x_0| - 1 \big)
        - \tfrac{2\t^2}{\g^2}
         e^{\tfrac{2}{\g}|x-x_0|} 
        \\
        & \quad 
        - \tfrac{2}{\g}(\b - r + \tfrac{\t^2}{2}) \operatorname{sign}(x-x_0)     
        \big(  e^{\frac{2}{\g}|x-x_0|} - 1 \big)
        \Big]
        \\
        &\ge e^{k\tau}\Big[
        \tfrac{k + r}{2}(  e^{\tfrac{2}{\g}|x-x_0|} - \tfrac{4}{\g}|x-x_0| - 2) 
        \\
        &\quad +\Big\{ \tfrac{k+r}{2}-\tfrac{2}{\g}(\b+r)-(\tfrac{2}{\g^2}+\tfrac{1}{\g})\t^2 \Big\}  e^{\tfrac{2}{\g}|x-x_0|}\Big]
        \\
        &\ge e^{k\tau}\Big[
        -(k+r)\mathbf{1}_{\{|x-x_0|\le \g\}}
        + e^{\tfrac{2}{\g}|x-x_0|}\Big],
        \end{aligned}
    \end{equation}
    and $\pat_\tau \Upsilon - \widetilde{\cL}_\b \Upsilon \ge -1$.
    Observe 
    \begin{equation}\label{f'd}
        g'(x)=\tfrac{L^\gam}{\g_1}e^{-\tfrac{1}{\g_1}x}\mathbf{1}_{\{x\le \tilde{x}\}}
    +\tfrac{\zeta(1-\g)}{\g(\g_1-\g)}\left(\tfrac{\g_1-\g}{\zeta(1-\g_1)}\right)^{\tfrac{\g_1}{\g}}e^{-\tfrac{1}{\g}x}\mathbf{1}_{\{x> \tilde{x}\}}
    \end{equation}
    and that each $x$-coordinate of an element in $\cC_T^n$ is bounded below by a constant 
    $x_m\le \min\limits_{[0,T]}\cX^R(\tau)=\cX^R(T)$.
    Let us temporarily replace the positive coefficients in \eqref{f'd} with  $c_1$ and $c_2$, and denote by
    $g'(x)=c_1e^{-\frac{1}{\g_1}x}\mathbf{1}_{\{x\le \tilde{x}\}}
    +c_2e^{-\frac{1}{\g}x}\mathbf{1}_{\{x> \tilde{x}\}}$.
    Then, by choosing $c_3=c_1e^{\frac{1}{\g}x_m-\frac{1}{\g_1}\tilde{x}}$, it follows for all $x\in (x_m,\tilde{x})$ that
    \[c_1e^{-\frac{1}{\g_1}x}\ge c_1e^{-\frac{1}{\g_1}\tilde{x}}=c_3e^{-\frac{1}{\g}x_m}\ge c_3e^{-\frac{1}{\g}x}.\]
    Thus, by choosing $c=\min\{1,c_2,c_3\}$, we conclude
    $g'(x)\ge ce^{-\frac{1}{\g}x}$ for all $x\ge x_m$.
    
    For sufficiently large $M>\sup_{[\k,T]}\vert\frac{d}{d\tau}\cX^R\vert$, we define $C= e^{1-\frac{1}{\g}x_m}MT$, $C_1=\max\{1,\zeta\bar{L}(k+r)e^{1+2\vert k \vert T}\}C$, and $C_2=cC\ge 4$. With these constants, we define the function $\Psi$ as
    \[\Psi(\tau,x;\tau_0,x_0)=\tau_0\pat_\tau V^0_{n,\e}(\tau,x)-C_1e^{\frac{1}{\g}x_0}\pat_x V^0_{n,\e}(\tau,x) -C_2\zeta\bar{L}\{e^{\vert k\vert T}\Phi(\tau,x;x_0)+\Upsilon(\tau;\tau_0)\}.\]
    Then since $c<1$, $\Psi$ satisfies for all $(\tau,x)\in\cC_T^n$ that
    \begin{align*}
        \pat_\tau \Psi(\tau,x) -\widetilde{\cL}_\b\Psi(\tau,x)
        &\le -C_1e^{\frac{1}{\g}x_0}g'(x)-C_2\zeta\bar{L}
        (\pat_\tau -\widetilde{\cL})\{e^{\vert k\vert T}\Phi(\tau,x;x_0)+\Upsilon(\tau;\tau_0)\}
        \\
        &\le -C(k+r)\zeta\bar{L}e^{1+2\vert k \vert T}e^{\frac{1}{\g}x_0}g'(x)
        \\
        &\quad-C_2\zeta\bar{L}e^{\vert k\vert T+k\tau}\Big[-(k+r)\mathbf{1}_{\{|x-x_0|\le \g\}}
        + e^{\frac{2}{\g_1}|x-x_0|}\Big]
          +C_2\zeta\bar{L}
        \\
        &\le -cC(k+r)\zeta\bar{L}e^{1+2\vert k\vert T+\frac{x_0-x}{\g}}
        +cC(k+r)\zeta\bar{L}e^{\vert k\vert T+k\tau}\mathbf{1}_{\{|x-x_0|\le \g\}}\le 0.
    \end{align*}
    % for all $(\tau,x)\in \cC_T^n$.
    % Differentiating the equation for $V^0_{n,\e}$ with respect to $x$ yields $\pat_\tau(\pat_x V^0_{n,\e}) - \widetilde{\cL}_\b(\pat_x V^0_{n,\e}) = g' > 0$. Combining this with \eqref{aux:ineq} and $C_1\le C_1$, we deduce $\pat_\tau \Psi -\widetilde{\cL}_\b\Psi \le 0$ in $\cC_T^n$.
    We next verify that $\Psi \le 0$ on the parabolic boundary of the truncated domain $\cC_{\k,T}^n = \cC_T^n \cap \{\tau > \k\}$.
    First, on the left free boundary $x=\cX^R(\tau)$, differentiating $V^0_{n,\e}(\tau,\cX^R(\tau))=\phi(\cX^R(\tau))$ with respect to $\tau$ yields $\pat_\tau V^0_{n,\e} = (\phi' - \pat_x V^0_{n,\e})\frac{d}{d\tau}\cX^R$. Since $\phi' > 0$, $\pat_x V^0_{n,\e} \ge 0$, and $\frac{d}{d\tau}\cX^R < 0$, we have
    \[\tau_0\pat_\tau V^0_{n,\e} \le -\tau_0\pat_x V^0_{n,\e} \frac{d}{d\tau}\cX^R \le MTe^{1-\frac{1}{\g}x_m+\frac{1}{\g}x_0}\pat_xV^0_{n,\e} \le C_1e^{\frac{1}{\g}x_0}\pat_x V^0_{n,\e},\]
    which implies $\Psi \le 0$ as $\Phi, \Upsilon \ge 0$. 
    Second, on the right boundary $x=n$, choosing $n$ sufficiently large guarantees $\Psi(\tau,n) \le 0$ because $\tau_0 \pat_\tau V^0_{n,\e} \le 2\zeta\bar{L}\tau_0$ and the term $e^{\frac{2}{\g}|n-x_0|}$ in $\Phi$ dominates. 
    Third, on the initial boundary $\tau=\k$, the regularity result \cite[Theorem 19, p.321]{FriedmanParabolicPDE} implies $\pat_\tau V^0_{n,\e}(0,x) = \widetilde{\cL}\phi+g \le \zeta\bar{L}$. By continuity, we can choose $\k_0>0$ small enough such that $\pat_\tau V^0_{n,\e}(\k,x)\le 2\zeta\bar{L}$ for all $\k\in(0,\k_0)$. Since $\Upsilon(\k;\tau_0) = \tau_0-\k$, $\tau_0\in [2\k,T]$ and by choosing $M>0$ large so that $C_2 \ge 4$, we deduce 
    $$\begin{aligned}
\tau_0\pat_\tau V^0_{n,\e}(\k,x)-C_2\zeta\bar{L}\Upsilon(\k;\tau_0)
&\le \z\bar{L}\big[ 2\tau_0 + C_2(\k-\tau_0) \big] \\
&\le \z\bar{L}\big[ 2(2-C_2)\k + C_2\k \big] = \z\bar{L}(4-C_2)\k \le 0.
\end{aligned}$$
    % \begin{align*}
    %     \tau_0\pat_\tau V^0_{n,\e}(\k,x)-C_2\zeta\bar{L}\Upsilon(\k;\tau_0)
    %     &\le 2\zeta\bar{L}\tau_0+C_2\zeta\bar{L}(\k-\tau_0)
    %     \\ 
    %     &\le \z\bar{L}\Big[(2-C_2)\tau_0+C_2\k\Big]
    %     \\
    %     &\le \z\bar{L}\Big[(2-C_2)2\k+C_2\k\Big]
    %     \le 0
    % \end{align*}
    This yields
    $\tau_0\pat_\tau V^0_{n,\e}(\k,x)-C_2\zeta\bar{L}\Upsilon(\k;\tau_0) \le 0$, and hence $\Psi(\k,x) \le 0$.
    
    By the comparison principle \cite[Theorem 1.1]{YangYan2008}, $\Psi(\tau,x;\tau_0,x_0)\le 0$ in $\cC_{\k,T}^n$. Passing to the limits $\e\to 0^+$ and $n\to\infty$, the weak convergence imply that the inequality for $V_0$ holds almost everywhere. Since $V_0\in C^{1+\frac{\a}{2},2+\a}(\{V_0<0\})$ for all $\a\in(0,1)$ by the Schauder theory, the inequality holds pointwise at $(\tau,x) = (\tau_0,x_0)$, eliminating the auxiliary functions (since $\Phi(\tau_0,x_0;x_0)=0$ and $\Upsilon(\tau_0;\tau_0)=0$).
    This implies
    $\tau_0\pat_\tau V_0(\tau_0,x_0)-C_1e^{\frac{1}{\g}x_0}\pat_x V_0(\tau_0,x_0) \le 0$,
    which completes the proof.

\subsection{{Proof of Theorem~\ref{X^W:properties}}}\label{app:proof:X-W-properties}
(1) 
    % The local Lipschitz property of $\cX^B$ can be established by using the previous estimate in Lemma~\ref{V_0:dtau,dx} and proceeding as in Theorem~\ref{X^R:properties}.
    % In this case, the strict positivity of $\pat_x V_0$ in the set $\{V_0<0\}$ is guaranteed by the strong maximum principle.
    To prove the local Lipschitz continuity of $\cX^B$, let us fix a small constant $\k\in(0,\k_0)$ and
    set $N\in\mathbb{N}$ large so that
    \begin{equation}\label{Vd,N}
        -\frac{1}{N}>\max\limits_{\tau\in[\k,T]} V_0(\tau,\cX^R(\tau))=\max\limits_{\tau\in[\k,T]} \phi(\cX^R(\tau)). 
    \end{equation}
   Observe that $V_0$ satisfies $\pat_\tau V_0-\widetilde{\cL} V_0=g$ in the set $\{V_0<0\}$. To handle the non-differentiability of $g$ at $x = \tilde{x}$, we define $W := \pat_x V_0$. By differentiating the equation with respect to $x$, we obtain$$\pat_\tau W-\widetilde{\cL} W=g'>0 \quad \text{in }\ \{V_0<0\} \setminus \{x = \tilde{x}\}.$$Since $\pat_x V_0 \ge 0$ in $\cC_T$ (see Lemma~\ref{Vd:dtau,dx}), the strong maximum principle yields that $W > 0$ strictly away from $\tilde{x}$. Suppose for contradiction that $W$ attains a minimum of $0$ at an interior point $(\tau, \tilde{x})$. Applying the parabolic Hopf lemma from the adjacent subdomains $\{x < \tilde{x}\}$ and $\{x > \tilde{x}\}$ induces strictly negative and positive limits for the outward normal derivatives, respectively. This implies that the left and right spatial limits of $\pat_{xx} V_0$ at $\tilde{x}$ are strictly negative and positive, respectively. This contradicts the continuity of $\pat_{xx} V_0$ guaranteed by standard parabolic Schauder estimates. Thus, we conclude that $\pat_x V_0 > 0$ everywhere in the set $\{V_0<0\}$.
    Consequently, for each $\tau\in [\k,T]$ and $n\ge N$,
    we deduce by the inverse function theorem that there exists a unique point $\cX^B_n(\tau)$ such that 
    \begin{equation}\label{Vd.lv curve}
        V_0(\tau,\cX^B_n(\tau))=-\frac{1}{n}.
    \end{equation}
    It can be seen from $\pat_x V_0 \ge 0$ in $\cC_T$ that $\cX^B_n(\tau)$ is monotone with respect to $n\in \mathbb{N}$ and thus, $\displaystyle\lim_{n\to \infty}\cX^B_n(\tau)=\cX^B(\tau)$ for each $\tau\in [\k,T]$.
    Moreover, from the strict positivity of $\pat_x V_0$ and the implicit function theorem, we deduce that $\cX^B_n$ forms a continuously differentiable curve on $(\k,T)$.
    We differentiate the above equality \eqref{Vd.lv curve} with respect to $\tau$ to get
    \begin{equation*}
        \pat_\tau V_0(\tau,\cX^B_n(\tau))
        +\pat_x V_0(\tau,\cX^B_n(\tau))
        \left[\frac{d}{d\tau}{\cX^B_n}(\tau)\right]=0
    \quad \text{for all }\ \tau\in (\k,T).    
    \end{equation*}
    Since $\pat_x V_0(\tau,\cX^B_n(\tau))>0$ holds for all $\tau\in [\k,T]$ and $n\ge N$, utilizing the estimate \eqref{ineq(3)} further yields
    \[\left\vert\frac{d}{d\tau}{\cX^B_n}(\tau)\right\vert
    =\left\vert\frac{\pat_\tau V_0(\tau,\cX^B_n(\tau))}{\pat_x V_0(\tau,\cX^B_n(\tau))}\right\vert
    \le C\frac{1}{\tau}\exp\!\left({\frac{1}{\g}\cX^B_n(\tau)}\right)
    \quad \text{for all }\ \tau\in [2\k,T].\]
    Note that $\cX^B_n(\tau)$ is bounded above by $\cX^B(2\k)$.
    Therefore, we have
    \begin{equation}\label{X^W:derivative}
        \left\vert \frac{d}{d\tau}\cX^B_n(\tau)  \right\vert
        \le \frac{C}{2\k}\exp\!\left({\frac{1}{\g}\cX^B(2\k)}\right)
        \quad \text{for all }\ \tau\in [2\k,T].  
    \end{equation}
    % Consequently, these uniform boundedness and the Arzela-Ascoli theorem yield a uniformly convergent subsequence say $\{\cX^B_{n_k}(\tau)\}_{k=1}^\infty$.
    % We substitute $\{\cX^B_n(\tau)\}$ by its uniformly convergent subsequence at the identity \eqref{Vd.lv curve} and take $k\to\infty$ to see that $\cX^B_n$ converges to $\cX^B$ uniformly on $[2\k,T]$ up to a subsequence.
    
    To prove the Lipschitz continuity of $\cX^B$ on $[2\k,T]$, choose $\tau_1,\tau_2\in [2\k,T]$.
    Then it follows from the mean value theorem and the inequality \eqref{X^W:derivative} that
    \[\left\vert\frac{\cX^B_n(\tau_1)-\cX^B_n(\tau_2)}{\tau_1-\tau_2}\right\vert
    \le \frac{C}{2\k}\exp\!\left({\frac{1}{\g}\cX^B(2\k)}\right).\]
    By using the convergence $\cX^B_n\to \cX^B$ as $n\to\infty$, we see that $\cX^B$ is Lipschitz continuous on $[2\k,T]$.
    Moreover, since $\k\in (0,\k_0)$ is arbitrary, we conclude that $\cX^B$ is locally Lipschitz continuous on $(0,T)$.
    
    (2) 
    Monotonicity of $\cX^B$ follows from the estimate $\pat_\tau V_0 \ge 0$ in Lemma~\ref{Vd:dtau,dx}.
    The strict monotonicity of $\cX^B$ can be obtained by following the same procedure as in proving the strict monotonicity of $\cX^R$. 
    
    (3) Suppose to the contrary that $\cX^B(\tau)$ does not diverge to $\infty$ as $\tau\to 0^+$.
    Then, there exists a sequence $\{\tau_n\}_{n=1}^\infty$ and a constant $C>0$ such that
    \[\displaystyle\lim_{n\to\infty}\tau_n= 0^+
    \quad \text{and} \quad
    \cX^B(\tau_n)\le C\quad \text{for all }\ n\in \mathbb{N}.\]
    Since $\cX^B$ is bounded from below, there exists a convergent subsequence $\{\cX^B(\tau_{n_k})\}_{k=1}^\infty$.
    From that $V_0(\tau_{n_k},\cX^B(\tau_{n_k}))=0$ for all $k\in\mathbb{N}$, $\displaystyle\lim_{k\to\infty}\tau_{n_k}=0$ and $V_0\in C(\overline{\cC_T})$, it follows that
    $V_0(0,\displaystyle\lim_{k\to\infty}\cX^B(\tau_{n_k}))=0$.
    This contradicts the initial condition $V_0(0,\cdot)=\phi(\cdot)<0$.

\subsection{{Proof of Theorem~\ref{thm:maximal-strong-solution}}}\label{app:proof:thm-maximal-strong-solution}
Let $Y$ be the function obtained in Lemma~\ref{Y} and
    $\widehat{Y}$ be another solution of \eqref{obs1}.
    Suppose the set $\mathcal{N}:=\{\widehat{Y}>Y\}$ is nonempty.
    and $\widehat{Y}-Y$ attains its strictly positive maximum $M := \max_{{\cN}}(\widehat{Y}-Y) > 0$ at some point $(\tau',x') \in {\cN}$.
    Since $\widehat{Y}-Y = 0$ on the boundary $\cP\cN$, the maximum point $(\tau',x')$ cannot lie on the boundary and thus must lie in $\cN$.
    % Since $\cN = \cN' \cup (\cN\setminus\cN')$, 
    We divide the proof into two main cases based on the value of $\pat_x Y$ and $\pat_x \widehat{Y}$ at the maximum point $(\tau',x')$.
    
    % \medskip
    % \noindent
    % \textbf{Case 1.} $(\tau',x') \in \cN'$.
    % In this case, by Lemma~2.1 in \cite{YangYan2008}, a strict contradiction occurs with the inequality $\pat_\tau (\widehat{Y}-Y)-\mathcal{L}(\widehat{Y}-Y)\le 0$.  
    
    % \medskip
    % \noindent
    % \textbf{Case 2.} $(\tau',x') \in \cN\setminus\cN'$.
    % In this set, we observe that $\pat_x \widehat{Y} \le \pat_x Y \le 0$. We further divide this case into two subcases depending on the value of $\pat_x Y(\tau',x')$.
    
    \medskip
    \noindent
    \textbf{Case 1.} $\pat_x Y(\tau',x')=\pat_x \widehat{Y}(\tau',x') < 0$.
    By the continuity of the spatial derivatives, there exists a sufficiently small cylinder $C \subset \cN$ centered at $(\tau',x')$ such that
    \[
        \widehat{Y}-Y > 0, \quad \pat_x Y < 0, \quad \text{and} \quad \pat_x \widehat{Y} < 0 \quad \text{in } C.
    \]
    In this region, spatial derivatives of both functions lie strictly below the gradient constraint, which implies that
    \[
        \pat_\tau (\widehat{Y}-Y) - \cL(\widehat{Y}-Y) \le 0 \quad \text{almost everywhere in } C.
    \]
    However, since $\widehat{Y}-Y$ attains its strict maximum value $M$ in the interior of $C$, the strong maximum principle implies that $\widehat{Y}-Y \equiv M$ throughout $C$. Substituting this constant $M$ into the operator yields
    \[
        \pat_\tau (\widehat{Y}-Y) - \cL(\widehat{Y}-Y) = \beta M > 0 \quad \text{in } C,
    \]
    which contradicts the previous inequality.
    
    \medskip
    \noindent
    \textbf{Case 2.} $\pat_x Y(\tau',x')= \pat_x \widehat{Y}(\tau',x') = 0$.
    In this case, $\pat_x Y(\tau',x') = 0$ implies that $x' \ge \cX^B(\tau')$. 
    For all $x$ in the interval $[\cX^B(\tau'), x']$, we know that $\pat_x \widehat{Y} \le 0$ and $\pat_x Y = 0$, yielding $\pat_x(\widehat{Y}-Y)(\tau',x) \le 0$. This means $\widehat{Y}-Y$ is decreasing in the $x$-direction. Consequently,
    \[
        (\widehat{Y}-Y)(\tau',x) = M \quad \text{and} \quad \pat_x \widehat{Y}(\tau',x) = 0 \quad \text{for all } x \in [\cX^B(\tau'), x'].
    \]
    Now, let $C \subset \cN$ be a small cylinder centered at $(\tau',\cX^B(\tau'))$, such that $C \subset \{Y>\psi\} \cap \{\widehat{Y}>\psi\}$. In the left sub-region $C \cap \{x < \cX^B(\tau)\}$, the function $\widehat{Y}-Y$ satisfies
    \[
        \pat_\tau (\widehat{Y}-Y) - \cL (\widehat{Y}-Y) \le 0.
    \]
    Notice that $\widehat{Y}-Y$ attains its maximum $M$ at the point $(\tau',\cX^B(\tau'))$, which lies exactly on the right lateral boundary of this sub-region. Because the free boundary $\cX^B$ is monotone decreasing, the interior cylinder condition is satisfied. Therefore, applying the parabolic Hopf lemma \cite[Theorem 4.2]{Caffarelli} yields
    \[
        \pat_x (\widehat{Y}-Y)(\tau',\cX^B(\tau')) > 0,
    \]
    which clearly contradicts the established fact that $\pat_x (\widehat{Y}-Y)(\tau',\cX^B(\tau')) = 0$.

Consequently, the maximum of $\widehat{Y}-Y$ on $\cN$ is less than or equal to zero, which implies that the set $\cN$ is empty, i.e., $Y \ge \widehat{Y}$ in $\O_T$.
This shows that $Y$ is the maximal strong solution among all strong solutions satisfying \eqref{obs1}. 
If the gradient-constrained region of $\widehat{Y}$ also takes the form $\{(\tau,x): x \ge \widehat{\cX}^B(\tau)\}$ for some monotone decreasing free boundary curve $\widehat{\cX}^B$, we can symmetrically deduce that $\widehat{Y} \ge Y$ in $\O_T$ by applying the same argument to the set $\{Y>\widehat{Y}\}$. 
Therefore, $Y = \widehat{Y}$, establishing the uniqueness of the strong solution under the prescribed condition on the free boundary.

\subsection{{Proof of Lemma~\ref{lem:strict-convexity-Q}}}\label{app:proof:lem-strict-convexity-Q}
Using the relation $Q(t,z)=Y(T-t,\ln z)$ with $\tau=T-t$ and $x=\ln z$, direct computation yields
\[\pat_{zz}Q(t,z)=e^{-x}\frac{\pat}{\pat x}\left(e^{-x}\pat_x Y(\tau,x)\right)=e^{-x}\pat_x V_0(\tau,x).\]
By the same argument as in the proof of Theorem~\ref{X^W:properties} (i),
% In the region $\mathcal{C}_T\setminus\{V_0=0\}$, let $W := \partial_x V_0 \ge 0$. For $x \neq \tilde{x}$, differentiating the equation with respect to $x$ yields $(\partial_\tau-\mathcal{L})W = g' > 0$. The strong maximum principle then ensures $W > 0$ strictly away from $\tilde{x}$. If $W$ attains a minimum of $0$ at the interior point $(\tau,\tilde{x})$, applying Hopf's Lemma from the adjacent subdomains ($x < \tilde{x}$ and $x > \tilde{x}$) induces strictly negative and positive limits for $\partial_{xx} V_0$ at $\tilde{x}$, respectively. This contradicts the continuity of $\partial_{xx} V_0$ guaranteed by standard parabolic Schauder estimates. 
we deduce that, $\partial_x V_0 > 0$ everywhere in $\mathcal{C}_T\setminus \{V_0=0\}$.
This implies $\partial_{zz}Q > 0$ in $\cC_T\setminus \{V_0=0\}$.

On the other hand, in $\O_T\setminus \cC_T$, we have $\pat_x Y(\tau,x) =\psi'(x)$, which yields
\[
    \pat_{zz}Q(t,z)=e^{-x}\frac{\pat}{\pat x}\big(e^{-x}\psi'(x)\big)=e^{-x}\phi'>0.
\]
Combining these results, we deduce that $\pat_{zz}Q > 0$ everywhere in $\{\pat_z Q<0\}\setminus \pat\{Q>J_R\}$. Because the parametrization of free boundary $\pat\{Q>J_R\}$ is $z=e^{\cX^R(T-t)}$, which is smooth, it has Lebesgue measure zero. Consequently, $\pat_{zz}Q>0$ a.e. in $\{\pat_z Q <0\}$. This almost everywhere strict positivity of the second derivative guarantees the strict convexity of $Q$ with respect to $z$ in $\{\pat_z Q <0\}$.

\subsection{{Proof of Theorem~\ref{thm:min-max}}}\label{app:proof:thm-min-max}
The proof follows the standard verification argument for singular control with discretionary stopping. Since the present problem has the same structural features as in \citet{JeonKimYang2026}, we only provide a sketch and refer the reader to \citet[Theorem 6]{JeonKimYang2026} for the full technical details.

We divide the argument into three steps.

\medskip
\noindent
\textbf{Step 1. A lower bound for fixed \(D\).}
Fix \(D\in\mathcal{NI}(t,T)\), and let \(\tau_t^D(z)\) be the retirement hitting time defined in \eqref{eq:tauD-def}. Consider the process
$$
\mathcal O_s^{D}
:=
\int_t^s e^{-\beta(u-t)}\tilde u\bigl(\mathcal Z_u^{t,z,D}\bigr)\,du
+
e^{-\beta(s-t)}\mathcal Q\bigl(s,\mathcal Z_s^{t,z,D}\bigr),
\qquad s\in[t,\tau_t^D(z)].
$$
Applying It\^o's formula (see \citet{ChiarollaHaussmann1994}) to \(e^{-\beta(s-t)}\mathcal Q(s,\mathcal Z_s^{t,z,D})\), and using
$$
d\mathcal Z_s^{t,z,D}
=
(\beta-r)\mathcal Z_s^{t,z,D}\,ds
-\theta \mathcal Z_s^{t,z,D}\,dB_s
+
\Xi_s^{t,z}\,dD_s,
$$
we obtain
\begin{align*}
d\mathcal O_s^{D}
={}&
e^{-\beta(s-t)}
\bigl(
\partial_t\mathcal Q+\mathscr L\mathcal Q+\tilde u
\bigr)\bigl(s,\mathcal Z_s^{t,z,D}\bigr)\,ds 
+e^{-\beta(s-t)}
\partial_z\mathcal Q\bigl(s,\mathcal Z_s^{t,z,D}\bigr)\Xi_s^{t,z}\,dD_s^c
\\
&\quad
+e^{-\beta(s-t)}
\Big(
\mathcal Q\bigl(s,\mathcal Z_s^{t,z,D}\bigr)
-
\mathcal Q\bigl(s,\mathcal Z_{s-}^{t,z,D}\bigr)
\Big) 
-e^{-\beta(s-t)}
\theta \mathcal Z_s^{t,z,D}
\partial_z\mathcal Q\bigl(s,\mathcal Z_s^{t,z,D}\bigr)\,dB_s,
\end{align*}
where \(D^c\) denotes the continuous part of \(D\). Since \(\mathcal Q\) satisfies the variational inequality in Problem~\ref{eq:main3}, we have
$$
\partial_t\mathcal Q+\mathscr L\mathcal Q+\tilde u\ge 0
\quad\text{whenever } \partial_z\mathcal Q=0,
$$
and
$$
\partial_t\mathcal Q+\mathscr L\mathcal Q+\tilde u=0
\quad\text{whenever } \partial_z\mathcal Q<0,\ \mathcal Q>J_R.
$$
Moreover, since \(\partial_z\mathcal Q\le 0\) and \(D\) is non-increasing, the control term is nonnegative, while the jump term is also nonnegative by the monotonicity of \(\mathcal Q\) in \(z\). After localization and taking expectations, the local martingale term drops out, yielding
$$
\mathcal Q(t,z)
\le
\mathbb E_t\left[
\int_t^{\tau_t^D(z)} e^{-\beta(s-t)}\tilde u\bigl(\mathcal Z_s^{t,z,D}\bigr)\,ds
+
e^{-\beta(\tau_t^D(z)-t)}
\mathcal Q\bigl(\tau_t^D(z),\mathcal Z_{\tau_t^D(z)-}^{t,z,D}\bigr)
\right].
$$
By the definition of \(\tau_t^D(z)\), the terminal point lies in the retirement region, and hence
$$
\mathcal Q\bigl(\tau_t^D(z),\mathcal Z_{\tau_t^D(z)-}^{t,z,D}\bigr)
=
J_R\bigl(\mathcal Z_{\tau_t^D(z)-}^{t,z,D}\bigr).
$$
Therefore,
$$
\mathcal Q(t,z)\le \mathcal J\bigl(t,z;D,\tau_t^D(z)\bigr).
$$

\medskip
\noindent
\textbf{Step 2. An upper bound for fixed \(\tau\) under the barrier strategy.}
Fix \(\tau\in\mathcal S(t,T)\), and consider the barrier strategy \(D^B(t,z)\). By construction, the associated controlled process \(\mathcal Z^{t,z,D^B}\) is kept below the upper free boundary \(z_B\). Define
$$
\mathcal O_s^{B}
:=
\int_t^s e^{-\beta(u-t)}\tilde u\bigl(\mathcal Z_u^{t,z,D^B}\bigr)\,du
+
e^{-\beta(s-t)}\mathcal Q\bigl(s,\mathcal Z_s^{t,z,D^B}\bigr),
\qquad s\in[t,\tau].
$$
Applying It\^o's formula as above, we obtain the same decomposition. Now, however, the barrier strategy acts only when \(\mathcal Z^{t,z,D^B}\) hits the binding boundary, where \(\partial_z\mathcal Q=0\). Therefore, the control term vanishes. Furthermore, away from the retirement region, the variational inequality implies
$$
\partial_t\mathcal Q+\mathscr L\mathcal Q+\tilde u\le 0
$$
along the controlled path. After localization and taking expectations, we conclude that
$$
\mathcal J\bigl(t,z;D^B(t,z),\tau\bigr)\le \mathcal Q(t,z).
$$
The required integrability and limiting arguments are handled exactly as in \citet{JeonKimYang2026}.

\medskip
\noindent
\textbf{Step 3. Equality for the candidate pair.}
Finally, choose the pair \(\bigl(D^B(t,z),\tau_R(t,z)\bigr)\). In this case,
$$
\mathcal Z_s^{t,z,D^B}< z_B(s)
\quad\text{for } s<\tau_R(t,z),
$$
and the process is stopped when it first hits the retirement boundary. Along this path,
$$
\partial_t\mathcal Q+\mathscr L\mathcal Q+\tilde u=0
\quad\text{in the working region,}
$$
the singular control acts only on the binding boundary where \(\partial_z\mathcal Q=0\), and the terminal condition at \(\tau_R(t,z)\) satisfies
$$
\mathcal Q\bigl(\tau_R(t,z),\mathcal Z_{\tau_R(t,z)-}^{t,z,D^B}\bigr)
=
J_R\bigl(\mathcal Z_{\tau_R(t,z)-}^{t,z,D^B}\bigr).
$$
Hence all inequality terms in Steps 1 and 2 become equalities, and we obtain
$$
\mathcal Q(t,z)
=
\mathcal J\bigl(t,z;D^B(t,z),\tau_R(t,z)\bigr).
$$

Combining Steps 1--3 yields, for every \(D\in\mathcal{NI}(t,T)\) and \(\tau\in\mathcal S(t,T)\),
$$
\mathcal J\bigl(t,z;D^B(t,z),\tau\bigr)
\le
\mathcal J\bigl(t,z;D^B(t,z),\tau_R(t,z)\bigr)
\le
\mathcal J\bigl(t,z;D,\tau_t^D(z)\bigr).
$$
Therefore,
$$
\overline J(t,z)
=
\mathcal J\bigl(t,z;D^B(t,z),\tau_R(t,z)\bigr)
=
\mathcal Q(t,z).
$$
Since \(\underline J(t,z)\le \overline J(t,z)\) always holds, we conclude that the game admits a value and that
$$
J(t,z)=\underline J(t,z)=\overline J(t,z)=\mathcal Q(t,z).
$$
This completes the proof.

\subsection{{Proof of Theorem~\ref{thm:main}}}\label{app:proof:thm-main}
By Theorem~\ref{thm:min-max}, we already know that
$$
J(t,\xi)=\mathcal Q(t,\xi),
\qquad (t,\xi)\in\mathcal D_T.
$$
Moreover, by Theorem~\ref{thm:Q-structure}, for each fixed \(t\in[0,T)\), the map
$$
\xi\longmapsto J(t,\xi)=\mathcal Q(t,\xi)
$$
is strictly convex on \((0,z_B(t))\). Since \(\partial_z\mathcal Q(t,\xi)=0\) for \(\xi\ge z_B(t)\), we have
$$
\lim_{\xi\uparrow z_B(t)}\partial_\xi J(t,\xi)
=
\partial_z\mathcal Q\bigl(t,z_B(t)\bigr)
=
0.
$$
On the other hand, in the retirement region one has \(J(t,\xi)=J_R(\xi)\), and therefore
$$
\lim_{\xi\downarrow 0}\partial_\xi J(t,\xi)
=
\lim_{\xi\downarrow 0}J_R'(\xi)
=
-\infty.
$$
Hence, for any \(w\ge 0\), there exists a unique \(\xi^*=\xi^*(t,w)\in(0,z_B(t)]\) such that
$$
w=-\partial_\xi J(t,\xi^*)=-\partial_z\mathcal Q(t,\xi^*).
$$

Next, define \(D^*\), \(\mathcal Z^*\), and \(\tau^*\) as in the statement of the theorem. By differentiating the dual value representation along the optimal barrier strategy, one obtains, exactly as in \citet[Lemma 14]{JeonKimYang2026}, the budget identity
\begin{align}
-H_t\,\partial_z J(t,\xi)
=
\mathbb E_t\Bigg[
\int_t^{\tau_R(t,\xi)}
H_sD_s^B(t,\xi)
\Big(
\hat c(\mathcal Z_s^{t,\xi,D^B})
+\zeta \hat l(\mathcal Z_s^{t,\xi,D^B})
-\zeta\bar L
\Big)\,ds
\notag\\
\qquad\qquad\qquad\qquad
+
H_{\tau_R(t,\xi)}
D_{\tau_R(t,\xi)-}^B(t,\xi)
\bigl(-J_R'(\mathcal Z_{\tau_R(t,\xi)-}^{t,\xi,D^B})\bigr)
\Bigg].
\label{eq:budget-identity}
\end{align}
Applying \eqref{eq:budget-identity} at \(\xi=\xi^*\), and using the definitions of \(D^*\), \(\mathcal Z^*\), and \(\tau^*\), yields
\begin{align}
w
=
-\partial_zJ(t,\xi^*)
=
\mathbb E_t\Bigg[
\int_t^{\tau^*}
\frac{H_s}{H_t}D_s^*
\bigl(c_s^*+\zeta l_s^*-\zeta\bar L\bigr)\,ds
+
\frac{H_{\tau^*}}{H_t}
D_{\tau^*-}^*
W_{\tau^*}^*
\Bigg].
\label{eq:budget-equality}
\end{align}
Since
$$
\mathcal Z_s^*
=
\Xi_s^{t,\xi^*}D_s^*
=
\xi^* e^{\beta(s-t)}\frac{H_s}{H_t}D_s^*,
$$
we have
$$
\xi^*\frac{H_s}{H_t}D_s^*
=
e^{-\beta(s-t)}\mathcal Z_s^*,
\qquad
\xi^*\frac{H_{\tau^*}}{H_t}D_{\tau^*-}^*
=
e^{-\beta(\tau^*-t)}\mathcal Z_{\tau^*-}^*.
$$
Multiplying \eqref{eq:budget-equality} by \(\xi^*\), we obtain
\begin{align}
\xi^*w
=
\mathbb E_t\Bigg[
\int_t^{\tau^*}
e^{-\beta(s-t)}
\mathcal Z_s^*
\bigl(c_s^*+\zeta l_s^*-\zeta\bar L\bigr)\,ds
+
e^{-\beta(\tau^*-t)}
\mathcal Z_{\tau^*-}^*
W_{\tau^*}^*
\Bigg].
\label{eq:budget-equality-2}
\end{align}

Now, by the definition of \(\tilde u\) and the optimality of \((\hat c,\hat l)\) in the static maximization problem,
$$
\tilde u(z)
=
u(\hat c(z),\hat l(z))
-
z\bigl(\hat c(z)+\zeta\hat l(z)-\zeta\bar L\bigr).
$$
Hence,
$$
\tilde u(\mathcal Z_s^*)
=
u(c_s^*,l_s^*)
-
\mathcal Z_s^*
\bigl(c_s^*+\zeta l_s^*-\zeta\bar L\bigr).
$$
Similarly, by the duality relation between \(V_R\) and \(J_R\),
$$
J_R(z)=V_R(-J_R'(z))+zJ_R'(z),
$$
or equivalently,
$$
J_R(\mathcal Z_{\tau^*-}^*)
=
V_R(W_{\tau^*}^*)
-
\mathcal Z_{\tau^*-}^*W_{\tau^*}^*.
$$

Combining these identities with Theorem~\ref{thm:min-max}, which gives
$$
J(t,\xi^*)=\mathcal Q(t,\xi^*)
=
\mathcal J(t,\xi^*;D^*,\tau^*),
$$
we deduce
\begin{align*}
J(t,\xi^*)+\xi^*w
&=
\mathbb E_t\left[
\int_t^{\tau^*}e^{-\beta(s-t)}u(c_s^*,l_s^*)\,ds
+
e^{-\beta(\tau^*-t)}V_R(W_{\tau^*}^*)
\right].
\end{align*}
Since the continuation value \(V_R(W_{\tau^*}^*)\) is attained by the post-retirement Merton strategy,
$$
c_s^*=K_1W_s^*,
\qquad
l_s^*=\bar L,
\qquad
\pi_s^*=\frac{\theta}{\sigma\gamma_1}W_s^*,
\qquad s\ge \tau^*,
$$
it follows that
\begin{align*}
J(t,\xi^*)+\xi^*w
=
\mathbb E_t\left[
\int_t^{\tau^*}e^{-\beta(s-t)}u(c_s^*,l_s^*)\,ds
+
\int_{\tau^*}^{\infty}e^{-\beta(s-t)}u(c_s^*,\bar L)\,ds
\right].
\end{align*}

Finally, since \(J(t,\xi^*)+\xi^*w\ge \inf_{\xi>0}(J(t,\xi)+\xi w)\), weak duality in \eqref{eq:weak-duality} yields
\begin{align*}
V(t,w)
&\ge
\mathbb E_t\left[
\int_t^{\tau^*}e^{-\beta(s-t)}u(c_s^*,l_s^*)\,ds
+
\int_{\tau^*}^{\infty}e^{-\beta(s-t)}u(c_s^*,\bar L)\,ds
\right]
\\
&=
J(t,\xi^*)+\xi^*w
\ge
\inf_{\xi>0}\bigl(J(t,\xi)+\xi w\bigr)
\ge
V(t,w).
\end{align*}
Hence all inequalities are equalities, and therefore
$$
V(t,w)=\inf_{\xi>0}\bigl(J(t,\xi)+\xi w\bigr).
$$
This proves part (a) and shows that \((c^*,l^*,\pi^*,\tau^*)\) is optimal.

It remains to identify the wealth and portfolio processes before retirement. By the envelope relation,
$$
W_s^*=-\partial_z\mathcal Q(s,\mathcal Z_s^*),
\qquad s\in[t,\tau^*).
$$
Applying the generalized It\^o formula to \(-\partial_z\mathcal Q(s,\mathcal Z_s^*)\) and comparing the resulting dynamics with
$$
dW_s^*
=
\bigl(rW_s^*+(\mu-r)\pi_s^*-c_s^*+\zeta(\bar L-l_s^*)\bigr)\,ds
+\sigma \pi_s^*\,dB_s,
\qquad s\in[t,\tau^*),
$$
we obtain
$$
\pi_s^*=\frac{\theta}{\sigma}\mathcal Z_s^*\,\partial_{zz}\mathcal Q(s,\mathcal Z_s^*),
\qquad s\in[t,\tau^*).
$$
This completes the proof.

\subsection{{Proof of Corollary~\ref{cor:primal-retirement-threshold}}}\label{app:proof:cor-primal-retirement-threshold}
By Theorem~\ref{thm:main},
\[
W_s^*=-\partial_z\mathcal Q(s,\mathcal Z_s^*),
\qquad s\in[t,\tau^*),
\]
and
\[
\tau^*
=
\inf\{s\in[t,T)\mid \mathcal Z_s^*\le z_R(s)\}\wedge T.
\]
Now fix \(s\in[t,T)\). By Theorem~\ref{thm:Q-structure}, for each fixed \(s\), the map
\[
z\longmapsto -\partial_z\mathcal Q(s,z)
\]
is continuous and strictly decreasing on \((0,z_B(s))\). Therefore,
\[
\mathcal Z_s^*\le z_R(s)
\quad\Longleftrightarrow\quad
-\partial_z\mathcal Q(s,\mathcal Z_s^*)\ge -\partial_z\mathcal Q\bigl(s,z_R(s)\bigr).
\]
Using the definitions of \(W_s^*\) and \(\mathcal W_R(s)\), this is equivalent to
\[
\mathcal Z_s^*\le z_R(s)
\quad\Longleftrightarrow\quad
W_s^*\ge \mathcal W_R(s).
\]
Substituting this equivalence into the definition of \(\tau^*\) yields \eqref{eq:tau-star-primal}. The descriptions of \(\mathcal{RR}_{\rm primal}\) and \(\mathcal{WR}_{\rm primal}\) follow immediately from the monotone correspondence between \(z\in(0,z_B(t))\) and \(w\in(0,\mathcal W_R(t))\).

\section{Function spaces}\label{sec:appendix-function-spaces}

For the reader's convenience, we collect here the definitions of the function spaces used in the main text.

\begin{dfn}[Function spaces on a parabolic domain]\label{dfn:function-spaces}
    Let \(I\subset \mathbb{R}\) be an interval and let \(T>0\). Define the parabolic cylinder by
    \[
    I_T:=[0,T)\times I,
    \qquad (t,z)\in I_T.
    \]
    \begin{enumerate}
    \item For \(m,n \in \mathbb{N}\), the space \(C^{m,n}(I_T)\) is defined by
    \[
    C^{m,n}(I_T)
    :=
    \left\{
    Q: I_T \to \mathbb{R}
    \;\middle|\;
    \partial_t^k \partial_z^\ell Q \in C(I_T)
    \text{ for all } 1 \le k \le m,\ 1 \le \ell \le n
    \right\},
    \]
    consisting of functions that are \(m\)-times continuously differentiable in time and \(n\)-times continuously differentiable in space.

    \item For each \(n \in \mathbb{N}_0\) and \(\alpha \in (0,1)\), the parabolic H\"older space \(C^{\frac{n+\alpha}{2},\, n+\alpha}(I_T)\) consists of functions \(Q \in C^{\lfloor n/2 \rfloor, n}(I_T)\) whose parabolic H\"older norm is finite:
    \[
    \|Q\|_{C^{\frac{n+\alpha}{2},\, n+\alpha}(I_T)}
    :=
    \sum_{0 \le 2j+k \le n} \sup_{I_T} |\partial_t^j \partial_z^k Q|
    + [Q]_{n+\alpha, I_T},
    \]
    where the seminorm \([Q]_{n+\alpha, I_T}\) is defined through the H\"older continuity of the highest-order derivatives with respect to the parabolic metric
    \[
    d\bigl((z_1,t_1),(z_2,t_2)\bigr)
    :=
    \bigl(|z_1-z_2|^2 + |t_1-t_2|\bigr)^{1/2}.
    \]

    \item For $1 \le p \le \infty$ and $n \in \mathbb{N}$, the parabolic Sobolev space $W_p^{n,2n}(I_T)$ is defined by$$W_p^{n,2n}(I_T) := \left\{ Q \in L^p(I_T) : \partial_t^j \partial_z^k Q \in L^p(I_T) \text{ for all } 0 \le 2j+k \le 2n \right\},$$equipped with the norm
    \[\|Q\|_{W_p^{n,2n}(I_T)} := \left( \sum_{0 \le 2j+k \le 2n} \|\partial_t^j \partial_z^k Q\|_{L^p(I_T)}^p \right)^{1/p}.\]
    Similarly, the local parabolic Sobolev space $W^{n,2n}_{p,\mathrm{loc}}(I_T)$ consists of functions $Q$ such that$$Q \in W_p^{n,2n}(K) \qquad \text{for every compact subset } K \subset \mathrm{int}(I_T).$$

    \item For an interval \(I \subset \mathbb{R}\), \(k \in \mathbb{N}_0\), and \(\alpha \in (0,1)\), we denote by \(C^{k+\alpha}(I)\) the H\"older space of all functions \(f \in C^k(I)\) such that
    \[
    \|f\|_{C^{k+\alpha}(I)}
    :=
    \sum_{j=0}^k \sup_{z \in I} |f^{(j)}(z)|
    +
    \sup_{\substack{z_1,z_2\in I\\ z_1\neq z_2}}
    \frac{|f^{(k)}(z_1)-f^{(k)}(z_2)|}{|z_1-z_2|^\alpha}
    <\infty.
    \]
    In particular, for \(k=0\) and \(\alpha=1\), the space \(C^{0,1}(I)\) corresponds to the set of Lipschitz continuous functions on \(I\). 
    Similarly, the local H\"older space \(C^{k+\alpha}_{\mathrm{loc}}(I)\) (and \(C^{0,1}_{\mathrm{loc}}(I)\)) consists of functions \(f\) such that
    \[
    f \in C^{k+\alpha}(K) \quad (\text{respectively, } f \in C^{0,1}(K))
    \qquad
    \text{for every compact subset } K \subset \mathrm{int}(I).
    \]

    \item Finally, the term \textit{smoothness} refers to infinite differentiability (i.e., belonging to the class $C^\infty$).
    \end{enumerate}
\end{dfn}


\begin{thebibliography}{}

\bibitem[Barles, 1999]{BARLES1999191}
Barles, G. (1999).
\newblock Nonlinear neumann boundary conditions for quasilinear degenerate
  elliptic equations and applications.
\newblock {\em J. Differential Equations}, 154(1):191--224.


\bibitem[Blanchet et~al., 2006]{Blanchet-et-al-2006}
Blanchet, A., Dolbeault, J., and Monneau, R. (2006).
\newblock {On the continuity of the time derivative of the solution to the
  parabolic obstacle problem with variable coefficients}.
\newblock {\em J. Math. Pures Appl.}, 85(3):371--414.

\bibitem[Bodie et~al., 1992]{BodieMertonSamuelson1992}
Bodie, Z., Merton, R.~C., and Samuelson, W.~F. (1992).
\newblock Labor supply flexibility and portfolio choice in a life cycle model.
\newblock {\em J. Econom. Dynam. Control}, 16(3--4):427--449.

\bibitem[Bodie et~al., 2004]{BodieDetempleOtrubaWalter2004}
Bodie, Z., Detemple, J.~B., Otruba, S., and Walter, S. (2004).
\newblock Optimal consumption-portfolio choices and retirement planning.
\newblock {\em J. Econom. Dynam. Control}, 28(6):1115--1148.

\bibitem[Bovo and De~Angelis, 2025]{BOVO2025104555}
Bovo, A. and De~Angelis, T. (2025).
\newblock On the saddle point of a zero-sum stopper vs. singular-controller game.
\newblock {\em Stochastic Process. Appl.}, 182:104555.

\bibitem[Bovo and De~Angelis, 2026]{BovoDeAngelis2026}
Bovo, A. and De~Angelis, T. (2026).
\newblock Finite-time horizon, stopper vs. singular-controller games on the half-line.
\newblock {\em Math. Oper. Res.}, Articles in Advance.

\bibitem[Bovo et~al., 2025]{BDI2022}
Bovo, A., De~Angelis, T., and Issoglio, E. (2025).
\newblock Variational inequalities on unbounded domains for zero-sum singular
  controller vs. stopper games.
\newblock {\em Math. Oper. Res.}, 50(1):277--312.

\bibitem[Bovo et~al., 2024]{BDP24}
Bovo, A., De~Angelis, T., and Palczewski, J. (2024).
\newblock Zero-sum stopper vs. singular-controller games with constrained control directions.
\newblock {\em SIAM J. Control Optim.}, 62(4):2203--2228.

\bibitem[Bovo et~al., 2025]{Bovo2024}
Bovo, A., De~Angelis, T., and Palczewski, J. (2025).
\newblock Stopper vs. singular controller games with degenerate diffusions.
\newblock {\em Appl. Math. Optim.}, 91:3.

\bibitem[Caffarelli et~al., 2013]{Caffarelli}
Caffarelli, L., Li, Y., and Nirenberg, L. (2013).
\newblock Some remarks on singular solutions of nonlinear elliptic equations
  iii: Viscosity solutions including parabolic operators.
\newblock {\em Comm. Pure Appl. Math.}, 66.

\bibitem[Chen et al., 2022]{CJW22}
Chen, K., Jeon, J., and Wong, H. (2022).
\newblock Optimal retirement problem under partial information.
\newblock {\em Math. Oper. Res.}, 47(3):1802--1832.

\bibitem[Chiarolla and Haussmann, 1994]{ChiarollaHaussmann1994}
Chiarolla, M.~B. and Haussmann, U.~G. (1994).
\newblock The optimal control of the cheap monotone follower.
\newblock {\em Stochastics Stochastics Rep.}, 49(1--2):99--128.

\bibitem[Choi et~al., 2008]{ChoiShimShin2008}
Choi, K.~J., Shim, G., and Shin, Y.~H. (2008).
\newblock Optimal portfolio, consumption-leisure and retirement choice problem with {CES} utility.
\newblock {\em Math. Finance}, 18(3):445--472.

\bibitem[Choi et~al., 2024]{ChoiKwakLim2024}
Choi, K.~J., Kwak, M., and Lim, B.~H. (2024).
\newblock A dynamic model of sustainable work-life balance.
\newblock Working paper, available at SSRN 4863880.

\bibitem[Crandall et~al., 2000]{Crandall2000}
Crandall, M., Kocan, M., and Świech, A. (2000).
\newblock L-p-theory for fully nonlinear uniformly parabolic equations.
\newblock {\em Comm. Partial Differential Equations}, 25:1997--2053.

\bibitem[Friedman, 1975]{F2}
Friedman, A. (1975).
\newblock {Parabolic variational inequalities in one space dimension and
  smoothness of the free boundary}.
\newblock {\em J. Funct. Anal.}, 18(2):151--176.

\bibitem[Friedman, 1983]{FriedmanParabolicPDE}
Friedman, A. (1983).
\newblock {\em Partial Differential Equations of Parabolic Type}.
\newblock Robert E. Krieger Publishing Company, Malabar, Florida.
\newblock Reprint of the 1964 original edition.

\bibitem[Ha et al., 2026]{ha2026finitehorizonoptimalconsumptioninvestment}
Ha, G. G., Jeon, J., and Ok, J. (2026).
\newblock Finite-Horizon Optimal Consumption and Investment with Time-Varying Job-Switching Costs.
\newblock {\em arXiv preprint arXiv:2603.08050}.

\bibitem[Jeon and Oh, 2023]{JeonOh2023}
Jeon, J. and Oh, J. (2023).
\newblock Labor supply flexibility and portfolio selection with early
  retirement option.
\newblock {\em Appl. Math. Optim.}, 88(3):88.

\bibitem[Jeon et al., 2023]{JKP23}
Jeon, J., Kwak, M., and Park, K. (2023).
\newblock Horizon effect on optimal retirement decision.
\newblock {\em Quant. Finance}, 23(1):123--148.

\bibitem[Jeon et al., 2026]{JeonKimYang2026}
Jeon, J., Kim, T., and Yang, Z. (2026).
\newblock The finite-horizon retirement problem with borrowing constraint: A
  zero-sum stopper vs. singular-controller game.
\newblock {\em Math. Oper. Res.}, forthcoming.

\bibitem[Jeon et al., 2026]{JKY_int}
Jeon, J., Kim, T., and Yang, Z. (2026).
\newblock Integral equation characterization for the finite-horizon retirement problem with borrowing constraint.
\newblock Working paper.

\bibitem[Kukuljan, 2022]{Kukuljan2022}
Kukuljan, T. (2022).
\newblock Higher order parabolic boundary harnack inequality in $c^1$ and
  $c^{k, \alpha}$ domains.
\newblock {\em Discrete Contin. Dyn. Syst.}, 42(6):2667--2698.

\bibitem[Ladyzhenskaya et al., 1968]{Ladyzhenskaya}
Ladyzhenskaya, O.~A., Solonnikov, V.~A., and Ural'tseva, N.~N. (1968).
\newblock {\em Linear and Quasi-linear Equations of Parabolic Type}.
\newblock American Mathematical Society, Providence, RI.

\bibitem[Lieberman, 1996]{Lieberman}
Lieberman, G. (1996).
\newblock {\em Second Order Parabolic Differential Equations}.
\newblock World Scientific, Singapore.

\bibitem[Park and Wong, 2023]{PW23}
Park, K. and Wong, H. (2023).
\newblock Robust retirement with return ambiguity: Optimal G-stopping time in dual space.
\newblock {\em SIAM J. Control Optim.}, 61(3):1009--1037.

\bibitem[Park et~al., 2023]{PWY23}
Park, K., Wong, H., and Yan, T. (2023).
\newblock Robust retirement and life insurance with inflation risk and model ambiguity.
\newblock {\em Insurance Math. Econom.}, 110:1--30.



\bibitem[Torres-Latorre, 2024]{TorresLatorre2024}
Torres-Latorre, C. (2024).
\newblock Parabolic boundary harnack inequalities with right-hand side.
\newblock {\em Arch. Ration. Mech. Anal.}, 248(5):73.

\bibitem[Tso, 1985]{TSO}
Tso, K. (1985).
\newblock {On an Aleksandrov-Bakel’man type maximum principle for
  second-order parabolic equations}.
\newblock {\em Comm. Partial Differential Equations}, 10(5):543--553.

\bibitem[Yang and Koo, 2018]{YK}
Yang, Z. and Koo, H. (2018).
\newblock Optimal consumption and portfolio selection with early retirement option.
\newblock {\em Math. Oper. Res.}, 43(4):1378--1404.

\bibitem[Yang and Yan, 2008]{YangYan2008}
Yang, Z. and Yan, H. (2008).
\newblock A comparison principle of a system of variational inequalities in
  unbounded set.
\newblock {\em J. South China Normal Univ. Nat. Sci. Ed.}, 2008.

\bibitem[Yang et~al., 2021]{YKS21}
Yang, Z., Koo, H., and Shin, Y. (2021).
\newblock Optimal retirement in a general market environment.
\newblock {\em Appl. Math. Optim.}, 21:1083--1130.

\bibitem[Zhu, 1992]{zhu1992dynamic}
Zhu, H. (1992).
\newblock {\em Dynamic Programming and Variational Inequalities in Singular
  Stochastic Control}.
\newblock PhD thesis, Brown University.
\newblock ProQuest Dissertations \& Theses Global.

\end{thebibliography}
\end{document}